\documentclass{article}
\usepackage{amsmath,amsthm,amssymb}
\usepackage{graphicx,cases,color,afterpage,multirow,multicol,mathptmx,helvet,courier,type1cm,derivative,hyperref,placeins}
\usepackage{algorithmic,algorithm}
\allowdisplaybreaks

\newcommand{\dt}{\mathrm{d}t}

\newcommand{\dtau}{\mathrm{d}\tau}

\newcommand{\OT}{{\Omega_T}}
\newcommand{\RT}{{\R^d \times (0,\infty)}}
\newcommand{\OFEM}{{\Omega_{\mathrm{FEM}}}}
\newcommand{\OFDM}{{\Omega_{\mathrm{FDM}}}}
\newcommand{\OIN}{{\Omega_{\mathrm{IN}}}}
\newcommand{\GammaT}{{\Gamma_T}}

\newcommand{\Et}{\tilde{E}}
\newcommand{\epst}{\tilde{\varepsilon}}
\newcommand{\sigt}{\tilde{\sigma}}
\newcommand{\Eb}{\mathbb{E}}

\newcommand{\epsd}{{\delta \varepsilon}}
\newcommand{\sigd}{{\delta\sigma}}

\newcommand{\Etobs}{\tilde{E}_\mathrm{obs}}
\newcommand{\ptt}{\partial_{tt}}
\newcommand{\pt}{\partial_{t}}

\newcommand{\pn}{\partial_{n}}

\newcommand{\R}{\mathbb{R}}
\newcommand{\dobs}{\delta_\mathrm{obs}}
\renewcommand{\div}{\nabla \cdot}

\usepackage[textwidth=1.8cm, textsize=scriptsize]{todonotes}

\begin{document}
\newtheorem{theorem}{Theorem}[section]
\newtheorem{lemma}[theorem]{Lemma}

\numberwithin{equation}{section}

\graphicspath{  
{FIGURES/}{pics}}

\title{Variational optimization approach  for reconstruction of   dielectric permittivity and conductivity functions using partial boundary measurements}


\author{E. Lindström \thanks{Eric Lindström, Email: erilinds@chalmers.se} \and L. Beilina \thanks{Larisa Beilina, Email: larisa.beilina@chalmers.se \\ Department of Mathematical Sciences, Chalmers University of Technology and University of Gothenburg, SE-412 96 Gothenburg Sweden}}

\date{February 10, 2026}

\maketitle


\begin{abstract}

We present a variational optimization approach for
the solution of a coefficient inverse problem of simultaneous
reconstruction of the dielectric permittivity and conductivity
functions in time-dependent Maxwell's system using limited boundary
observations of the electric field.

The variational optimization approach is based on constructing a weak
 form of a Lagrangian which allows to use finite element based
 reconstruction algorithms. The optimality conditions for the
 Lagrangian and stability estimate for the adjoint problem are derived, as well as Frech\'et differentiability of it
 and of the regularized Tikhonov functional are also presented.
Two- and three-dimensional   numerical studies  confirm our theoretical investigations.

\end{abstract}



\section{Introduction}
\label{sec:intro}

In this work we present analysis of variational optimization approach for
reconstruction of spatially distributed  dielectric permittivity and conductivity functions
using partial  time-dependent   measurements of the scattered electric field
at the part of the boundary of investigated  domain.
This problem is a typical coefficient inverse problem (CIP) where both
coefficients of the time-dependent Maxwell's equations, dielectric
permittivity and conductivity functions, should be reconstructed from
boundary measurements.
These coefficients describe a spatially distributed property of the
medium of interest.  Hence,  quantitative and qualitative reconstruction of these
coefficients leads to an image of the interior of the medium and
 to quantitative estimation of these  parameters. 

It is well known that CIPs are ill-posed problems \cite{BKS,Engl,IJT11,tikhonov,TGSK}.
In previous studies \cite{BOOK, BTKM2, TBKF1, TBKF2}, a novel
non-local, approximately globally convergent method for reconstructing
the dielectric permittivity function was developed and validated
through numerical experiments.
In above cited works
the scalar wave equation was taken as an approximate
mathematical model  for the Maxwell's equations.
The two-stage global adaptive optimization method was introduced in \cite{BOOK}
 for reconstructing the dielectric permittivity function in non-conductive media.
The
two-stage numerical procedure of \cite{BOOK} was verified
in several works \cite{BTKM2, TBKF1, TBKF2}
on
experimental data collected by the microwave scattering facility.

The experimental and numerical tests
of above cited works
show that developed methods provide accurate
imaging of all three components of interest in imaging of targets:
shapes, locations and refractive indices of non-conductive media.
 In \cite{convex1}, see also references therein, authors
show reconstruction of complex dielectric permittivity function using
convexification method and frequency-dependent data. 
Potential applications of all above cited works are in the detection and
characterization of improvised explosive devices (IEDs).

In contrast to the previously cited works, this study presents a
theoretical and numerical analysis of a variational optimization
approach for solving the coefficient inverse problem for a 
stabilized Maxwell system in conductive media, based on the
construction of a Lagrangian associated with a regularized Tikhonov
functional.
The optimality conditions for the
Lagrangian, as well as stability estimates for the adjoint problem,
are rigorously derived. In addition, the Fréchet differentiability of
both the adjoint problem and the regularized Tikhonov functional is
established.

 The numerical examples presented in this work demonstrate that
both two- and three-dimensional conjugate gradient (CGA) and adaptive
conjugate gradient (ACGA) algorithms can efficiently and accurately
reconstruct the dielectric permittivity and conductivity functions
using backscattered electric field data. While the two-dimensional
examples are primarily of theoretical interest, the three-dimensional
numerical results focus on applications in microwave medical imaging,
specifically the detection of malignant melanoma, MM.

We note that one of the most important application of algorithms of this paper is
microwave imaging including microwave medical imaging
\cite{ref1,ref2,IEEE2024} and imaging of improvised explosive devices
(IEDs) \cite{BOOK} where qualitative and quantitative determination of
both, the dielectric permittivity and electric conductivity functions
from boundary measurements, is needed.  Potential application of
  theory and algorithms developed in this work are in cancer detection using
microwaves. 
Previous
studies \cite{IEEE2024,ieee1} have highlighted variations in the
relative dielectric permittivity and conductivity contrasts between
malignant and normal tissues, demonstrating that malignant tumors
exhibit higher relative permittivity values compared to normal tissues
at frequencies below 10 GHz.  Several engineering works have tackled
this issue - see some of them in \cite{ref1,ref2,IEEE2024} and
references therein.  See also recent studies
\cite{BL1,ICEAA2024_BEN,ICEAA2024_LB} where authors have reported
first results for the detection of breast cancer and malignant
melanoma (MM) using variational optimization approach analyzed in this
work. 
It is worth noting that
microwave medical imaging, when based solely on boundary measurements
of backscattered electric waves at frequencies around 1 GHz,
represents a non-invasive imaging technique.
Thus, it is very attractive
addition to the existing imaging technologies like X-ray mammography
\cite{noninvasive,Pastorino}, ultrasound \cite{gonch1,gonch2} and MRI
imaging \cite{tomography} .  
 We have performed numerical tests on the reconstruction of  MM using two of the
 real-life  MM models  developed in the recent 
 work \cite{ICEAA2025_LB} with realistic weighted dielectric properties of MM at 6 GHz taken from
 \cite{ICEAA2024_BEN,IEEE2024}.

The main contributions of our paper
can be summarized as follows:

\begin{itemize}

\item Derivation of the weak formulation of the Lagrangian and  optimality conditions of the Lagrangian in the weak form.

\item Derivation of stability estimate for the adjoint problem.

\item  Proof the Fr\'echet differentiability of the regularized Tikhonov's functional in the weak form.

\item Proof of the existence of the minimizers for the non-regularized and regularized  Tikhonov's functionals.

\item  Two- and three-dimensional   numerical tests  which validate developed variational optimization technique
  via CGA and ACGA algorithms.

\end{itemize}

An outline of the work is as follows: in section \ref{sec:model} we
present the mathematical model and  formulate  CIP.
Section \ref{sec:tikh}  introduces Tikhonov's functional and the Lagrangian
 in the weak form and shows proofs for optimality conditions of the Lagrangian in a weak form.
Section \ref{sec:stability} formulates energy estimate for the
forward problem and establishes stability result for the adjoint
problem.
Section \ref{sec:analysis}
 presents mathematical analysis of the optimization problem and Section
 \ref{sec:existence} presents proof for the existence of the minimizer.
Section \ref{sec:algorithms}  presents CGA and ACGA algorithms and
  Section \ref{sec:numex} shows
  two- and three-dimensional
numerical
examples of reconstruction of both dielectric permittivity  and conductivity  functions.
 Finally, section \ref{sec:concl}
discusses obtained results and future research.

\section{The mathematical model}
\label{sec:model}

Let us consider the Cauchy problem for Maxwell's equations in $\R^d, \, d= 2,3,$ in inhomogeneous conductive media:
\begin{subnumcases}{}\label{system1}
    \pt B = - \nabla \times E  & for $ (x, t) \in \RT $, \label{fara1}\\
    \pt D = \nabla \times H - \sigma E - J & for $ (x, t) \in \RT $, \label{ampere1}\\
    \nabla \cdot D = 0 & for $(x, t) \in \RT$, \label{gauss1}\\
    E(x, 0) = f_0, ~ \pt E(x, 0) = f_1 & for $x \in \R^d$, \label{initcond1} 
\end{subnumcases}
where $B$, $E$, $D$, $H$ and $J$ are time-dependent vector fields in $\mathbb{R}^d$ representing the magnetic flux density, the electric field, the electric flux density, the magnetic field and the electric current density respectively, and $\sigma$ is a spatial function representing the conductivity. We also have vector valued $f_0$ and $f_1$ which are arbitrary initial conditions.

This article will handle linear and isotropic materials, which implies 
\begin{equation}
    D = \varepsilon E ~ \text{and} ~ B = \mu H, 
\end{equation}
where the electric permittivity $\varepsilon(x) = \varepsilon_0\varepsilon_r(x)$ and the magnetic permeability $\mu(x) = \mu_0\mu_r(x)$ are spatial scalar functions. Now we can rewrite \eqref{system1} in terms of $E$, $H$ and $J$ as
\begin{subnumcases}{}\label{system2}
    \mu \pt H = - \nabla \times E  & for $ (x, t) \in \RT $, \label{fara2}\\
    \varepsilon \pt E = \nabla \times H - \sigma E - J & for $ (x, t) \in \RT $, \label{ampere2}\\
    \nabla \cdot (\varepsilon E) = 0 & for $(x, t) \in \RT$, \label{gauss2}\\
    E(x, 0) = f_0, ~ \pt E(x, 0) = f_1 & for $x \in \R^d$, \label{initcond2} 
\end{subnumcases}
From here we will take the time-derivative of \eqref{ampere2} to be able to merge it with \eqref{fara2}. This gives us the equation
\begin{equation}\label{firsteq}
    \varepsilon \ptt E = - \nabla \times (\frac{1}{\mu} \nabla \times E) - \sigma \pt E- \pt J.  
\end{equation}
Moving all terms dependent on $E$ of the \eqref{firsteq} to the left side  after merging the first two equations we end up with the system
\begin{subnumcases}{}\label{system3}
    \varepsilon \ptt E + \sigma \pt E + \nabla \times (\frac{1}{\mu} \nabla \times E) = -\pt J  & for $ (x, t) \in \RT $, \label{system3eq1}\\
    \nabla \cdot (\varepsilon E) = 0 & for $(x, t) \in \RT$, \label{gauss3}\\
    E(x, 0) = f_0, ~ \pt E(x, 0) = f_1 & for $x \in \R^d$, \label{initcond3} 
\end{subnumcases}

In this article we will consider solution of CIP in the non-magnetic object, or assume that $\mu_r \equiv 1$, which is roughly the case for human tissues. We will also note the the theoretical connection between $\varepsilon_0$, $\mu_0$ and the speed of light $c$, namely $c = (\varepsilon_0 \mu_0)^{-1/2}$. This allows us to rewrite our system once again into 
\begin{subnumcases}{}\label{system4}
    \frac{\varepsilon_r}{c^2} \ptt E + \mu_0 \sigma \pt E + \nabla \times \nabla \times E = -\mu_0 \pt J & for $ (x, t) \in \RT $, \label{system4eq1}\\
    \nabla \cdot (\varepsilon_r E) = 0 & for $(x, t) \in \RT$, \label{gauss4}\\
    E(x, 0) = f_0, ~ \pt E(x, 0) = f_1 & for $x \in \R$, \label{initcond4} 
\end{subnumcases}
The issue with (\ref{system4}--d) is that $c^{-2}$ and $\mu_0$ are very small constants, the former even smaller than expected accuracy from double precision float points. To contend with this we will substitute $t := ct$, such that we can instead inspect
\begin{subnumcases}{}
    \varepsilon_r \ptt E + c \mu_0 \sigma \pt E + \nabla \times \nabla \times E = -c \mu_0\pt J  & for $ (x, t) \in \RT $, \label{system5eq1}\\
    \nabla \cdot (\varepsilon_r E) = 0 & for $(x, t) \in \RT$, \label{gauss5}\\
    E(x, 0) = f_0, ~ \pt E(x, 0) = f_1 & for $x \in \R$, \label{initcond5} 
\end{subnumcases}

\textit{Remark: } For the convenience, since we are interested in reconstruction of $\varepsilon_r$ and $\sigma$, we will simplify the notation and omit the subscript $r$ from $\varepsilon_r$ and include the constants in $\sigma$. Thus $\varepsilon := \varepsilon_r$ (which should not be confused with the absolute electric permittivity) and $\sigma := c \mu_0 \sigma$ from now. Similarly, we will let $\pt J := c \mu_0 \pt J$. 

For this application we will restrict ourselves to a bounded, convex domain $\Omega \subset \R^d$ with the smooth boundary $\Gamma$.
We will also denote the corresponding space-time domain as $\Omega_T := \Omega \times (0,T)$ with the space-time boundary $\GammaT := \Gamma \times (0,T) $, where $T > 0$ is the final time. Further, we will consider the stabilized model problem by introducing the Coulomb gauge-type term \eqref{gauss5} in \eqref{system5eq1} (see details in \cite{Ass,div_cor}) and thus consider 
\begin{subnumcases}{}\label{forward}
    \varepsilon \ptt E + \sigma \pt E - \Delta E - \nabla \nabla \cdot (\varepsilon - 1) E = -\pt J  & for $ (x, t) \in \OT $, \label{system6eq1}\\
    E(x, 0) = f_0, ~ \pt E(x, 0) = f_1 & for $x \in \Omega$, \label{initcond6} \\
    \pn E = 0 & for $(x, t) \in \GammaT$. \label{boundcond6}
\end{subnumcases}
In recent work \cite{BR2} it was shown that under sufficient conditions on $\varepsilon$ and $\sigma$ the system above approximates the original Maxwell's equations well. We denote the task of finding the field $E$ as the forward problem. 

Before we state the inverse problem we will introduce the domain decomposition (see \cite{BL1}). We will assume $\overline{\Omega} = \OFDM \cup \OFEM$ and $\OFEM \subset \Omega$ as well as $\partial \OFEM \subset \OFDM$. For
the sake of
convenience  we will also introduce $\OIN := \Omega \setminus \OFDM$. With this setup we will add some realistic assumptions on $\sigma$ and $\varepsilon$ in terms of two  parameters sets
\begin{equation}\label{coeffs}
    \begin{aligned}
        &\varepsilon \in C_\varepsilon := \{ \varepsilon \in C^2(\overline{\Omega}) \, : \, \varepsilon(x) = 1 \quad \forall x \in \OFDM, \,  1 \leq \varepsilon(x) \leq d_1 \quad \forall x \in \OIN \}, \\
        &\sigma \in C_\sigma := \{ \sigma \in C^2(\overline{\Omega}) \, : \, \sigma(x) = 1 \quad \forall x \in \OFDM, \, 1 \leq \sigma(x) \leq d_2 \quad \forall x \in \OIN \},
    \end{aligned}
\end{equation}
where $d_1 \geq 1$ and $d_2 \geq 0$ are known real constants. We can now state the inverse problem. 

\textbf{Inverse problem (IP): }\textit{Assume that the functions $\varepsilon\in C_\varepsilon$ and $\sigma \in C_\sigma$. 
For a given observation function  $\Etobs \in [L_2(\GammaT)]^3$ of the electric field
determine $\varepsilon(x)$ and $\sigma(x)$ for $x \in \OIN$ such that  }
\begin{equation}
    E(x,t) = \Etobs \quad \forall (x,t) \in \GammaT.
\end{equation}
Here, $E$ is the electric field which satisfy to the system (\ref{system6eq1}--c) and $\Etobs$ is some known boundary measurements.

\textit{Remark: }Note that $\Etobs $ may only be defined on parts of $\GammaT$, which would correspond to missing data in applications.
Usually, this function contains the noise level $\delta$ which involves ill-posedness of the solution of our IP.
Let $E_{\rm obs }^*$ denote noiseless exact observations at $\GammaT$ and
  \begin{equation}\label{noisyset}
  B_\delta[E_{\rm obs }^*] = \{ E:|| E  - E_{\rm obs }^*  || \leq \delta \}.
  \end{equation}
We assume that data  $ \Etobs  \in B_\delta[E_{\rm obs }^*]$.
In the next section we present the numerical method for  solution of ill-posed inverse problem.


\section{Tikhonov's functional and Lagrangian}

\label{sec:tikh}

Before we move on to defining the corresponding Tikhonov functional and Lagrangian for our problem we will introduce some simplifying notation. We will let the $L^2$ product of two functions $f, g : \R^d \rightarrow \R^d$ over some domain $X$ be denoted by
\begin{equation*}
    (f,g)_X := \int_X f g \, \mathrm{d}X,
\end{equation*}
and the induced norm by $\|\cdot\|_X$. If the $L^2$ in question happens to be a surface integral, we empathize this with brackets, $\langle f, g \rangle_X$. Since $X = \Omega$ and $X = \Gamma$ will occur frequently we will omit this subscript in those cases. If $X$ is a set in both space and time we emphasize the $L^2$ product with double parentheses, $((f, g))_X$ (or double brackets for surface integrals, $\langle \langle f, g \rangle \rangle_X$), and again in the case where $X = \OT$ or $X = \GammaT$ we omit the subscript, again for the sake of brevity. Finally, we will let weighted products and norms with a scalar function $\phi \geq 0, \, \phi \neq 0$ be flagged with a subscript as well, e.g. $(f, g)_{\phi, X} := (\phi f, g)_X$ and $\| \cdot \|_{\phi,X}$.  

Now we will move on to define the regularized Tikhonov functional as 
\begin{equation}\label{tikh}
    F(\varepsilon, \sigma) := F(E(\varepsilon, \sigma), \varepsilon, \sigma) = \frac{1}{2} \| E - \Etobs \|^2_{\GammaT} + \frac{\gamma_{\varepsilon}}{2} \| \varepsilon - \varepsilon^0 \|^2 + \frac{\gamma_{\sigma}}{2} \| \sigma - \sigma^0 \|^2 , 
\end{equation}
where $\gamma_{\varepsilon}, \, \gamma_{\sigma}$ are regularization parameters, $\dobs$ is a $\delta$-function with support in some observation points $x_\mathrm{obs}$, and $\varepsilon^0, \, \sigma^0$ are initial guesses for $\varepsilon$ and $\sigma$, respectively. By minimizing the Tikhonov functional we have essentially turned the coefficient inverse problem into an optimization problem.

However, since $E$ is an implicit function of both $\varepsilon$ and $\sigma$, it is quite difficult to minimize the Tikhonov functional using gradient methods. Therefore we introduce the corresponding Lagrangian, but before we do so we will define some relevant function spaces for $E$ and $\lambda$. Let 
\begin{equation}
    \begin{aligned}
        &H^1_E := \{ E \in [H^1(\OT)]^d \, : \, E (x,0) = f_0(x),\, \pt E (x,0) = f_1(x) \quad \forall x \in \Omega \}, \\
        &H^1_\lambda := \{ \lambda \in [H^1(\OT)]^d \, : \, \lambda (x,T) = 0,\, \pt \lambda (x,T) = 0 \quad \forall x \in \Omega \}, \\
        &U := H^1_E \times H^1_\lambda \times C_\varepsilon \times C_\sigma,  \\
    \end{aligned}
\end{equation}
Next, we will define the Lagrangian, which in strong form looks like 
\begin{equation}\label{Lstrong}
    L(E, \lambda, \varepsilon, \sigma ) = F(E, \varepsilon, \sigma) + ((\varepsilon \ptt E + \sigma \pt E - \Delta E - \nabla\nabla \cdot (\varepsilon -1)E + \pt J , \lambda)). 
\end{equation}
However, to assume lower regularity on our solution we will commonly use the weak formulation of $L(u)$ which we present in the following Lemma.

\begin{lemma}

The weak form of the Lagrangian $L(u)$ given in \eqref{Lstrong}
is
\begin{equation}\label{Lweak}
    \begin{aligned}
    L(u) :&= F(E,\varepsilon,\sigma) - (f_1 , \lambda (x,0))_\varepsilon - (( \pt E, \pt \lambda))_\varepsilon - (f_0, \lambda (x,0))_\sigma \\
    &- (( E, \pt \lambda ))_\sigma 
    + ((\nabla E, \nabla \lambda )) + (( \nabla \cdot (\varepsilon - 1) E, \nabla \cdot \lambda )) + (( \pt J, \lambda ))  
    \end{aligned}
\end{equation}
for $u := (E, \lambda, \varepsilon, \sigma) \in U$.

\end{lemma}

\begin{proof}
    We can prove that the weak form of the Lagrangian simply consists of the Tikhonov functional plus the weak formulation of the forward problem
 obtained via integration by parts of \eqref{Lstrong}.
If we assume $H^2$ regularity on $E$ and use conditions \eqref{coeffs}  for functions $\varepsilon,\sigma$ together with  $\lambda(\cdot,T) = \partial_t \lambda(\cdot,T) = 0$, this is shown as follows:
    \begin{equation}
    \begin{aligned}
    L (E, \lambda, \varepsilon, \sigma) =& F(E,\varepsilon,\sigma) - (f_1 , \lambda (x,0))_\varepsilon - (( \pt E, \pt \lambda))_\varepsilon - (f_0, \lambda (x,0))_\sigma \\
    -& (( E, \pt \lambda ))_\sigma 
    + ((\nabla E, \nabla \lambda )) + (( \nabla \cdot (\varepsilon - 1) E, \nabla \cdot \lambda )) + (( \pt J, \lambda ))  \\
    =& F(E, \varepsilon, \sigma) + (\pt E , \lambda)_{\varepsilon}|_{t = 0}^T - ((\pt E, \pt \lambda ))_\varepsilon + (E, \lambda)_\sigma |_{t = 0}^T \\
    -& ((E, \pt \lambda))_\sigma 
    - \langle\langle \pn E, \lambda \rangle\rangle + ((\nabla E , \nabla \lambda )) - \langle\langle \nabla \cdot (\varepsilon -1 ) E, n\cdot \lambda\rangle\rangle \\
    +& ((\nabla \cdot (\varepsilon -1)E, \nabla \cdot \lambda)) 
    + ((\pt J, \lambda)) \\
    =&F(E, \varepsilon, \sigma) + ((\varepsilon \ptt E + \sigma \pt E - \Delta E - \nabla \nabla \cdot (\varepsilon - 1) E + \pt J, \lambda)),
    \end{aligned}
\end{equation}
which holds for all $\lambda \in H^1_\lambda$ and all $E \in H^1_E \cap H^2(\OT)$ with homogeneous Neumann conditions.
\end{proof}

\textit{Remark: } One can easily see that if $E := E(\varepsilon, \sigma) \in H^1_E$ is a weak solution to the forward problem we have $L(E, \lambda, \varepsilon, \sigma) = F(E, \varepsilon, \sigma)$ for all $(E,\lambda, \varepsilon, \sigma)  \in  U$.


To minimize $L$ we look for stationary points $u = (E,\lambda,\varepsilon,\sigma)  $ such that for all $\bar{u} = (\bar{E}, \bar{\lambda}, \bar{\varepsilon}, \bar{\sigma})$ 
\begin{equation}
    L'(u; \bar{u}) = 0,
\end{equation}
where $L'$ is the Fréchet derivative of $L$. The statement above is equivalent to 
\begin{equation}\label{Loptcond}
    L'(u; \bar{u}) = \left(\pdv{L}{E}(u; \bar{E}),\pdv{L}{\lambda}(u; \bar{\lambda}),\pdv{L}{\varepsilon}(u; \bar{\varepsilon}),\pdv{L}{\sigma}(u; \bar{\sigma})\right) =  0.
\end{equation}

\begin{lemma}

The optimality conditions \eqref{Loptcond} in the weak form are given by
following expressions:
\begin{equation} \label{optcond1}
    \begin{aligned}
        \pdv{L}{E}(u; \bar{E}) &= \langle\langle E - \Etobs, \bar{E}\rangle \rangle_{\GammaT}  -
	(( \pt \bar{E}, \pt \lambda))_\varepsilon 
        - (( \bar{E}, \pt \lambda ))_\sigma 
    + ((\nabla \bar{E}, \nabla \lambda )) \\
    &+ (( \nabla \cdot (\varepsilon - 1) \bar{E}, \nabla \cdot \lambda )) = 0,
 \end{aligned}
\end{equation}
\begin{equation} \label{optcond2}
\begin{aligned}
        \pdv{L}{\lambda}(u; \bar{\lambda}) &= -(f_1, \bar{\lambda}(x,0))_\varepsilon - ((\pt E, \pt \bar{\lambda} ))_\varepsilon - (f_0 , \bar{\lambda} (x,0))_\sigma - ((E, \pt \bar{\lambda} ))_\sigma \\
        &+ ((\nabla E, \nabla \bar{\lambda} )) + ((\nabla \cdot (\varepsilon - 1)E, \nabla \cdot \bar{\lambda})) + ((\pt J , \bar{\lambda})) = 0, \\
\end{aligned}
\end{equation}
\begin{equation} \label{optcond3}
\begin{aligned}
        \pdv{L}{\varepsilon}(u; \bar{\varepsilon}) &= \gamma_{\varepsilon}(\varepsilon-\varepsilon^0,\bar{\varepsilon}) - (\bar{\varepsilon}f_1 , \lambda (x,0)) - ((\bar{\varepsilon} \pt E, \pt \lambda)) +(( \nabla \cdot \bar{\varepsilon} E, \nabla \cdot \lambda )) = 0, \\
\end{aligned}
\end{equation}
\begin{equation} \label{optcond4}
 \begin{aligned}
        \pdv{L}{\sigma}(u; \bar{\sigma}) &= \gamma_{\sigma}(\sigma - \sigma^0, \bar{\sigma}) - (\bar{\sigma}f_0, \lambda (x,0))- ((\bar{\sigma} E, \pt \lambda )) = 0.
    \end{aligned}
\end{equation}

\end{lemma}

\begin{proof}

To find the Frech\'{e}t derivative (\ref{Loptcond}) of the
 Lagrangian (\ref{Lweak}) we consider $L(u + \bar{u}) - L(u)~
 \forall \bar{u} \in U$ and single out the linear part of the
 obtained equality with respect to $ \bar{u}$.  In our derivation of
 the Frech\'{e}t derivative we assume that  all
 functions $u=(E,\lambda, \varepsilon, \sigma) \in U$ 
in the Lagrangian
 (\ref{Lweak})  can be varied independently on each other called all-at-once
  approach.

Let us 
start first  with  derivation of the derivative with respect to $E$. Firstly, we can write
\begin{equation}\label{firstordE}
\begin{aligned}
    L(E+\bar{E}, \lambda, \varepsilon, \sigma) &= F(E+\bar{E},\varepsilon,\sigma) - (f_1 , \lambda (x,0))_\varepsilon - (( \pt (E+\bar{E}), \pt \lambda))_\varepsilon \\
    &- (f_0, \lambda (x,0))_\sigma - (( (E+\bar{E}), \pt \lambda ))_\sigma 
    + ((\nabla (E+\bar{E}), \nabla \lambda )) \\
    &+ (( \nabla \cdot (\varepsilon - 1) (E+\bar{E}), \nabla \cdot \lambda )) + (( \pt J, \lambda )).
\end{aligned}
\end{equation}
Here we note that for the Tikhonov functional we have:
\begin{equation}\label{FreTikh}
\begin{split}
    F(E+\bar{E}, \varepsilon, \sigma) &= \frac{1}{2} \| E - \Etobs + \bar{E}\|^2_{\GammaT} + \frac{\gamma_{\varepsilon}}{2} \| \varepsilon - \varepsilon^0 \|^2 + \frac{\gamma_{\sigma}}{2} \| \sigma - \sigma^0 \|^2 \\
    &= \frac{1}{2} \| E - \Etobs \|^2_{\GammaT} + \langle \langle E - \Etobs, \bar{E}\rangle\rangle_{\GammaT}  + \frac{1}{2} \| \bar{E} \|^2_{\GammaT} + \frac{\gamma_{\varepsilon}}{2} \| \varepsilon - \varepsilon^0 \|^2 \\
    &+ \frac{\gamma_{\sigma}}{2} \| \sigma - \sigma^0 \|^2 
    = F(E,\varepsilon,\sigma) + \langle \langle E - \Etobs, \bar{E} \rangle \rangle_{\GammaT} + \frac{1}{2} \| \bar{E} \|^2_{\GammaT}.
    \end{split}
\end{equation}

Inserting   in \eqref{firstordE} instead of $ F(E+\bar{E}, \varepsilon, \sigma)$  the right hand side of \eqref{FreTikh}
 we get:
\begin{equation}\label{firstordEnew}
\begin{aligned}
    L(E+\bar{E}, \lambda, \varepsilon, \sigma) &=
 F(E,\varepsilon,\sigma) + \langle \langle E - \Etobs, \bar{E} \rangle \rangle_{\GammaT} + \frac{1}{2} \| \bar{E} \|^2_{\GammaT} \\
&- (f_1 , \lambda (x,0))_\varepsilon - (( \pt (E+\bar{E}), \pt \lambda))_\varepsilon \\
    &- (f_0, \lambda (x,0))_\sigma - (( (E+\bar{E}), \pt \lambda ))_\sigma 
    + ((\nabla (E+\bar{E}), \nabla \lambda )) \\
    &+ (( \nabla \cdot (\varepsilon - 1) (E+\bar{E}), \nabla \cdot \lambda )) + (( \pt J, \lambda )).
\end{aligned}
\end{equation}

Thus, if we separate the terms with $\bar{E}$ in \eqref{firstordEnew} and recognize terms corresponding to the Lagrangian $ L(E,\lambda,\varepsilon,\sigma)$
 we can rewrite it as 
\begin{align*}
    L(E+\bar{E}, \lambda, \varepsilon, \sigma) &= L(E,\lambda,\varepsilon,\sigma) + \langle \langle E - \Etobs, \bar{E} \rangle \rangle_{\GammaT} 
-
(( \pt \bar{E}, \pt \lambda))_\varepsilon \\
    &- (( \bar{E}, \pt \lambda ))_\sigma 
    + ((\nabla \bar{E}, \nabla \lambda )) + (( \nabla \cdot (\varepsilon - 1) \bar{E}, \nabla \cdot \lambda )) + \frac{1}{2} \| \bar{E} \|^2_{\GammaT} \\
    &= L(E,\lambda, \varepsilon, \sigma) + \pdv{L}{E} (E, \lambda, \varepsilon, \sigma; \bar{E}) + \mathcal{O}(\|\bar{E}\|^2).
\end{align*}

Now, considering   $ L(E+\bar{E}, \lambda, \varepsilon, \sigma) -  L(E,\lambda,\varepsilon,\sigma) $ and neglecting $\mathcal{O}(\|\bar{E}\|^2)$ we get optimality condition \eqref{optcond1}.

The three other derivatives are approached very similarly. The derivative with respect to $\lambda$ is simply a case of linearity of gradients and integrals,
\begin{align*}
    L(E,\lambda + \bar{\lambda}, \varepsilon, \sigma) &= F(E,\varepsilon,\sigma) - (f_1 ,(\lambda + \bar{\lambda}) (x,0))_\varepsilon - (( \pt E, \pt (\lambda+\bar{\lambda}))_\varepsilon \\
    &- (f_0, (\lambda+\bar{\lambda})(x,0))_\sigma - (( E, \pt (\lambda+\bar{\lambda}) ))_\sigma 
    + ((\nabla E, \nabla (\lambda+\bar{\lambda}) )) \\
    &+ (( \nabla \cdot (\varepsilon - 1) E, \nabla \cdot (\lambda+\bar{\lambda}) )) + (( \pt J, (\lambda+\bar{\lambda}) )) \\
    &= L(E,\lambda, \varepsilon, \sigma) -(f_1, \bar{\lambda}(x,0))_\varepsilon - ((\pt E, \pt \bar{\lambda} ))_\varepsilon - (f_0 , \bar{\lambda} (x,0))_\sigma \\
    &- ((E, \pt \bar{\lambda} ))_\sigma + ((\nabla E, \nabla \bar{\lambda} )) + ((\nabla \cdot (\varepsilon - 1)E, \nabla \cdot \bar{\lambda})) + ((\pt J , \bar{\lambda})) \\
    &= L(E, \lambda, \varepsilon, \sigma) + \pdv{L}{\lambda}(E, \lambda, \varepsilon, \sigma; \bar{\lambda})  + \mathcal{O}(\|\bar{\lambda}\|^2). 
\end{align*}

 Considering   now  $ L(E, \lambda + \bar{\lambda} , \varepsilon, \sigma) -
  L(E,\lambda,\varepsilon,\sigma) $ and neglecting $\mathcal{O}(\|\bar{\lambda}\|^2)$ we get optimality condition \eqref{optcond2}.

For the derivative with respect to $\varepsilon$ we have some second order terms from $F$,  and we get:
\begin{align*}
    L(E, \lambda, \varepsilon + \bar{\varepsilon}, \sigma) ) &= F(E,\varepsilon + \bar{\varepsilon},\sigma) - ((\varepsilon + \bar{\varepsilon})f_1 , \lambda (x,0)) - (((\varepsilon + \bar{\varepsilon}) \pt E, \pt \lambda)) \\
    &- (f_0, \lambda (x,0))_\sigma - (( E, \pt \lambda ))_\sigma 
    + ((\nabla E, \nabla \lambda )) \\
    &+ (( \nabla \cdot ((\varepsilon + \bar{\varepsilon}) - 1) E, \nabla \cdot \lambda )) + (( \pt J, \lambda ))  \\
    &= \frac{1}{2} \| E - \Etobs \|^2_{\GammaT} + \frac{\gamma_{\varepsilon}}{2} \| (\varepsilon + \bar{\varepsilon}) - \varepsilon^0 \|^2 + \frac{\gamma_{\sigma}}{2} \| \sigma - \sigma^0 \|^2 \\
    &- (f_1 , \lambda (x,0))_\varepsilon - (( \pt E, \pt \lambda))_\varepsilon - (f_0, \lambda (x,0))_\sigma - (( E, \pt \lambda ))_\sigma \\
    &+ ((\nabla E, \nabla \lambda )) + (( \nabla \cdot (\varepsilon - 1) E, \nabla \cdot \lambda )) + (( \pt J, \lambda )) \\
    &- (\bar{\varepsilon}f_1 , \lambda (x,0)) - ((\bar{\varepsilon} \pt E, \pt \lambda)) +(( \nabla \cdot \bar{\varepsilon} E, \nabla \cdot \lambda )) \\
    &= L(E,\lambda,\varepsilon,\sigma) + \gamma_{\varepsilon}(\varepsilon-\varepsilon^0,\bar{\varepsilon}) - (\bar{\varepsilon}f_1 , \lambda (x,0)) - ((\bar{\varepsilon} \pt E, \pt \lambda)) \\
    &+(( \nabla \cdot \bar{\varepsilon} E, \nabla \cdot \lambda )) + \frac{\gamma_{\varepsilon}}{2}\| \bar{\varepsilon}\|^2 \\
    &= L(E,\lambda,\varepsilon,\sigma) + \pdv{L}{\varepsilon} (E,\lambda,\varepsilon,\sigma; \bar{\varepsilon}) + \mathcal{O}(\|\bar{\varepsilon}\|^2).
\end{align*}
  Considering   now
$ L(E, \lambda, \varepsilon  + \bar{\varepsilon}, \sigma) -
  L(E,\lambda,\varepsilon,\sigma) $ and neglecting $\mathcal{O}(\|\bar{\varepsilon}\|^2)$ we get optimality condition \eqref{optcond3}.

The derivation of the derivative with respect to $\sigma$ is very similar, thus we omit some steps and write:
\begin{align*}
    L(E,\lambda, \varepsilon,\sigma + \bar{\sigma}) &= F(E,\varepsilon,\sigma+\bar{\sigma}) - (f_1 , \lambda (x,0))_\varepsilon - (( \pt E, \pt \lambda))_\varepsilon - ((\sigma+\bar{\sigma})f_0, \lambda (x,0)) \\
    &- (((\sigma+\bar{\sigma}) E, \pt \lambda )) 
    + ((\nabla E, \nabla \lambda )) + (( \nabla \cdot (\varepsilon - 1) E, \nabla \cdot \lambda )) + (( \pt J, \lambda )) \\
    &= L(E,\lambda,\varepsilon,\sigma) + \gamma_{\sigma}(\sigma - \sigma^0, \bar{\sigma}) - (\bar{\sigma}f_0, \lambda (x,0))- ((\bar{\sigma} E, \pt \lambda )) + \frac{\gamma_{\sigma}}{2} \|\bar{\sigma}\|^2 \\
    &= L(E,\lambda,\varepsilon,\sigma) + \pdv{L}{\sigma} (E,\lambda,\varepsilon,\sigma; \bar{\sigma}) + \mathcal{O} (\| \bar{\sigma}\|^2). 
\end{align*}
  Considering   now
$ L(E, \lambda, \varepsilon, \sigma + \bar{\sigma} ) -
  L(E,\lambda,\varepsilon,\sigma) $ and neglecting
   $\mathcal{O}(\| \bar{\sigma}  \|^2)$ we get the final optimality condition
   \eqref{optcond4}.

\end{proof}

It is quite simple to deduce that the optimality condition $ \pdv{L}{\lambda}(u; \bar{\lambda}) = 0$ for all $\bar{\lambda} \in H^1_\lambda$ is equivalent to the weak formulation of the forward problem, and the optimality condition $\pdv{L}{E}(u; \bar{E}) = 0$ for all $\bar{E} \in H^1_0$ also describes  the \textit{adjoint problem}. 

Interpreting the Fréchet derivative as a weak formulation we want to find $\lambda \in H^1(\OT)$ such that 
\begin{subnumcases}{}\label{adj}
    \varepsilon\ptt \lambda - \sigma \pt \lambda - \Delta \lambda  - (\varepsilon - 1) \nabla \nabla \cdot \lambda =0 & for $(x,t) \in \OT$, \label{adj1}\\
    \lambda(x,T) = \pt \lambda(x,T) = 0 & for $x \in \Omega$, \label{adj2}\\
    \pn \lambda = -(E - \Etobs) & for $(x,t) \in \GammaT$. \label{adj3}
\end{subnumcases}
Note that this problem is solved backwards in time, thus the termination conditions and the sign swapped time derivatives compared to the forward problem.

\section{Stability analysis for the forward and adjoint problems}

\label{sec:stability}

In this section we will present stability estimates for the forward problem as well as the adjoint problem. We refer to \cite{lad} for general approach of derivation of stability estimates for PDEs. The proof of the forward problem's stability is derived in \cite{BL2}.  

\begin{theorem}
Assume $\varepsilon \in C_\varepsilon$ and $ \sigma \in C_\sigma$. Let $\Omega \subset \mathbb{R}^{3}$ be a bounded domain with the piecewise smooth boundary $\partial \Omega$. For any $ t\in \left( 0,T\right)$ let $\Omega_{t}= \Omega\times \left( 0, t \right)$ and $\Gamma_{t} = \Gamma \times (0, t).$ Suppose that there exists a solution $E\in H^{2}\left( \Omega_{T}\right) $ of the model problem (\ref{forward}--c). Then the vector $E$ is unique and there exists a constant $C = C(\|\varepsilon\|_{C^2(\Omega)}, \| \sigma\|, t)$ such that the following energy estimate holds true:

\begin{equation}\label{energyestimate}
\begin{split}
\vert\vert\vert E (t)\vert\vert \vert^2  := & \left\Vert \partial _{t}E(t) \right\Vert_{\varepsilon}^{2} + \left\Vert E(t) \right\Vert_{\sigma}^{2} + \left\Vert \nabla E (t)\right\Vert^{2} + \left\Vert \nabla \cdot E(t) \right\Vert _{\varepsilon - 1}^{2} \\
&\leq C \left[ \int_t^T\| \pt J (\tau)\|^2 \, \dtau + \left\Vert f_{1} \right\Vert _{\varepsilon}^{2} + \left\Vert \nabla f_{0}\right\Vert ^{2} + \left\Vert f_{0}\right\Vert _{\sigma}^{2} + \left\Vert \nabla \cdot f_{0}\right\Vert_{\varepsilon - 1}^{2}\right].
\end{split}
\end{equation}
\end{theorem} 

The adjoint stability looks quite similar to the forward stability, both in terms of statement and proof. Thus we use the same ideas here for the proof as in \cite{BL2}. 

\begin{theorem}
    Assume $\varepsilon \in C_\varepsilon$ and $\sigma \in C_\sigma$ holds. Let $\Omega \subset \mathbb{R}^{3}$ be a bounded domain with the piecewise smooth boundary $\partial \Omega$. For any $ t\in \left( t,T\right)$ let $\Omega_{t}= \Omega\times \left( t, T \right)$ and $\Gamma_{t} = \Gamma \times (t, T).$ Suppose that there exists a solution $\lambda \in H^{2}\left( \Omega_{T}\right) $ of the adjoint problem (\ref{adj}--c). Then the vector $\lambda$ is unique and there exists a constant $C = C(\|\varepsilon\|_{C^2},  t)$ such that the following energy estimate holds true:
    \begin{equation}\label{energylambda}
        |||\lambda (t)|||^2 \leq C \|E - \Etobs \|_\GammaT^2 
    \end{equation}

\end{theorem}

    \begin{proof}
        We begin our proof by multiplying \eqref{adj} by $2\pt \lambda$ and integrate over $\Omega_{t}$
        \begin{equation}
            ((\varepsilon\ptt \lambda - \sigma \pt \lambda - \Delta \lambda  - (\varepsilon - 1) \nabla \nabla \cdot \lambda, 2\pt\lambda ))_{\Omega_{t}} = 0, \nonumber
        \end{equation}
        or equivalently and by naming the terms
        \begin{align*}
            I_1 + I_2 + I_3 + I_4 &:= 2((\ptt \lambda, \pt \lambda ))_ {\epsilon, \Omega_{t}} - 2\| \pt \lambda \|^2_{\sigma, \Omega_{t}}  -2 ((\Delta \lambda, \pt \lambda))_{\Omega_{t}} \\
            &-2 (( \nabla \nabla \cdot \lambda, \pt \lambda ))_{\varepsilon - 1, \Omega_{t}} = 0.
        \end{align*}
        We will estimate these terms one by one. First, by using integration by parts in time we have 
        \begin{align*}
            I_1 &= 2((\ptt \lambda, \pt \lambda ))_ {\varepsilon, \Omega_{t}} 
            = \pt \| \pt \lambda \|^2_{\varepsilon, \Omega_{t}} 
            = \|\pt \lambda\|^2_{\varepsilon} |_{\tau=t}^T 
            = -\| \pt \lambda (t)\|^2_{\varepsilon},
        \end{align*} 
        where we made use of the chain rule and the finality conditions put on $\lambda$. Before we show an upper bound for $I_2$, we note that
        \begin{equation*}
            \sqrt{\sigma} \lambda (t) = \sqrt{\sigma} \lambda (T)  + \int_T^{t}  \sqrt{\sigma} \pt \lambda(\tau) \, \dtau = - \int^T_{t}  \sqrt{\sigma} \pt \lambda(\tau) \, \dtau,  
        \end{equation*}
        By squaring both sides and integrating over $\Omega$ we arrive at 
        \begin{equation*}
            \| \lambda (t) \|^2_{\sigma} \leq \|\pt \lambda \|^2_{\sigma, \Omega_{t}} 
        \end{equation*}
        and thus
        \begin{equation*}
            I_2 = -2\|\pt \lambda \|^2_{\sigma, \Omega_{t}} = -\|\pt \lambda \|^2_{\sigma, \Omega_{t}} -\|\pt \lambda \|^2_{\sigma, \Omega_{t}}\leq  -\| \lambda (t) \|^2_{\sigma} - \|\pt \lambda \|^2_{\sigma, \Omega_{t}}.
        \end{equation*}
        For estimation of $I_3$ we simply integrate by parts,
        \begin{align*}
            I_3 &= -2(( \Delta \lambda, \pt \lambda ))_{\Omega_{t}} \\
            &= -2 ((\pn \lambda, \pt \lambda))_{\Gamma_{t}} + 2((\nabla \lambda, \nabla \pt \lambda ))_{\Omega_{t}} \\
            &= 2\langle \langle E-\Etobs,\pt \lambda \rangle \rangle +\pt \| \nabla  \lambda \|^2_{\Omega_{t}} \\
            &\leq \|E-\Etobs\|^2_{\Gamma_t} + \| \pt \lambda \|^2_{\Gamma_t} + \|\nabla \lambda \|^2_{\Omega_t} |^T_{\tau=t} \\
            &= \|E-\Etobs\|^2_{\Gamma_t} + \| \pt \lambda \|^2_{\Gamma_t}-\|\nabla\lambda(t)\|^2,
        \end{align*}
        where we used that $2\langle \langle E-\Etobs,\pt \lambda \rangle \rangle \leq \|E-\Etobs\|^2_{\Gamma_t} + \| \pt \lambda \|^2_{\Gamma_t}$.

We use a similar approach to estimate $I_4$. 
        \begin{align*}
            I_4 &= -2(( \nabla\nabla \cdot \lambda , \pt \lambda ))_{\varepsilon-1,\Omega_{t}} \\
            &= -2((\nabla\cdot\lambda , \pt \lambda \cdot n))_{\varepsilon-1, \Gamma_{t}} + 2((\nabla \cdot \lambda, \nabla \cdot (\varepsilon-1)\pt\lambda))_{\Omega_{t}} \\
            &= 2(( \nabla \cdot \lambda, (\varepsilon-1) \nabla \cdot \pt \lambda + \nabla \varepsilon\cdot \pt \lambda ))_{\Omega_{t}} \\
            &= \pt \|\nabla \cdot \lambda \|^2_{\varepsilon - 1, \Omega_{t}} + 2 ((\nabla\cdot \lambda, \nabla \varepsilon \cdot \pt \lambda ))_{\Omega_{t}} \\
            &= \| \nabla \cdot \lambda (t)\|^2_{\varepsilon-1}|^T_{\tau = t} + 2 ((\nabla\cdot \lambda, \nabla \varepsilon \cdot \pt \lambda ))_{\Omega_{t}}.\\
            &= -\| \nabla \cdot \lambda (t)\|^2_{\varepsilon-1} + 2 ((\nabla\cdot \lambda, \nabla \varepsilon \cdot \pt \lambda ))_{\Omega_{t}}.
        \end{align*}
        Finally, collecting all terms in estimates for $I_1, I_2, I_3, I_4$ and adding a  new  positive term  $ \|\pt \lambda\|^2_{\Omega_{t}}$ we get:
        \begin{equation}\label{proofsum}
        \begin{aligned}
            0 &= I_1 + I_2 + I_3+ I_4 \\
            &\leq -\| \pt \lambda (t)\|^2_{\varepsilon} -\| \lambda (t) \|^2_{\sigma} - \|\pt \lambda \|^2_{\sigma, \Omega_{t}} -\|\nabla\lambda(t)\|^2 -\| \nabla \cdot \lambda (t)\|^2_{\varepsilon-1} \\
            &+ 2 ((\nabla\cdot \lambda, \nabla \varepsilon \cdot \pt \lambda ))_{\Omega_{t}} + \|E-\Etobs\|^2_{\Gamma_t} + \| \pt \lambda \|^2_{\Gamma_t} + \|\pt \lambda\|^2_{\Omega_{t}}
        \end{aligned}
        \end{equation}
        Now, we gather all terms   integrated over $\Omega$ in \eqref{proofsum} on the left hand side, together with additional term  $\| \pt \lambda \|^2_{\Gamma_t}$  which we will estimate later, and leave the rest of the terms on the right hand side in order to get the following estimate:
 \begin{equation}\label{newsum}
 \begin{split}
           & \| \pt \lambda (t)\|^2_{\varepsilon} + \| \lambda (t) \|^2_{\sigma} + \|\nabla\lambda(t)\|^2 + \| \nabla \cdot \lambda (t)\|^2_{\varepsilon-1} - \| \pt \lambda \|^2_{\Gamma_t} \\
            &\leq -\|\pt \lambda \|^2_{\sigma, \Omega_{t}} + 2 ((\nabla\cdot \lambda, \nabla \varepsilon \cdot \pt \lambda ))_{\Omega_{t}} + \| E - \Etobs\|^2_{\Gamma_{t}}  + \|\pt \lambda\|^2_{\Omega_{t}} \\
            &\leq \| E - \Etobs\|^2_{\Gamma_{t}}
	    + \| \nabla \varepsilon\|^2_\infty\| \pt \lambda\|^2_{\Omega_{t}}
	    + \|\lambda\|^2_{\sigma, \Omega_{t}} + \|\nabla\lambda\|^2_{\Omega_{t}} + \| \nabla \cdot \lambda\|^2_{\Omega_{t}},
	    \end{split}
        \end{equation}

To estimate the term  $\| \pt \lambda \|^2_{\Gamma_t}$  in the left hand side of \eqref{newsum} we  use the condition \eqref{adj2}, or  $\lambda (x,T) =0$,
   such that  we can write
  \begin{equation*}
        \lambda (x,t) =  \lambda (x,T)  + \int_T^{t}   \pt \lambda(x,\tau) \, \dtau =
	- \int^T_{t} \pt \lambda(x,\tau) \, \dtau. 
        \end{equation*}
      Now,   squaring both sides of the above expression and  integrating then over $\Gamma$ we get
        \begin{equation*}
            \| \lambda  \|_\Gamma^2 \leq \|\pt \lambda \|^2_{\Gamma_{t}},
        \end{equation*}
        which can be rewritten as
        \begin{equation}\label{estimategamma}
          - \|\pt \lambda \|^2_{\Gamma_{t}} \leq   -  \| \lambda  \|_\Gamma^2.
        \end{equation}
Inserting the estimate \eqref{estimategamma} into the left hand side of  \eqref{newsum}, moving then
 the term $ \| \lambda  \|_\Gamma^2$ to the right hand side of the obtained expression and  estimating this term using the  trace  theorem
\begin{equation*}
 \| \lambda  \|_\Gamma^2 \leq \int_{\Omega} (\lambda)^2  dx  +
 \int_{\Omega} ( \nabla \lambda)^2  dx = \| \lambda\|^2  +  \| \nabla \lambda\|^2
\end{equation*}

we finally obtain:
        \begin{equation}\label{tripplelambda}
	\begin{split}
             ||| \lambda(t)|||^2 &= \| \pt \lambda (t)\|^2_{\varepsilon} + \| \lambda (t) \|^2_{\sigma} + \|\nabla\lambda(t)\|^2 + \| \nabla \cdot \lambda (t)\|^2_{\varepsilon-1} \\
            &\leq \| \lambda  \|_\Gamma^2  -\|\pt \lambda \|^2_{\sigma, \Omega_{t}} + 2 ((\nabla\cdot \lambda, \nabla \varepsilon \cdot \pt \lambda ))_{\Omega_{t}} + \| E - \Etobs\|^2_{\Gamma_{t}} + \|\pt \lambda\|^2_{\Omega_{t}} \\
            &\leq \| E - \Etobs\|^2_{\Gamma_{t}} + C (\| \nabla \varepsilon\|^2_\infty\| \pt \lambda\|^2_{\Omega_{t}} + \|\lambda\|^2_{\sigma, \Omega_{t}} + \|\nabla\lambda\|^2_{\Omega_{t}} + \| \nabla \cdot \lambda\|^2_{\Omega_{t}}),
	    \end{split}
        \end{equation}
        or equivalently
  \begin{equation}\label{tripplelambda}
	\begin{split}
             ||| \lambda(t)|||^2 &= \| \pt \lambda (t)\|^2_{\varepsilon} + \| \lambda (t) \|^2_{\sigma-1} + \|\nabla\lambda(t)\|^2 + \| \nabla \cdot \lambda (t)\|^2_{\varepsilon-1} \\
            &\leq \| E - \Etobs\|^2_{\Gamma_{t}} + C (\| \nabla \varepsilon\|^2_\infty\| \pt \lambda\|^2_{\Omega_{t}} + \|\lambda\|^2_{\sigma, \Omega_{t}} + \|\nabla\lambda\|^2_{\Omega_{t}} + \| \nabla \cdot \lambda\|^2_{\Omega_{t}}).
	    \end{split}
        \end{equation}

Using
        \begin{equation*}
            a(t) \leq b(t) + C\int_{t}^T a(\tau)  \, \dtau 
        \end{equation*}
        with 
        \begin{align*}
            a(t) := |||\lambda(t)|||^2, \quad b(t) := \int_t^T \| E - \Etobs\|^2_{\Gamma}, \quad C:=C(\|\varepsilon \|_{C^2}).   
        \end{align*}
        An application of Grönwall's inequality  to
	 \eqref{tripplelambda}  yields the result given  in \eqref{energylambda}.
    \end{proof}

\section{Mathematical analysis of optimization problem}

\label{sec:analysis}

To minimize our functional we will use gradient methods, so we begin to investigate the Fréchet derivative of $J$ with respect to $(\varepsilon, \sigma)$. Before we do so we will have to establish mathematical analysis of our optimization problem.

Let  us denote  perturbations of coefficients  $\varepsilon, \sigma$
 as
 $\epsd := \epst 
- \varepsilon$, $\sigd := \sigt - \sigma$, respectively.

\begin{lemma}\label{lemma1}
    The mapping $(\varepsilon, \sigma) \rightarrow E(\varepsilon, \sigma)$ is Lipschitz continuous, i.e. for any $\varepsilon, \varepsilon + \delta \varepsilon$
      and  $\sigma, \sigma + \delta \sigma$  satisfying conditions \eqref{coeffs}  we have

\begin{equation}\label{lipcontE}
\| E(\varepsilon + \delta \varepsilon,  \sigma + \delta \sigma)    -
E(\varepsilon,  \sigma)   \|  \leq  C ( \| \epsd \| + \| \sigd \| ),
\end{equation}
where $C = C(\tau, \|f_0\|, \|f_1\|) = const. > 0.$
\end{lemma}

\begin{proof}
Let us introduce a new problem, similar to (\ref{system6eq1}--c) but for  perturbated coefficients
 $ \epst =
\epsd + \varepsilon$, $\sigt =  \sigd + \sigma$:
\begin{subnumcases}{}
    \epst \ptt \Et + \sigt \pt \Et - \Delta \Et + \nabla \nabla \cdot \Et - \nabla \nabla \cdot (\epst \Et) = 0  & for $ (x, t) \in \OT $, \label{system7eq1}\\
    \Et(x, 0) = f_0, ~ \pt \Et(x, 0) = f_1 & for $x \in \Omega$, \label{initcond7} \\
    \pn\Et = 0 & for $(x, t) \in \GammaT$. \label{boundcond7}
\end{subnumcases}
If we now subtract (\ref{system6eq1}--c) from (\ref{system7eq1}--c), we get a third system 
\begin{subnumcases}{}
    \epst \ptt \Et - \varepsilon \ptt E + \sigt \pt \Et - \sigma \pt E - \Delta \Et + \Delta E \label{system8eq1}\\ 
    + \nabla \nabla \cdot \Et  - \nabla \nabla \cdot E - \nabla \nabla \cdot (\epst \Et ) + \nabla \nabla \cdot (\varepsilon E) = 0 
    & for $ (x, t) \in \OT $,  \nonumber \\
    \Et(x, 0) - E(x, 0) = 0, ~ \pt \Et (x, 0) - \pt E(x, 0) = 0 & for $x \in \Omega$, \label{initcond8} \\
    \pn(\Et - E) = 0 & for $(x, t) \in \GammaT$. \label{boundcond8}
\end{subnumcases}
By introducing $\Eb := \delta E = \Et - E$ and rearranging  terms, we can rewrite \eqref{system8eq1} as
\begin{align*}
    &\epst \ptt \Et - \varepsilon \ptt E + \sigt \pt \Et - \sigma \pt E - \Delta \Et + \Delta E + \nabla \nabla \cdot \Et  - \nabla \nabla \cdot E - \nabla \nabla \cdot (\epst \Et ) + \nabla \nabla \cdot (\varepsilon E) \\
    &= \epst \ptt \Et - \varepsilon \ptt \Et + \varepsilon \ptt \Et - \varepsilon \ptt E + \sigt \pt \Et - \sigma \pt \Et + \sigma \pt \Et - \sigma \pt E - \Delta \Et + \Delta E \\
    &+ \nabla \nabla \cdot \Et  - \nabla \nabla \cdot E - \nabla \nabla \cdot (\epst \Et ) - \nabla \nabla \cdot (\varepsilon \Et ) + \nabla \nabla \cdot (\varepsilon \Et )+ \nabla \nabla \cdot (\varepsilon E) \\
    &= \delta\varepsilon \ptt \Et + \varepsilon \ptt \Eb + \delta\sigma \pt \Et + \sigma \pt \Eb - \Delta \Eb + \nabla \nabla \cdot \Eb - \nabla \nabla \cdot(\delta \varepsilon \Et) - \nabla\nabla\cdot(\varepsilon\Eb) = 0,
\end{align*}
and thus, the whole system can be rewritten as
\begin{subnumcases}{}
    \varepsilon \ptt \Eb + \sigma  \pt \Eb - \Delta \Eb 
    + \nabla \nabla \cdot \Eb  - \nabla \nabla \cdot (\varepsilon \Eb ) \label{system9eq1} \\ 
    = -\epsd \ptt \Et - \sigd \pt \Et - \nabla \nabla \cdot (\epsd \Et)  
    & for $ (x, t) \in \OT $,  \nonumber \\
    \Eb(x, 0) = 0, ~ \partial_t \Eb(x, 0) = 0 & for $x \in \Omega$, \label{initcond9} \\
    \pn\Eb = 0 & for $(x, t) \in \GammaT$.  \label{boundcond9}
\end{subnumcases}

Applying stability estimate of the Theorem 4.1 for the problem \eqref{system9eq1}-\eqref{boundcond9}
we can  write that
\begin{equation}\label{Edineqnew}
\begin{split}
    |||  \Eb |||^2 &\leq  C  \int_t^T \left [ \| \epsd \| \| \ptt \Et \|  +
    \| \sigd \| \| \pt \Et \| +
    \| \nabla \nabla \cdot (\epsd \Et) \| \right]^2  d\tau.  \\
    \end{split}
\end{equation}

For estimation of the term $\|\nabla \nabla \cdot (\epsd \Et)\|$  in
 \eqref{Edineqnew}one can use:
\begin{equation} \label{estimnablaeps}
\begin{split}
\| \nabla \nabla \cdot (\epsd \Et)\| = \|\nabla [ \nabla \delta \varepsilon \tilde{E} + \delta \varepsilon \nabla \cdot \tilde{E}] \| \leq C ~~ \|\delta \varepsilon \| \| \tilde{E}\|.
\end{split}
\end{equation}

Next, applying first Taylor's expansions
\begin{equation}
\begin{split}
\tilde{E}(t + \tau) &= \tilde{E}(t) + \partial_t \tilde{E}(t) \tau +
\partial_{tt} \tilde{E}(t)\frac{\tau^2}{2} + O(\tau^3),\\
\tilde{E}(t + \tau) &= \tilde{E}(t) + \partial_t \tilde{E}(t)\tau +  O(\tau^2)
\end{split}
\end{equation}
 for estimation of terms   $\partial_{tt} \tilde{E},  \partial_{t} \tilde{E} $   in the right hand side of the estimate
\eqref{Edineqnew}
such that
\begin{equation}
\begin{split}
\partial_{tt} \tilde{E}(t)  & \approx
\frac{ \tilde{E}(t+\tau) - 2 \tilde{E}(t) + \tilde{E}(t - \tau)}{\tau^2},\\
\partial_{t} \tilde{E}(t)  & \approx \frac{ \tilde{E}(t+ \tau) - \tilde{E}(t-\tau) }{2\tau}
\end{split}
\end{equation}

and  using  after the Theorem 4.1 for the stability  estimate 
of solution $\tilde{E}$ of the   problem \eqref{system7eq1}-\eqref{boundcond7}  we obtain finally
\begin{equation}\label{Lest}
   ||| \Eb |||^2 \leq C(\tau, \|f_0\|, \|f_1\|) ( \| \epsd \| + \| \sigd \|)^2.
\end{equation}
From \eqref{Lest}  follows \eqref{lipcontE}  and thus,
  the mapping $(\varepsilon, \sigma) \rightarrow E(\varepsilon, \sigma)$ is Lipschitz continuous.

\end{proof} 

\begin{lemma}
    The mapping $(\varepsilon, \sigma) \rightarrow E(\varepsilon, \sigma)$ is Fréchet differentiable  in the sense that  for any $\delta \varepsilon, \delta \sigma$ satisfying conditions \eqref{coeffs}  there exists a bounded linear operator $ A(\epsd, \sigd)$  such that
  \begin{equation*}
  \begin{split}
       \lim_{\| ( \epsd, \sigd ) \| \rightarrow \, 0 } \frac{\| \Et - E - A(\epsd, \sigd) \| }{\| ( \epsd, \sigd ) \|} = 0.
	\end{split}
    \end{equation*}   

Here, $\| ( \epsd, \sigd ) \|:= \| \epsd\| + \| \sigd \|$,  $A(\epsd, \sigd) =
 A(\epsd) + A(\sigd) $, where $ A(\epsd),~ A(\sigd)$ are bounded  linear operators.
\end{lemma}

\begin{proof}

 Let us continue with the notation introduced in Lemma \ref{lemma1}.
Using $\Eb = \Et - E$ we will let $w := \Et - E - A(\epsd, \sigd) =  \Eb - A(\epsd, \sigd)$.  We look  for an operator $A(\epsd, \sigd)$ such that 
    \begin{equation*}
        \lim_{\| ( \epsd, \sigd ) \| \rightarrow \, 0 } \frac{\| w \| }{\| ( \epsd, \sigd ) \|} = 0,
    \end{equation*}
    where $\Et := E(\varepsilon + \epsd, \sigma + \sigd)$.

Now using $\Et = \Eb + E$ we can rewrite (\ref{system9eq1}--c) as 
\begin{subnumcases}{}
    \varepsilon \ptt \Eb + \sigma  \pt \Eb - \Delta \Eb 
    + \nabla \nabla \cdot \Eb  - \nabla \nabla \cdot (\varepsilon \Eb )
    \label{system10eq1} \\ 
    = -\epsd \ptt (\Eb + E) - \sigd \pt (\Eb + E) - \nabla \nabla \cdot (\epsd (\Eb + E))  
    & for $ (x, t) \in \OT $,  \nonumber \\
    \Eb(x, 0) = 0, ~ \partial_t \Eb(x, 0) = 0 & for $x \in \Omega$, \label{initcond10} \\
    \pn\Eb = 0  & for $(x, t) \in \GammaT$.  \label{boundcond10}
\end{subnumcases}

Let  $ u := A(\epsd, \sigd)$  and  $\Eb = w + u$. This notation allows us rewrite  the system above as

\begin{subnumcases}{}
    \varepsilon \ptt (w + u) + \sigma  \pt (w + u) - \Delta (w + u) 
    \nonumber \\ 
    + \nabla \nabla \cdot (w + u) - \nabla \nabla \cdot (\varepsilon (w + u) ) \label{systemwueq1}\\
    = -\epsd \ptt (\Eb + E) -\sigd \pt (\Eb + E) - \nabla \nabla \cdot (\epsd (\Eb + E))  
    & for $ (x, t) \in \OT $,  \nonumber \\
    \Eb(x, 0) = 0, ~ \partial_t \Eb(x, 0) = 0 & for $x \in \Omega$, \label{initcondwu} \\
    \pn\Eb = 0 & for $(x, t) \in \GammaT$,  \label{boundcondwu}
\end{subnumcases}
where $w$ solves 
\begin{subnumcases}{\label{problemw}}
    \varepsilon \ptt w + \sigma  \pt w - \Delta w  + \nabla \nabla \cdot w - \nabla \nabla \cdot (\varepsilon w ) \label{systemweq1}\\
    = -\epsd \ptt \Eb - \sigd \pt \Eb - \nabla \nabla \cdot (\epsd \Eb)  
    & for $ (x, t) \in \OT $,  \nonumber \\
    w(x, 0) = 0, ~\partial_t w(x, 0) = 0 & for $x \in \Omega$, \label{initcondw} \\
    \pn w = 0 & for $(x, t) \in \GammaT$,  \label{boundcondw}
\end{subnumcases}
and $u$ solves 
\begin{subnumcases}{}
    \varepsilon \ptt u + \sigma  \pt u - \Delta u  + \nabla \nabla \cdot u - \nabla \nabla \cdot (\varepsilon u ) \label{systemueq1}\\
    = -\epsd \ptt E - \sigd \pt E - \nabla \nabla \cdot (\epsd E)  
    & for $ (x, t) \in \OT $,  \nonumber \\
    u(x, 0) = 0, ~ \partial_t u(x, 0) = 0 & for $x \in \Omega$, \label{initcondu} \\
    \pn u = 0 & for $(x, t) \in \GammaT$.  \label{boundcondu}
\end{subnumcases}
Applying first the energy estimate derived in the Theorem 4.1 for the problem  \eqref{problemw}
for function $w$  which will depend on $\Eb$, and then using the Lipschitz continuity estimate \eqref{lipcontE} for $\Eb$,  will give us 
\begin{equation}
   ||| w |||^2 \leq C \int_t^T \left [ \| \epsd \ptt \Eb \| + \| \sigd \pt \Eb \| + \| \nabla \nabla \cdot (\epsd \Eb) \| \right]^2 d\tau \leq C ( \| \epsd \| + \| \sigd \| )^2.
\end{equation}
Thus,
\begin{equation*}
    \lim_{\| ( \epsd, \sigd ) \| \rightarrow \, 0 } \frac{ || w ||^2 }{\| ( \epsd, \sigd ) \|} \leq \lim_{\| ( \epsd, \sigd ) \| \rightarrow \, 0 } \frac{C ( \| \epsd \| + \| \sigd \| )^2}{\| ( \epsd, \sigd ) \|} = 0,
\end{equation*}
which concludes the proof.
\end{proof}


\begin{theorem}
    The   regularized functional 
    \begin{equation}
        F (\varepsilon, \sigma) := \frac{1}{2} \| E - \tilde{E}_{\rm obs} \|^2_{\GammaT} +
	\frac{\gamma_{\varepsilon}}{2} \| \varepsilon - \varepsilon^0\|^2 +
        \frac{\gamma_{\sigma}}{2} \| \sigma - \sigma^0\|^2
    \end{equation}
    is Fréchet differentiable and its gradient is

  \begin{equation*}
    \begin{split}
   F ' (\varepsilon, \sigma)(x) &:=  F' (\varepsilon, \sigma)(\delta \varepsilon, \delta \sigma)(x) =
 [ \int_0^T \lambda( \epsd \ptt E + \sigd \pt E + \nabla \nabla \cdot (\epsd E))~dt ] (x) \\
 &+ 
[\gamma_{\varepsilon}( \varepsilon - \varepsilon^0) \epsd  +
	    \gamma_{\sigma}( \sigma - \sigma^0 )\sigd]  (x)
 \end{split}
   \end{equation*}      
or in the weak  form, after integration by parts and using initial conditions,  the Fr\'echet derivative is written as:
 \begin{equation}\label{weakder}
 \begin{split}
   F ' (\varepsilon, \sigma)(x) &=
 - [ f_1 \cdot \lambda |_{t=0} \delta \varepsilon -
 \int_0^T \pt E \cdot \pt \lambda \delta \varepsilon ~dt
 + \int_0^T \nabla \cdot \lambda  \nabla \cdot (\delta \varepsilon E )~dt\\
 &- f_0 \cdot \lambda |_{t=0} \delta \sigma  -  \int_0^T  E \cdot \pt \lambda \delta \sigma ~dt  ] (x)\\
  &+ 
[\gamma_{\varepsilon}( \varepsilon - \varepsilon^0) \epsd  +
	    \gamma_{\sigma}( \sigma - \sigma^0 )\sigd]  (x)
\end{split}
\end{equation}

    \begin{proof}

Recalling that  $\Et = \Eb + E$  we can compute:
        \begin{equation}\label{deltaJ}
	\begin{split}
          \delta F :=  F (\varepsilon + \epsd, \sigma + \sigd) - F (\varepsilon, \sigma) &=
	  \frac{1}{2} \| \Et - \Etobs \|^2_{\GammaT} - \frac{1}{2} \| E - \Etobs \|^2_{\GammaT} \\
&+	\frac{\gamma_{\varepsilon}}{2} \| \varepsilon + \epsd - \varepsilon^0\|^2 -
	\frac{\gamma_{\varepsilon}}{2} \| \varepsilon  - \varepsilon^0\|^2 \\
 &+      \frac{\gamma_{\sigma}}{2} \| \sigma + \sigd - \sigma^0\|^2 -
	\frac{\gamma_{\sigma}}{2} \| \sigma  - \sigma^0\|^2\\
            &= \frac{1}{2} \| \Eb + (E - \Etobs) \|^2_{\GammaT} - \frac{1}{2} \| E - \Etobs \|^2_{\GammaT} \\
 &+	\frac{\gamma_{\varepsilon}}{2} \| (\varepsilon - \varepsilon^0)   + \epsd \|^2 -
	\frac{\gamma_{\varepsilon}}{2} \| \varepsilon  - \varepsilon^0\|^2 \\
 &+      \frac{\gamma_{\sigma}}{2} \| (\sigma  - \sigma^0)  + \sigd  \|^2 -
	\frac{\gamma_{\sigma}}{2} \| \sigma  - \sigma^0\|^2\\
            &= \frac{1}{2} \| \Eb \|^2_{\GammaT} + \langle \langle \Eb , \, E - \Etobs\rangle \rangle_{\GammaT}  \\
	    &+ \gamma_{\varepsilon}( \varepsilon - \varepsilon^0, \epsd )  +
	    \gamma_{\sigma}( \sigma - \sigma^0, \sigd ) \\
	   &+ \frac{\gamma_{\varepsilon}}{2} \| \epsd \|^2 +   \frac{\gamma_{\sigma}}{2} \| \sigd \|^2 . \\
\end{split}
\end{equation}


Let us write the variational formulation of the adjoint problem
 \eqref{adj1} - \eqref{adj3} for all $\Eb \in H_\lambda^1$
 \begin{equation*}
((\Eb, \, \varepsilon \ptt \lambda - \sigma \partial_t \lambda - \Delta \lambda + (1 - \varepsilon ) \nabla \nabla \cdot \lambda )) = 0,
\end{equation*}
and then perform integration by parts to get
 \begin{equation}\label{varformadjoint}
 \begin{split}
& \langle\langle E - \Etobs,  \Eb  \rangle \rangle_{\GammaT} -
	(( \pt \Eb  , \pt \lambda))_\varepsilon \\
       & - (( \Eb  , \pt \lambda ))_\sigma 
    + ((\nabla \Eb  , \nabla \lambda )) 
    + (( \nabla \cdot (\varepsilon - 1) \Eb , \nabla \cdot \lambda )) = 0.\\
\end{split}
\end{equation}
Then integrating by parts again \eqref{varformadjoint}  and noting that $\Eb$ satisfy the problem
  \eqref{system10eq1}--\eqref{boundcond10}  we get
\begin{equation*}
 \langle\langle E - \Etobs,  \Eb  \rangle \rangle_{\GammaT} +
  ((\lambda , \, \varepsilon \ptt \Eb + \sigma  \pt \Eb - \Delta \Eb + \nabla \nabla \cdot \Eb  - \nabla \nabla \cdot (\varepsilon \Eb )))  = 0.
\end{equation*}
We rewrite the last expression as
 \begin{equation}\label{varformadjoint2}
 \begin{split}
   ((\lambda , \, \varepsilon \ptt \Eb + \sigma  \pt \Eb - \Delta \Eb + \nabla \nabla \cdot \Eb  - \nabla \nabla \cdot (\varepsilon \Eb )))  = - \langle\langle E - \Etobs,  \Eb  \rangle \rangle_{\GammaT} .
\end{split}
\end{equation}

Now
  replacing the term  $\langle\langle E - \Etobs,  \Eb  \rangle \rangle_{\GammaT} $  in the  expression
  \eqref{deltaJ}
by the left hand side of the equation  \eqref{varformadjoint2} we get:
    \begin{align*}
         \delta F 
            &= \frac{1}{2} \| \Eb \|^2_{\GammaT} 
            - ((\lambda , \, \varepsilon \ptt \Eb + \sigma  \pt \Eb - \Delta \Eb + \nabla \nabla \cdot \Eb  - \nabla \nabla \cdot (\varepsilon \Eb ))) \\
  &+ \gamma_{\varepsilon}( \varepsilon - \varepsilon^0, \epsd )  +
	    \gamma_{\sigma}( \sigma - \sigma^0, \sigd ) \\
	   &+ \frac{\gamma_{\varepsilon}}{2} \| \epsd \|^2 +   \frac{\gamma_{\sigma}}{2} \| \sigd \|^2  \\
            &= \frac{1}{2} \| \Eb \|^2_{\GammaT} - ((\lambda , \, -\epsd \ptt \Et - \sigd \pt \Et - \nabla \nabla \cdot (\epsd \Et)))
	    \\
	      &+ \gamma_{\varepsilon}( \varepsilon - \varepsilon^0, \epsd )  +
	    \gamma_{\sigma}( \sigma - \sigma^0, \sigd ) \\
	   &+ \frac{\gamma_{\varepsilon}}{2} \| \epsd \|^2 +   \frac{\gamma_{\sigma}}{2} \| \sigd \|^2 . \\
        \end{align*}
Here, we have used first integration by parts and then definition of the right hand side in the
problem \eqref{system9eq1} - \eqref{boundcond9}.

Now we use again 	$   \Et = \Eb + E = \delta E +  E$ to  obtain

     \begin{equation*}
     \begin{split}
         \delta F &= \frac{1}{2} \| \delta E \|^2_{\GammaT}
	 - ((\lambda , \, -\epsd \ptt \delta E - \sigd \pt \delta E -
	  \nabla \nabla \cdot (\epsd \delta E) ))\\
	  &- ((\lambda , \, -\epsd \ptt E - \sigd \pt E - \nabla \nabla \cdot (\epsd E))) \\
  &+ \gamma_{\varepsilon}( \varepsilon - \varepsilon^0, \epsd )  +
	    \gamma_{\sigma}( \sigma - \sigma^0, \sigd ) \\
	   &+ \frac{\gamma_{\varepsilon}}{2} \| \epsd \|^2 +   \frac{\gamma_{\sigma}}{2} \| \sigd \|^2 . \\
	 \end{split}
        \end{equation*}

We move all second  order terms with $\delta$'s in the above expression to the right hand side by leaving all first order terms on the left hand side to get
     \begin{equation*}
     \begin{split}
         \delta F &+ ((\lambda , \, -\epsd \ptt E - \sigd \pt E - \nabla \nabla \cdot (\epsd E))) \\
	   &- \gamma_{\varepsilon}( \varepsilon - \varepsilon^0, \epsd )  -
	    \gamma_{\sigma}( \sigma - \sigma^0, \sigd ) \\
 &=  \frac{1}{2} \| \delta E \|^2_{\GammaT}
	 - ((\lambda , \, -\epsd \ptt \delta E - \sigd \pt \delta E -
	  \nabla \nabla \cdot (\epsd \delta E))) \\
	   &+ \frac{\gamma_{\varepsilon}}{2} \| \epsd \|^2 +   \frac{\gamma_{\sigma}}{2} \| \sigd \|^2 . \\
	 \end{split}
        \end{equation*}

The expression above can be estimated now as
  \begin{align*}
&|  \delta F +
 ((\lambda , \, -\epsd \ptt E - \sigd \pt E - \nabla \nabla \cdot (\epsd E)))
 - \gamma_{\varepsilon}( \varepsilon - \varepsilon^0, \epsd )  - \gamma_{\sigma}( \sigma - \sigma^0, \sigd )| \\
 &\leq  \mathcal{O} (\| \epsd \|^2 + \| \sigd \|^2)
        \end{align*}  
 which can be rewritten in the following form:
    \begin{equation}\label{Freder2}
    \begin{split}
|   F (\varepsilon + \epsd, \sigma + \sigd) & - F(\varepsilon, \sigma) \\
&-
 [((\lambda , \, \epsd \ptt E + \sigd \pt E + \nabla \nabla \cdot (\epsd E))\\
&+ \gamma_{\varepsilon}( \varepsilon - \varepsilon^0, \epsd )  +
	    \gamma_{\sigma}( \sigma - \sigma^0, \sigd ))]
 | \leq  \mathcal{O} (\| \epsd \|^2 + \| \sigd \|^2).
 \end{split}
        \end{equation}

Comparing now the definition of the
 Fr\'echet derivative written in the form
  \begin{equation*}
  F (\varepsilon + \epsd, \sigma + \sigd) - F(\varepsilon, \sigma) -  F' (\varepsilon, \sigma)(\delta \varepsilon, \delta \sigma) =  \mathcal{O} (\| \epsd \|^2 + \| \sigd \|^2)
\end{equation*}
with the inequality \eqref{Freder2} we observe that the functional $F$ is Fr\'echet differentiable and the Fr\'echet derivative is given by
    \begin{equation*}
    \begin{split}
   F ' (\varepsilon, \sigma)(x) &:=  F ' (\varepsilon, \sigma)(\delta \varepsilon, \delta \sigma)(x) =
 [ \int_0^T \lambda( \epsd \ptt E + \sigd \pt E + \nabla \nabla \cdot (\epsd E))~dt ] (x) \\
 &+ 
[\gamma_{\varepsilon}( \varepsilon - \varepsilon^0) \epsd  +
	    \gamma_{\sigma}( \sigma - \sigma^0 )\sigd]  (x)
 \end{split}
   \end{equation*}      
or in the weak  form, after integration by parts and using initial conditions for the model problem,  the Fr\'echet derivative is written as:
 \begin{equation}\label{derJweak}
 \begin{split}
   F ' (\varepsilon, \sigma)(x) &=
 - [ f_1 \cdot \lambda |_{t=0} \delta \varepsilon -
 \int_0^T \pt E \cdot \pt \lambda \delta \varepsilon ~dt
 + \int_0^T \nabla \cdot \lambda  \nabla \cdot (\delta \varepsilon E )~dt\\
 &- f_0 \cdot \lambda |_{t=0} \delta \sigma  -  \int_0^T  E \cdot \pt \lambda \delta \sigma ~dt  ] (x)\\
  &+ 
[\gamma_{\varepsilon}( \varepsilon - \varepsilon^0) \epsd  +
	    \gamma_{\sigma}( \sigma - \sigma^0 )\sigd]  (x)
\end{split}
\end{equation}

\end{proof}
\end{theorem}


\begin{lemma}
 Assume that the solution $E(\varepsilon, \sigma)$
 for the forward problem   \eqref{system6eq1} -- \eqref{boundcond6} and
 the solution $\lambda(\varepsilon, \sigma) $ for the adjoint problem  \eqref{adj1} -- \eqref{adj3} are sufficiently stable, i.e. satisfy stability estimates derived in the Theorem 4.1 for the forward problem and in the Theorem 4.2  for the adjoint problem, respectively.  Then
 the Frech\'{e}t derivative of the Tikhonov functional
  $F(\varepsilon, \sigma)$
 can be computed   also via  Frech\'{e}t derivative of the Lagrangian  $L$  as:
 \begin{equation}\label{derfunc}
 \begin{split}
 F'(\varepsilon, \sigma)
 := F'(E(\varepsilon, \sigma), \varepsilon, \sigma) &=
  \frac{\partial L}{\partial
 \varepsilon}(E(\varepsilon, \sigma),\lambda, \varepsilon, \sigma ) +
 \frac{\partial L}{\partial \sigma}(E(\varepsilon, \sigma), \lambda,
 \varepsilon, \sigma).
 \end{split}
 \end{equation}
where $L$ is the Lagrangian.
\end{lemma}
\begin{proof}
Comparing 
 the Fr\'echet derivative $F'(\varepsilon, \sigma)$ in  the weak form given by
 \eqref{derJweak}
with the  Fr\'echet derivatives
 $ \frac{\partial L}{\partial
 \varepsilon}(E(\varepsilon, \sigma),\lambda, \varepsilon, \sigma ),
 \frac{\partial L}{\partial \sigma}(E(\varepsilon, \sigma), \lambda,
 \varepsilon, \sigma)$ given by optimality conditions of the Lagrangian in
  \eqref{optcond3},  \eqref{optcond4}, respectively,
 and then taking them with $\bar{\varepsilon} =\delta \varepsilon,\,  \bar{\sigma} = \delta \sigma$,
 we obtain equality \eqref{derfunc}.
\end{proof}


\section{Existence of minimizer}

\label{sec:existence}

\begin{lemma}

 The following  expression  holds for the non-regularized Tikhonov functional $ F(\varepsilon,\sigma)$:
 \begin{equation}\label{minim1}
 \begin{split}
& F(\varepsilon,\sigma) = F(\varepsilon_n,\sigma_n) -
\frac{1}{2} \int_\GammaT   (\delta E_n)^2   ~dsdt 
+ \int_\GammaT [ E(\varepsilon, \sigma) - \Etobs ] \delta E_n ~dsdt,
 \end{split}
 \end{equation}
where $\delta E_n : = E(\varepsilon,\sigma) - E(\varepsilon_n,\sigma_n)  $.
\end{lemma}

\begin{proof}

We have:

 \begin{equation}\label{minim2}
 \begin{split}
&  F(\varepsilon,\sigma) - F(\varepsilon_n,\sigma_n)   =
\frac{1}{2}\int_\GammaT  [E(\varepsilon,\sigma) - \Etobs   ]^2   ~ds dt
-
\frac{1}{2} \int_\GammaT  [E(\varepsilon_n,\sigma_n) - \Etobs ]^2   ~ds ~ dt
\\
&=\frac{1}{2}  \int_\GammaT  ([E(\varepsilon,\sigma) - \Etobs   ]^2 - [E(\varepsilon_n,\sigma_n) - \Etobs ]^2 )~ds dt.
\end{split}
\end{equation}

Let $a =E(\varepsilon,\sigma), b =  \Etobs, c = E(\varepsilon_n,\sigma_n)$.
Then we   simplify

\begin{equation}\label{abc}
[E(\varepsilon,\sigma) - \Etobs   ]^2 - [E(\varepsilon_n,\sigma_n) - \Etobs ]^2 =  (a-b)^2 - (c - b)^2
\end{equation}

as follows:
\begin{equation}\label{minim3}
 \begin{split}
 (a-b)^2 - (c - b)^2  &=   (a-b -(c-b))(a-b + (c-b)) 
 =(a-c)(a-2b+c)  \\
 &=  (a-c)((a+c) - 2b) = (a-c)(a+c) - 2b(a-c).
 \end{split}
 \end{equation}

Substituting \eqref{minim3} into \eqref{abc} and then to \eqref{minim2}  we get:

\begin{equation}\label{minim4}
 \begin{split}
F(\varepsilon,\sigma) - F(\varepsilon_n,\sigma_n)   &=  \frac{1}{2}  \int_\GammaT  (E(\varepsilon,\sigma) - E(\varepsilon_n,\sigma_n))(E(\varepsilon,\sigma) + E(\varepsilon_n,\sigma_n)  )~ds dt\\
&-  \int_\GammaT \Etobs (E(\varepsilon,\sigma) - E(\varepsilon_n,\sigma_n))~ds dt.
 \end{split}
 \end{equation}

Let $\delta E_n := E(\varepsilon,\sigma) - E(\varepsilon_n,\sigma_n); E =  E(\varepsilon,\sigma), E_n =  E(\varepsilon_n,\sigma_n) $. Then from   \eqref{minim4} it follows:

\begin{equation}\label{minim5}
 \begin{split}
F(\varepsilon,\sigma) - F(\varepsilon_n,\sigma_n)   &=  \frac{1}{2}  \int_\GammaT  \delta E_n (E  + E_n  )~ds dt
-  \int_\GammaT \Etobs  \delta E_n  ~ds dt \\
&= \frac{1}{2}  \int_\GammaT  \delta E_n (E + E_n  )~ds dt
-  \int_\GammaT \Etobs  \delta E_n  ~ds dt \\
&+  \int_\GammaT E  \delta E_n  ~ds dt -  \int_\GammaT E  \delta E_n  ~ds dt \\
&=  \frac{1}{2}  \int_\GammaT  \delta E_n  E_n  ~ds dt -  \int_\GammaT \Etobs  \delta E_n  ~ds dt  \\
&+  \int_\GammaT E  \delta E_n  ~ds dt -\frac{1}{2}  \int_\GammaT E \delta E_n ~ds dt
\\
&= \int_\GammaT (E - \Etobs)  \delta E_n  ~ds dt  -
\frac{1}{2}  \int_\GammaT (E - E_n) \delta E_n ~ds dt \\
&= \int_\GammaT (E - \Etobs)  \delta E_n  ~ds dt  -
\frac{1}{2}  \int_\GammaT (\delta E_n)^2 ~ds dt,
 \end{split}
 \end{equation}
and thus, we obtained \eqref{minim1}.
\end{proof}

\begin{theorem}

Assume  that conditions \eqref{coeffs} for functions $\varepsilon,\sigma$  holds true. Then there exists a unique minimizer of the non-regularized functional $F(\varepsilon,\sigma)$.
\end{theorem}

\begin{proof}
Using the fact that  the functional $F(\varepsilon,\sigma)$ is bounded from below  then  there exists a minimizing sequences $\{\varepsilon_n\} \in  C_\varepsilon,  \{\sigma_n\} \in C_\sigma $  which  converges weakly to  admissible functions
 $ \varepsilon \in C_\varepsilon$,  $ \sigma \in  C_\sigma$. 
By the Lemma 6.1 the following  expression  holds for the non regularized Tikhonov functional $ F(\varepsilon,\sigma)$:
 \begin{equation}\label{in1}
 \begin{split}
& F(\varepsilon,\sigma) = F(\varepsilon_n,\sigma_n) -
\frac{1}{2} \int_{\Gamma_T}  [\delta E_n  ]^2   ~dsdt \\
 & + \int_{\Gamma_T} [ E(\varepsilon, \sigma) - \Etobs ] \delta E_n ~dsdt,
 \end{split}
 \end{equation}
where $\delta E_n : = E(\varepsilon,\sigma) - E(\varepsilon_n,\sigma_n)  $.

Taking  $\delta E_n : =  \Eb $ in \eqref{in1}  and using then  \eqref{varformadjoint2} 
we get
 \begin{equation}\label{in2}
 \begin{split}
  F(\varepsilon,\sigma) & \leq
 F(\varepsilon_n,\sigma_n)  -
 ((\lambda , \, \varepsilon \ptt \Eb + \sigma  \pt \Eb - \Delta \Eb + \nabla \nabla \cdot \Eb  - \nabla \nabla \cdot (\varepsilon \Eb ))) \\
 &= F(\varepsilon_n,\sigma_n)  -  ((\lambda , \, -\delta \varepsilon_n \ptt \Et - \delta \sigma_n \pt \Et - \nabla \nabla \cdot ( \delta \varepsilon_n  \Et))).
 \end{split}
 \end{equation}
Here, we have used definition of the right hand side in the
problem \eqref{system9eq1} - \eqref{boundcond9}.

 By H\"older's inequality the term  $\int_0^T \lambda( -\delta \varepsilon_n \ptt \Et - \delta \sigma_n \pt \Et - \nabla \nabla \cdot ( \delta \varepsilon_n  \Et)
dt   \in L^{2}(\Omega)$.
Since $\varepsilon_n \ \rightharpoonup \varepsilon$,  $\sigma_n \ \rightharpoonup \sigma$  in $L^2(\Omega)$,
we get
\begin{equation*}
\begin{split}
((\lambda , \, -\varepsilon_n \ptt \Et - \sigma_n \pt \Et - \nabla \nabla \cdot (\epsilon_n \Et)))
\rightarrow
((\lambda , \, -\varepsilon \ptt \Et - \sigma \pt \Et - \nabla \nabla \cdot (\varepsilon \Et)))
\quad \mbox{as} \quad n\rightarrow \infty.
\end{split}
\end{equation*}

From the above expression follows that
\begin{equation}\label{in3}
 \begin{split}
  F(\varepsilon,\sigma) &\leq \lim_{n\rightarrow \infty} F(\varepsilon_n,\sigma_n) 
 \end{split}
 \end{equation}

Further, since  $\varepsilon_n\ \rightharpoonup \varepsilon,  \sigma_n \ \rightharpoonup \sigma$ in $L^2(\Omega)$,  it is also lower semi-continuous, what means that if
\begin{equation*}
\begin{split}
\parallel \varepsilon \parallel_{L^2(\Omega)} &\leq \lim \inf_{n\rightarrow \infty}\parallel \varepsilon_n \parallel_{L^2(\Omega)},\\
\parallel \sigma \parallel_{L^2(\Omega)} &\leq \lim \inf_{n\rightarrow \infty}\parallel \sigma_n \parallel_{L^2(\Omega)},
\end{split}
\end{equation*}
then this implies that
\begin{equation*}
F(\varepsilon,\sigma) \leq  \lim \inf_{n\rightarrow \infty} F(\varepsilon_n,\sigma_n),
\end{equation*}
and hence $F(\varepsilon,\sigma)$ is lower semi-continuous.
From the Theorem  1.9.1.2 of
\cite{BOOK} 
it also follows
that the functional $F(\varepsilon,\sigma)$ is locally strictly convex in the neighborhood of the minimizer $(\varepsilon,\sigma)$.
Thus,  using the generalized Weierstrass theorem (see \cite{zeidler2012applied}, Theorem 2D), we conclude that the regularized functional $F(\varepsilon,\sigma)$ has a unique minimizer what finishes our proof.  
 \end{proof}

\begin{lemma}

 The following  expression  holds for the regularized Tikhonov functional $ F(\varepsilon,\sigma)$:
 \begin{equation}\label{minim2}
 \begin{split}
F(\varepsilon,\sigma) &= F(\varepsilon_n,\sigma_n) -
\frac{1}{2} \int_\GammaT   (\delta E_n)^2   ~dsdt 
+ \int_\GammaT [ E(\varepsilon, \sigma) - \Etobs ] \delta E_n ~dsdt\\
&- \frac{1}{2} \gamma_{\varepsilon} \int_\Omega  (\delta \varepsilon_n)^2   ~dx
+ \gamma_{\varepsilon}  \int_\Omega [ \varepsilon - \varepsilon^0] \delta \varepsilon_n ~dx\\
&- \frac{1}{2} \gamma_{\sigma} \int_\Omega  (\delta \sigma_n)^2   ~dx
+ \gamma_{\sigma}  \int_\Omega [ \sigma - \sigma^0] \delta \sigma_n ~dx
 \end{split}
 \end{equation}
where $\delta E_n : = E(\varepsilon,\sigma) - E(\varepsilon_n,\sigma_n)  $, $\delta \varepsilon_n := \varepsilon - \varepsilon_n$, $\delta \sigma_n := \sigma - \sigma_n$.
\end{lemma}

\begin{proof}
 In the proof  we use Lemma 6.2 for the non-regularized Tikhonov functional to obtain expression
  \eqref{minim1}. In the case of the regularized Tikhonov functional \eqref{tikh}  we   should estimate additional regularization terms  appearing in the computation of
   $F(\varepsilon,\sigma) - F(\varepsilon_n,\sigma_n)$.

Thus, we first estimate  the regularization terms for $\varepsilon$ as follows. Let us take
$ a = \varepsilon, b= \varepsilon^0, c= \varepsilon_n$ in \eqref{minim3} to obtain:
\begin{equation}\label{estimeps1}
\begin{split}
(\varepsilon - \varepsilon^0)^2 - (\varepsilon_n - \varepsilon^0)^2 =  (\varepsilon - \varepsilon_n)(\varepsilon + \varepsilon_n) - 2\varepsilon^0(\varepsilon - \varepsilon_n)
\end{split}
\end{equation}
and thus,  denoting $\delta \varepsilon_n := \varepsilon - \varepsilon_n$  we get:
\begin{equation}\label{estimeps2}
\begin{split}
& \frac{1}{2} \gamma_{\varepsilon} \int_{\Omega}[ (\varepsilon - \varepsilon^0)^2 - (\varepsilon_n - \varepsilon^0)^2 ] ~ dx = \frac{1}{2} \gamma_{\varepsilon} \int_{\Omega} [(\varepsilon - \varepsilon_n)(\varepsilon + \varepsilon_n) - 2\varepsilon^0(\varepsilon - \varepsilon_n) ]~dx \\
&= \frac{1}{2} \gamma_{\varepsilon} \int_{\Omega} [\delta \varepsilon_n (\varepsilon + \varepsilon_n) - 2\varepsilon^0 \delta \varepsilon_n ]~dx  +  \gamma_{\varepsilon}  \int_{\Omega}
\varepsilon  \delta \varepsilon_n ~dx -  \gamma_{\varepsilon}  \int_{\Omega}
\varepsilon  \delta  \varepsilon_n ~dx \\
&= - \frac{1}{2}  \gamma_{\varepsilon}  \int_{\Omega} (\varepsilon - \varepsilon_n) \delta \varepsilon_n ~dx +  \gamma_{\varepsilon}  \int_{\Omega}  (\varepsilon - \varepsilon^0) \delta \varepsilon_n ~dx \\
&= - \frac{1}{2}  \gamma_{\varepsilon}  \int_{\Omega}  (\delta \varepsilon_n)^2 ~dx +  \gamma_{\varepsilon}  \int_{\Omega}  (\varepsilon - \varepsilon^0) \delta \varepsilon_n ~dx.
\end{split}
\end{equation}

In a similar way we estimate terms for $\sigma$:
\begin{equation}\label{estsigma2}
\begin{split}
& \frac{1}{2} \gamma_{\sigma} \int_{\Omega}[ (\sigma - \sigma^0)^2 - (\sigma_n - \sigma^0)^2 ] ~ dx = \frac{1}{2} \gamma_{\sigma} \int_{\Omega} [(\sigma - \sigma_n)(\sigma + \sigma_n) - 2\sigma^0(\sigma - \sigma_n) ]~dx \\
&= \frac{1}{2} \gamma_{\sigma} \int_{\Omega} [\delta \sigma_n (\sigma + \sigma_n) - 2\sigma^0 \delta \sigma_n ]~dx  +  \gamma_{\sigma}  \int_{\Omega}
\sigma  \delta \sigma_n ~dx -  \gamma_{\sigma}  \int_{\Omega}
\sigma  \delta  \sigma_n ~dx \\
&= - \frac{1}{2}  \gamma_{\sigma}  \int_{\Omega} (\sigma - \sigma_n) \delta \sigma_n ~dx +  \gamma_{\sigma}  \int_{\Omega}  (\sigma - \sigma^0) \delta \sigma_n ~dx \\
&= - \frac{1}{2}  \gamma_{\sigma}  \int_{\Omega}  (\delta \sigma_n)^2 ~dx +  \gamma_{\sigma}  \int_{\Omega}  (\sigma - \sigma^0) \delta \sigma_n ~dx.
\end{split}
\end{equation}

Combining now \eqref{minim1} with \eqref{estimeps2},\eqref{estsigma2} we get \eqref{minim2}.
\end{proof}

\begin{theorem}

Assume  that conditions \eqref{coeffs} for functions $\varepsilon,\sigma$  holds true. Then there exists a unique minimizer of the regularized functional $F(\varepsilon,\sigma)$ given in \eqref{tikh}.
\end{theorem}

\begin{proof}

We use proof for the non-regularized functional noting that by the Lemma 6.3 the  expression
 \eqref{minim2} holds for the regularized Tikhonov functional $ F(\varepsilon,\sigma)$.
Since $\varepsilon_n \ \rightharpoonup \varepsilon$,  $\sigma_n \ \rightharpoonup \sigma$  in $L^2(\Omega)$,  we get  then that the
  terms  with $\delta \varepsilon_n, \delta \sigma_n$  in \eqref{minim2} tends to zero as
$\quad n\rightarrow \infty$,   as well as
\begin{equation*}
\begin{split}
\lim_{\quad n\rightarrow \infty}  ((\lambda , \, -\varepsilon_n \ptt \Et - \sigma_n \pt \Et - \nabla \nabla \cdot (\epsilon_n \Et)))
=
((\lambda , \, -\varepsilon \ptt \Et - \sigma \pt \Et - \nabla \nabla \cdot (\varepsilon \Et))).
\end{split}
\end{equation*}
 Using these facts in \eqref{minim2} we get the inequality 
\begin{equation*}
 \begin{split}
  F(\varepsilon,\sigma) &\leq \lim_{n\rightarrow \infty} F(\varepsilon_n,\sigma_n) .
 \end{split}
 \end{equation*}

The rest of the proof
 is the same as the proof  following after \eqref{in3}.
\end{proof}

\section{Algorithms for solution of inverse problem}

\label{sec:algorithms}

In our numerical simulations we are using
two different gradient-based algorithms: conjugate gradient algorithm (CGA)
 and  adaptive CGA (ACGA).
ACGA uses adaptive
finite element method in space which significantly improves reconstructions of
$\varepsilon, \sigma$ obtained on the initially non-refined mesh using   usual  CGA - see details
 on this method in \cite{BL1,BOOK,BTKM2}.

Let us define the following functions which are obtained from optimal conditions
$ \pdv{L}{\varepsilon}(u; \tilde{\varepsilon}) = 0$ and $\pdv{L}{\sigma}(u; \tilde{\sigma})= 0$
 derived in Lemma 3.2 
and which we will use in the conjugate gradient algorithm (CGA):
\begin{align}
   & \begin{aligned}\label{gradeps}
        & g_\varepsilon (x) := [\gamma_\varepsilon (\varepsilon - \varepsilon^0)   - \lambda(x,0) f_1 
         - \int_0^T \pt\lambda \pt E ~ \dt
       +  \int_0^T  \div \lambda \div  E~ \dt](x),
     \end{aligned} \\
   & \begin{aligned}
       g_\sigma (x) & := [\gamma_\sigma (\sigma - \sigma^0)   - f_0 \lambda(x,0)  - \int_0^T E \pt \lambda \, \dt] (x).
     \end{aligned}
\end{align}
In \eqref{gradeps} we have used that $\div ( \bar{\varepsilon} E) = \nabla  \bar{\varepsilon} E + \bar{\varepsilon}\div E \approx \bar{\varepsilon}\div E . $

Let $\varepsilon^m$ and $\sigma^m$ are the coefficients computed on the iteration $m$ of the CGA,
and $E^m := E(\varepsilon^m, \sigma^m)(x,t)$, $\lambda^m := \lambda (E^m, \varepsilon^m, \sigma^m)(x,t)$, $g_\varepsilon^m := g_\varepsilon^m(x)$ and $g_\sigma^m := g_\sigma^m(x)$ where $E(\varepsilon^m, \sigma^m)$ and $\lambda (E^m, \varepsilon^m, \sigma^m)$ are the solutions of the forward and adjoint problems, respectively.
Algorithm 1 summarizes the conjugate gradient algorithm in continuous setting.


\begin{algorithm}[hbt!]
  \centering
  \caption{ Conjugate Gradient Algorithm (CGA):For iterations $m = 0, \dots,M$ perform following steps. }
    \begin{algorithmic}[1]

    \STATE Initialization: set $m = 0$ and choose initial guesses $\varepsilon^0$, $\sigma^0, \gamma_\varepsilon^0, \gamma_\sigma^0, \alpha_\varepsilon^0, \alpha_\sigma^0$.

    \STATE Calculate $E^m$, $\lambda^m$, $g^m_\varepsilon$ and $g^m_\sigma$.

    \STATE Update the dielectric permittivity as:
          \begin{align*}
            \varepsilon^{m+1}     & := \varepsilon^m + \alpha_\varepsilon^m d_\varepsilon^m,                    \\
            \quad d_\varepsilon^m & := -g_\varepsilon^{m} + \beta_\varepsilon^m d_\varepsilon^{m-1}, \\
            \beta_\varepsilon^m   & := \frac{\|g_\varepsilon^{m}\|^2}{\|g_\varepsilon^{m-1}\|^2},
          \end{align*}
    and the  conductivity as:
          \begin{align*}
            \sigma^{m+1}     & := \sigma^m + \alpha_\sigma^m d_\sigma^m,                          \\
            \quad d_\sigma^m & := -g_\sigma^{m} + \beta_\sigma^m d_\sigma^{m-1}, \\
            \beta_\sigma^m   & := \frac{\|g_\sigma^{m}\|^2}{\|g_\sigma^{m-1}\|^2},
          \end{align*}
          where $\alpha_\varepsilon^m$, $\alpha_\sigma^m$ are optimal step sizes in CGA with $d_\varepsilon^0 := - g_\varepsilon^0$, $d_\sigma^0 := - g_\sigma^0$.

    \STATE Compute new optimal step-sizes as
          \begin{equation}
            \begin{split}
              \alpha_\varepsilon^{m+1} = - \frac{(g_\varepsilon^m, d_\varepsilon^m)}{ \gamma_\varepsilon^m ( d_\varepsilon^m, d_\varepsilon^m )}, \quad
              \alpha_\sigma^{m+1} =- \frac{(g_\sigma^m, d_\sigma^m)}{ \gamma_\sigma^m ( d_\sigma^m, d_\sigma^m )}.
            \end{split}
          \end{equation}

    \STATE Compute new regularization parameters iteratively for any $p \in (0,1)$ via  rules of \cite{BKS} as

          \begin{equation}
            \begin{split}
              \gamma_\varepsilon^{m+1} = \frac{ \gamma_\varepsilon^0}{(m+1)^p},  \quad
              \gamma_\sigma^{m+1} =  \frac{ \gamma_\sigma^0}{(m+1)^p}
            \end{split}
          \end{equation}

    \STATE Terminate the algorithm if either $\|\varepsilon^{m+1} - \varepsilon^m \| < \eta_\varepsilon^1$ or $\|\sigma^{m+1} - \sigma^m \| < \eta_\sigma^1$, or $\|g_\varepsilon^m\| < \eta_\varepsilon^2$ or $\|g_\sigma^m\| < \eta_\sigma^2$, where $\eta^1_\varepsilon$, $\eta^2_\varepsilon$, $\eta^1_\sigma$ and $\eta^2_\sigma$ are tolerances chosen by the user. 
	  Otherwise, set $m: = m+1$ and repeat the algorithm from step 2.

 \end{algorithmic}
\end{algorithm}

The main idea of the adaptive local mesh refinement used in ACGA
is that the finite element mesh $K_h$
  constructed in the  finite element domain $\Omega_{\rm FEM}$
should be refined in such elements $K$  of $K_h$ where
 both functions
$|h {\varepsilon}_h^M|, |h {\sigma}_h^M| $ achieve its maximum
values  at the final iteration $M$ in CGA:
\begin{equation}\label{meshref}
  \begin{split}
    |h {\varepsilon}_h^M|
    \geq \widetilde{\beta}_\varepsilon \max \limits_{K_h} |h {\varepsilon}_h^M|, \\
    |h {\sigma}_h^M| \geq \widetilde{\beta}_\sigma \max \limits_{K_h} |h {\sigma}_h^M|.
  \end{split}
\end{equation}
Here, numbers $ \widetilde{\beta}_\varepsilon \in (0,1) , \widetilde{\beta}_\sigma \in (0,1)$
should
be chosen computationally  depending on the results of reconstruction in CGA,
and $h = h(x)$ is a piecewise-constant mesh function
representing the local diameter of the elements and is defined as ( see details in \cite{KN})
\begin{equation}\label{meshfunction}
  h |_K = h_K ~~~ \forall K \in K_h.
\end{equation}

The mesh refinements recommendations \eqref{meshref} are based on a posteriori error estimates for the
errors $| \varepsilon - {\varepsilon}_h^M|$, $| \sigma -
 \sigma_h^M|$ in the reconstructed functions $\varepsilon_h^M, \sigma_h^M$,
respectively. The proofs of  a posteriori error estimates
 can be derived using technique of \cite{BOOK}
and
is
topic of the ongoing research.
Algorithm 2 summarizes an  adaptive conjugate
gradient algorithm in the discrete setting.

\begin{algorithm}[hbt!]
  \centering
  \caption{ Adaptive Conjugate Gradient Algorithm (ACGA): For mesh
    refinements $k = 0, \dots, N$ perform the steps below.  }
    \begin{algorithmic}[1]

\STATE Choose initial
            spatial mesh ${K_h}_0$ in $\OFEM$.

      \STATE Compute ${\varepsilon_h}_k$, $ {\sigma_h}_k$ on mesh ${K_h}_k$ according to the CGA algorithm.

      \STATE Refine locally such elements in the finite element mesh ${K_h}_k$  where
            \begin{equation*}
              \begin{split}
                |h {\varepsilon_h}_k|
                \geq {\widetilde{\beta}_{{\varepsilon}_k}}  \max_{K  \in {K_h}_k} |h {\varepsilon_h}_k|, \\
                |h {\sigma_h}_k| \geq {\widetilde{\beta}_{{\sigma}_k}}  \max_{K  \in {K_h}_k}  |h {\sigma_h}_k|.
              \end{split}
            \end{equation*}
            Here, $ {\widetilde{\beta}_{{\varepsilon}_k}} \in (0,1)  , {\widetilde{\beta}_{{\sigma}_k}} \in (0,1)$
            are constants chosen by the user.

      \STATE Construct the new mesh ${K_h}_{k+1}$ and interpolate ${\varepsilon_h}_k$, ${\sigma_h}_k$ as well as measurements $E_{\rm obs}$ onto it.

      \STATE Terminate the algorithm if either $\| {\varepsilon_h}_{k+1} - {\varepsilon_h}_k \| < \theta_\varepsilon^1$,
      or $\| {\sigma_h}_{k+1} - {\sigma_h}_k \| < \theta_\sigma^1$, or $\|g_\varepsilon^m\| < \theta_\varepsilon^2$,
      or $\|g_\sigma^m\| < \theta_\sigma^2$, ($m$ is the iteration in CGA), where $\theta^1_\varepsilon$, $\theta^2_\varepsilon$, $\theta^1_\sigma$ and $\theta^2_\sigma$ are tolerances chosen by the user. 
	    Otherwise, set $k:= k+1$ and repeat the algorithm from step 2.

 \end{algorithmic}
\end{algorithm}

Note that we perform   mesh refinements only in space and not in
time to be able smoothly run the domain decomposition finite element/finite difference method
for solutions of the forward and adjoint problems - see details of
this method for our model problem in \cite{BL1}.

\section{Numerical examples}

\label{sec:numex}

In this section we present two- and three-dimensional numerical tests
for the reconstruction of both the dielectric permittivity and
conductivity functions, which are used to assess the performance of
the optimization-based CGA and ACGA algorithms introduced in
Section~\ref{sec:algorithms}.

All two-dimensional experiments are carried out using the CGA
described in Section~\ref{sec:algorithms}, with the corresponding code
developed and tested in MATLAB. The three-dimensional experiments
employ the ACGA from Section~\ref{sec:algorithms} and are implemented
in C++ using the PETSc-based \cite{petsc} WavES library \cite{waves}. Visualization
of the three-dimensional results is performed using ParaView and GiD.

\subsection{2D tests}

The aim of this section is to present the results of applying our theory numerically in a 2D setting. 
This was done in two different tests. Before we get into the details, we specify the two-dimensional domain, which is considered to be
$\Omega = \{ (x,y): x \in (0,1), y \in (0,1) \}$
in space, and $\OT := \Omega\times(0,1.2]$ in space  and time. Our exact coefficients $\varepsilon$ and $\sigma$ that we aim to reconstruct are defined as 
\begin{align*}
    \varepsilon(x,y) &:= 1 + 3e^{-((x-0.5)^2+(y-0.7)^2)/0.002}, \\  
    \sigma(x,y) &:= 1 + \frac{3}{2}e^{-((x-0.5)^2+(y-0.7)^2)/0.002}.
\end{align*}

These constructions are somewhat realistic and inspired by studies \cite{BR2}, while also being in line with assumptions \eqref{coeffs}.
Figure  \ref{fig:badrec}-c) and Figure \ref{fig:goodrec}-c) show spatial distribution of exact functions  $\varepsilon$ and $\sigma$, correspondingly.
These coefficients were then reconstructed by choosing some initial guess for each, $\varepsilon^0$ and $\sigma^0$, and then applying conjugate gradient  Algorithm 1 of section \ref{sec:algorithms}.

In the two-dimensional numerical tests, we also use an approximation of system (\ref{system6eq1}--c), a damped wave system. Recently in
 \cite{RomanovKlib} it was shown that in the computational setting using plane waves Maxwell's equations can be approximated by the acoustic wave equation.
Since
in our two-dimensional numerical simulations  
we are initializing a plane wave only for the one component of the electric field  then
accordingly to \cite{RomanovKlib}  the Maxwell' system can be approximated by the  damped wave equation
  in the case of the conductive media.
 Before we introduce it, we define the boundaries $\Gamma_i, ~i = 1,2,3,4$, as the left, lower, right and upper sides of the unit square. Moreover, we let $\Gamma_{i,T} := \Gamma_i \times [0,1.2]$, and specifically for the left side $\Gamma_{1,t} = 
 \Gamma_1 \times [0,\frac{\pi}{10}]$. Now, we define the system used for numerical examples as

\begin{subnumcases}{}
    \varepsilon \ptt E + \sigma \pt E - \Delta E = 0  & for $ (x,y, t) \in \OT $, \label{system7eq1}\\
    E(x,y, 0) = 0, ~ \pt E(x,y, 0) = 0 & for $(x,y) \in \Omega$, \label{initcond7} \\
    \begin{aligned}
        &\pn E = \sin (20 t) && \text{for } (x,y,t) \in \Gamma_{1,t}, \\
        &\pn E = - \pt E && \text{for } (x,y,t) \in \Gamma_3 \cup \Gamma_{1,T} \setminus \Gamma_{1,t}, \\
        &\pn E = 0  && \text{for } (x,y,t) \in \Gamma_2 \cup \Gamma_4, \\
    \end{aligned}\label{boundcond7}
\end{subnumcases}

The system above represents introducing one period of a planar wave, $\sin(20t)$, from the boundary $\Gamma_1$, and after that make use of absorbing boundary conditions, while $\Gamma_2$ and $\Gamma_4$ are governed by homogeneuos Neumann conditions.

Before getting into the results, we will define  errors  which we have  studied  computationally :  relative $L^2$ norms and supremum, $\infty$ norms. This means that, for example, the relative $L^2$ error for $\varepsilon$ and $\sigma$ are calculated as 
\begin{equation}
  e_{\varepsilon}^m : =   \frac{\| \varepsilon^m - \varepsilon \|}{\|\varepsilon\|}, \quad  e_{\sigma}^m : =
  \frac{\|\sigma^m - \sigma \|}{\| \sigma \|},
\end{equation}
where $\varepsilon^m$ and $\sigma^m$
 are the reconstructed coefficients after $m$ iterations of conjugate gradient  Algorithm 1. The supremum errors are calculated in the same fashion but with the supremum norm instead of the $L^2$ norm. 
The errors $e^m_E$ for the discrepancy term $E- \tilde{E}_{\rm obs} $ were calculated as 
\begin{equation}
    e^m_E := \frac{\|E^m - \tilde{E}_{\rm obs} \|}{\| E^m \|}, 
\end{equation}
where $E^m := E (\varepsilon^m, \sigma^m)$, $E := E(\varepsilon,\sigma)$ and $ \tilde{E}_{\rm obs}    = E + \delta$ where $\delta \sim N(0,0.1)$. As for the coefficients, the supremum errors are calculated the same way, just with the supremum norm. 

The following results come from two different tests, using two different sets of initial guesses in the settings described above. 

\begin{figure}[h]
    \centering
    \includegraphics[trim = 3.5cm 2cm 1.4cm .9cm, width=.7\linewidth,clip= ]{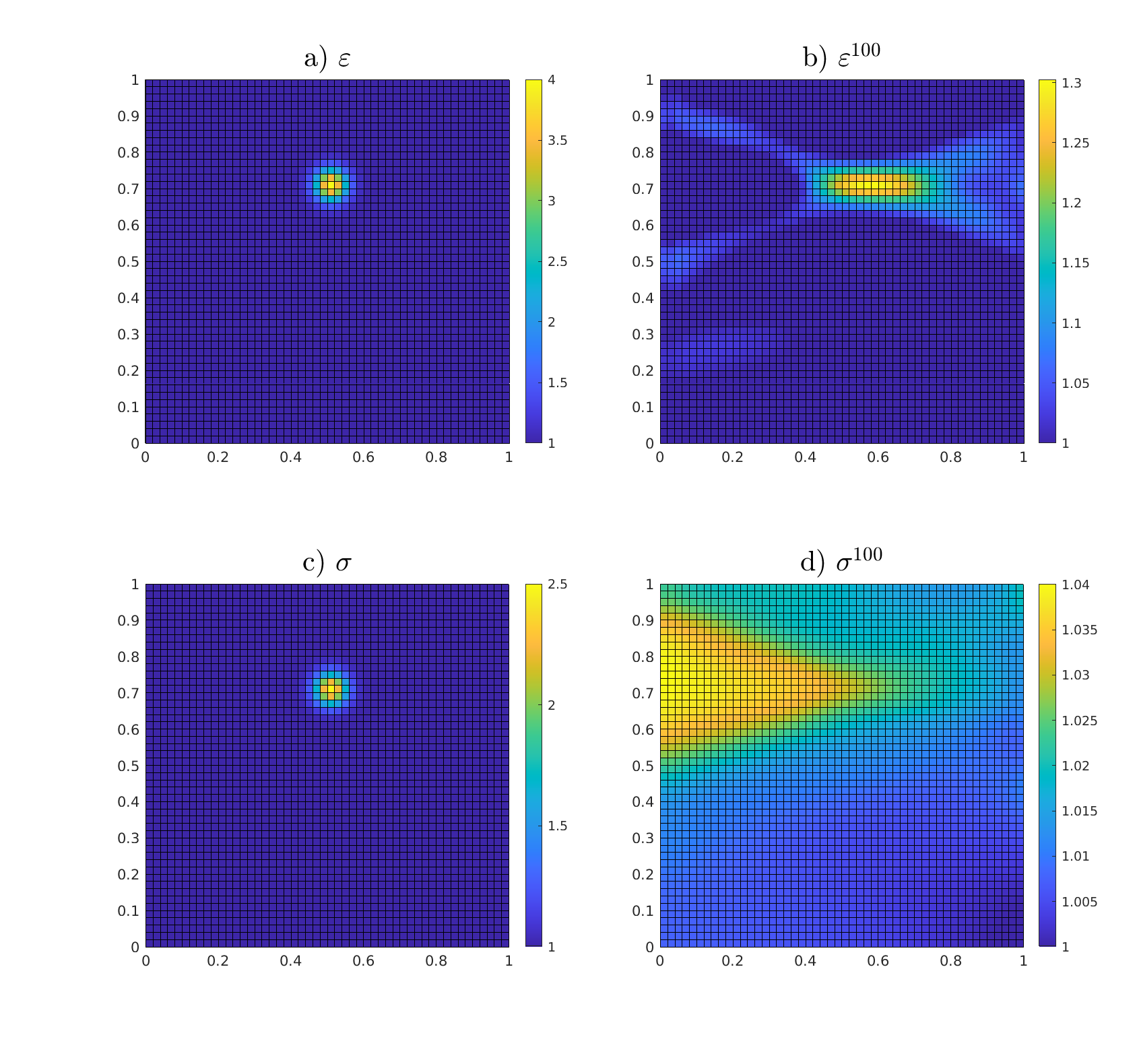}
    \caption{\textit{ 2D tests, test 1. Results from CGA, Algorithm 1. Figures a) and b) are connected to $\varepsilon$:
  a) exact $\varepsilon$,  b)  reconstructed $\varepsilon^{100}$ after 100 iterations in CGA.
Figures
     c) and d) are connected to $\sigma$: c) exact $\sigma$,   d) reconstructed  $\sigma^{100}$ after 100 iterations in CGA.}}
    \label{fig:badrec}
\end{figure}

\textbf{Test 1:} In the first test, we use the initial guesses $\varepsilon^0 (x,y) := 1$ and $\sigma^0(x,y) := 1$, representing the values of background medium. The reconstructions $\varepsilon^{100}$ and $\sigma^{100}$, made after 100 conjugate gradient steps can be seen in \autoref{fig:badrec}. One can immediately note that the reconstruction of $\varepsilon$ is quite localized in the regions we expect, and while the $\sigma$ reconstruction does not seem to be correctly localized in the $x$-direction, it does seem to be localized in the $y$-direction.

Analyzing the results further in \autoref{fig:baderr}, we observe that the $L^2$ errors and supremum errors decay consistently for  both coefficients $\varepsilon$, $\sigma$, as well as for   $E - \tilde{E}_{\rm obs} $. We note the exception of the supremum norm of $E - \tilde{E}_{\rm obs} $,  and a slight upwards trend at latter iterations in the $L^2$ error of $\sigma$. However, these deviations could be the result of reaching a local minima, since we are using conjugate gradient method.
These disturbances disappear as an initial guesses  $\varepsilon^0, \sigma^0$ becomes closer to the exact values of coefficients  $\varepsilon, \sigma$. We will show this in the Test 2.

\begin{figure}[h!]
    \centering
    \includegraphics[trim = 5.7cm 2.1cm 4.5cm 1.3cm, width=\linewidth,clip=]{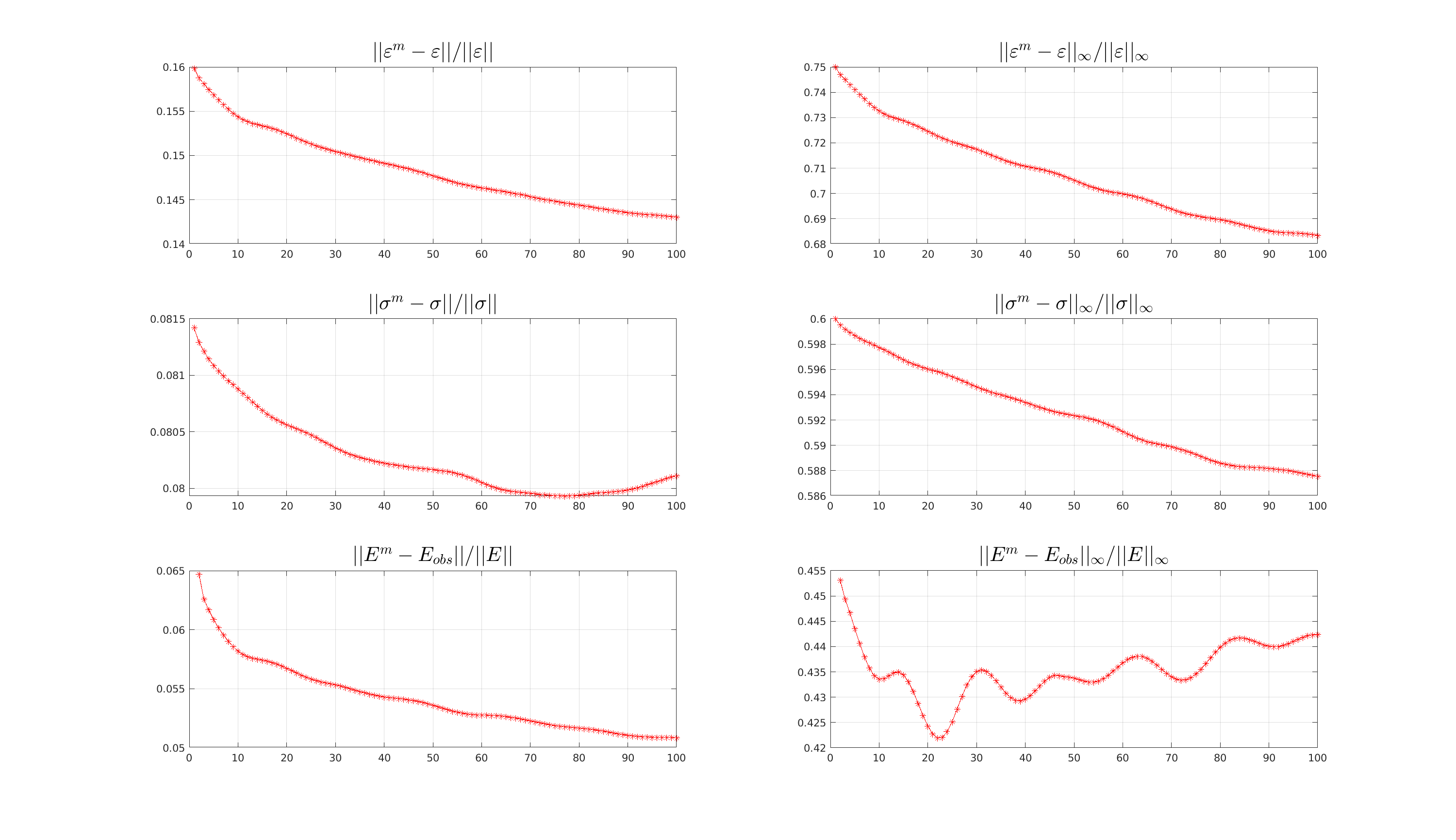}
    \put(-268,147){\small a)}
    \put(-78,147){\small b)}
    \put(-268,72){\small c)}
    \put(-78,72){\small d)}
    \put(-268,-2){\small e)}
    \put(-78,-2){\small f)}
    \caption{\textit{ 2D tests, test 1.
Convergence plots of errors over CGA iterations:
  a), b) errors $e^m_\varepsilon$ concerning $\varepsilon$;
     c), d) errors $e^m_\sigma$ concerning $\sigma$;  e), f) errors $e^m_E$ concerning the term $E- \tilde{E}_{\rm obs} $.
     We note that figures a), c) and e) show the relative $L^2$-errors, and figures  b), d) and f) show the supremum norms.}}
    \label{fig:baderr}
\end{figure}

\begin{figure}[h!]
    \centering
    \includegraphics[trim = 6cm 2.7cm 4.5cm 1.4cm, width=.7\linewidth,clip=]{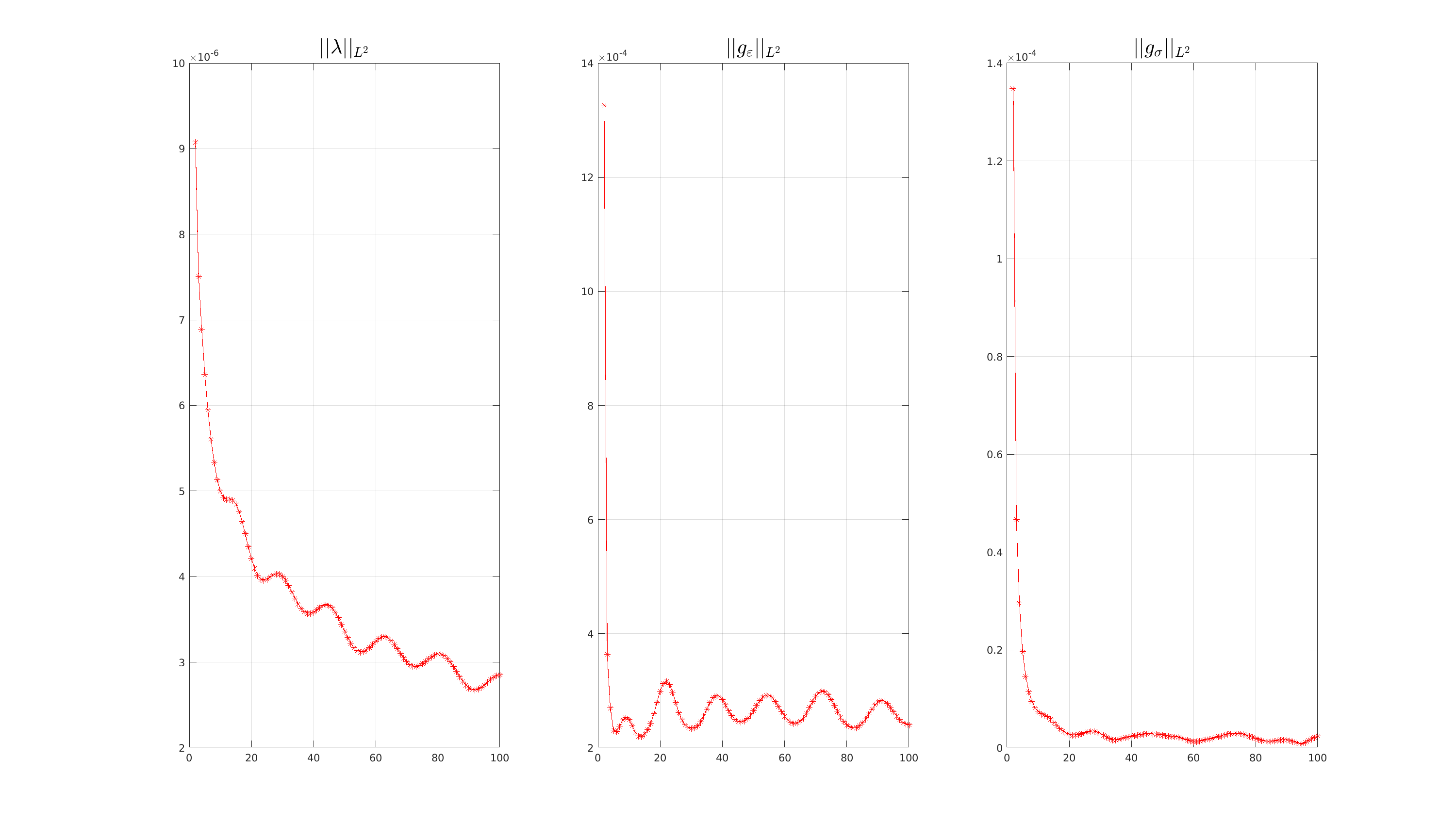}
    \put(-209,-8){a)}
    \put(-123,-8){b)}
    \put(-38,-8){c)}
    \caption{\textit{2D tests,  test 1.
Convergence plots of $L^2$ norms over CGA iterations:
a) $\| \lambda \|$; b)  $\|g_\varepsilon \|$;  c) $\| g_\sigma\|$.}}
    \label{fig:badnorms}
\end{figure}

Finally, we can affirm the claims we made, by looking  further at the norms of the adjoint   problem solution  $\| \lambda\|$,  and   norms of the  gradients $\| g_\varepsilon\|$ and $\| g_\sigma\|$ shown in \autoref{fig:badnorms}. We observe that the norm of the adjoint problem solution and the gradients decay, which is to be expected from Theorem 4.2 as well as  from the conjugate gradient Algorithm 1.

\begin{figure}[h!]
    \centering
    \includegraphics[trim = 3.5cm 2cm 1.4cm .9cm, width=.7\linewidth,clip= ]{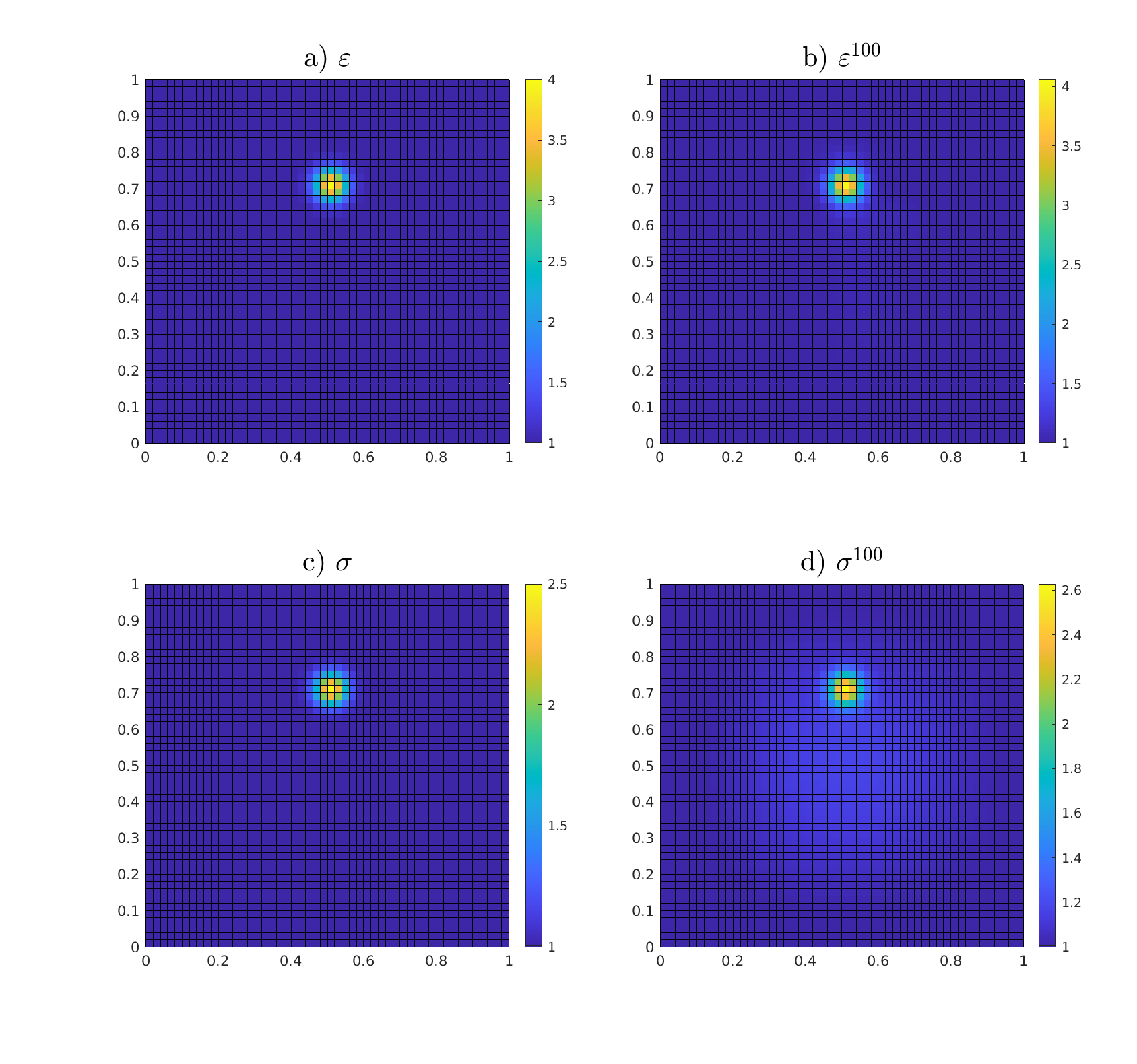}
    \caption{\textit{ 2D tests, test 2.
 Results from CGA, Algorithm 1. Figures a) and b) are connected to $\varepsilon$:
  a) exact $\varepsilon$,  b)  reconstructed $\varepsilon^{100}$ after 100 iterations in CGA.
Figures
     c) and d) are connected to $\sigma$: c) exact $\sigma$,   d) reconstructed  $\sigma^{100}$ after 100 iterations in CGA.}}
    \label{fig:goodrec}
\end{figure}

\newpage

\textbf{Test 2:} The second test was made the same way as the first, just with new initial guesses. These initial guesses were informed by the exact coefficients, but slightly perturbed, to see how the algorithms performed when initiated close to the global optima. The initial guesses were defined as
\begin{align*}
    &\varepsilon^0 (x,y) := \varepsilon(x,y) + 20\|\varepsilon\|_\infty x^2y^2(1-x)^2(1-y)^2, \\
    &\sigma^0(x,y) := \sigma(x,y) + 20\|\sigma\|_\infty x^2y^2(1-x)^2(1-y)^2,
\end{align*}
and the results  of reconstruction can be seen in \autoref{fig:goodrec}. We observe that after initiating the  Algorithm 1 close to the true values of the coefficients it does not deviate from the true solution. 
This conclusion is further strengthened in \autoref{fig:gooderr}, where we now observe  how all errors decay consistently without exceptions as seen in the Test 1. This indicates that the  conjugate gradient  Algorithm 1 really converging towards the true value  of parameters $\varepsilon, \sigma$ and not to some local minimum.

\begin{figure}[h]
    \centering
    \includegraphics[trim = 5.7cm 2.1cm 4.5cm 1.3cm, width=\linewidth,clip=]{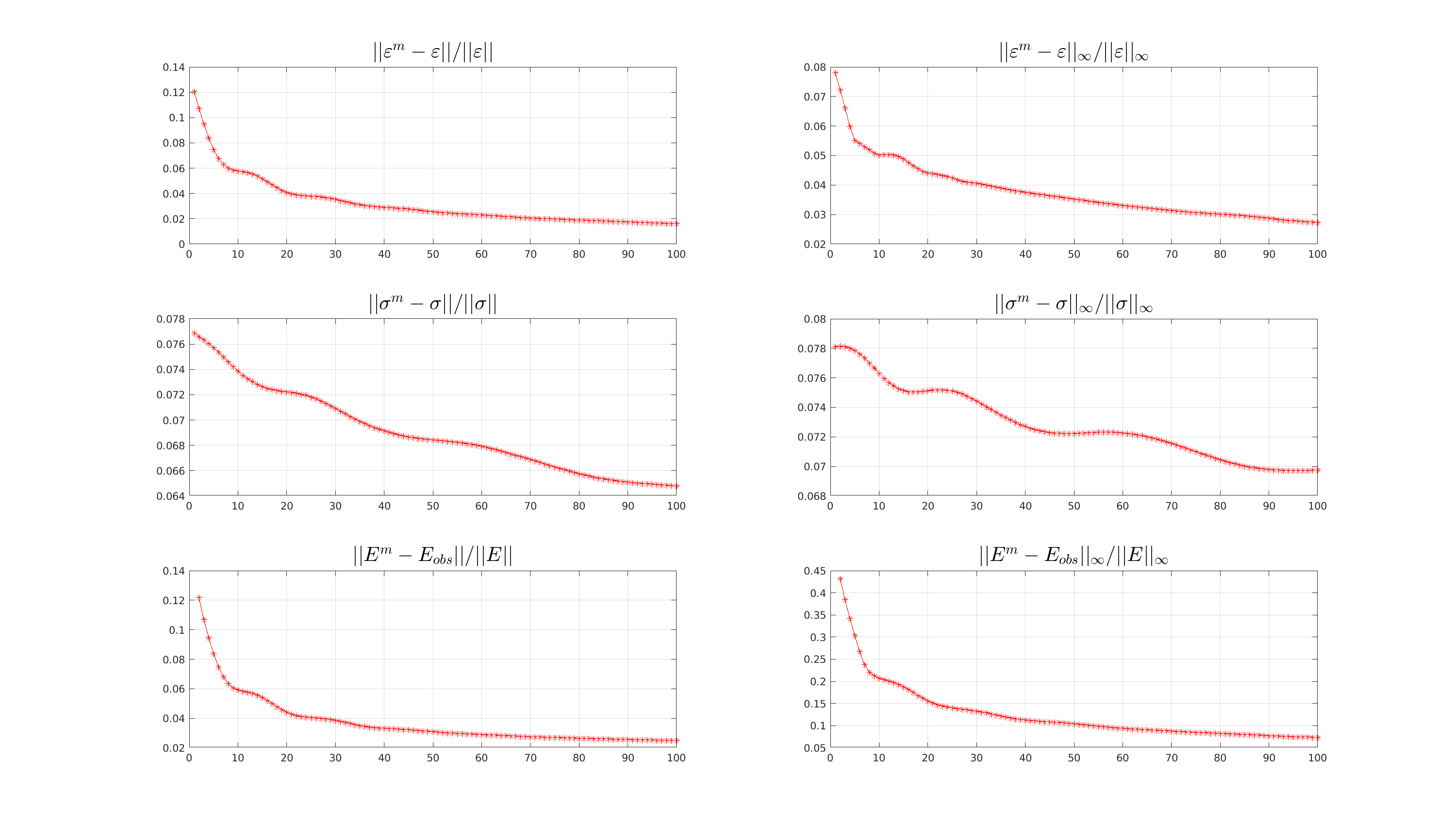}
    \put(-268,147){\small a)}
    \put(-78,147){\small b)}
    \put(-268,72){\small c)}
    \put(-78,72){\small d)}
    \put(-268,-2){\small e)}
    \put(-78,-2){\small f)}
    \caption{\textit{ 2D tests, test 2.
Convergence plots of errors over CGA iterations:
  a), b) errors $e^m_\varepsilon$ concerning $\varepsilon$;
     c), d) errors $e^m_\sigma$ concerning $\sigma$;  e), f) errors $e^m_E$ concerning the term $E- \tilde{E}_{\rm obs} $.
     We note that figures a), c) and e) show the relative $L^2$-errors, and figures  b), d) and f) show the supremum norms.}}
    \label{fig:gooderr}
\end{figure}

\begin{figure}[H]
    \centering
    \includegraphics[trim = 6cm 2.7cm 4.5cm 1.4cm, width=.7\linewidth,clip=]{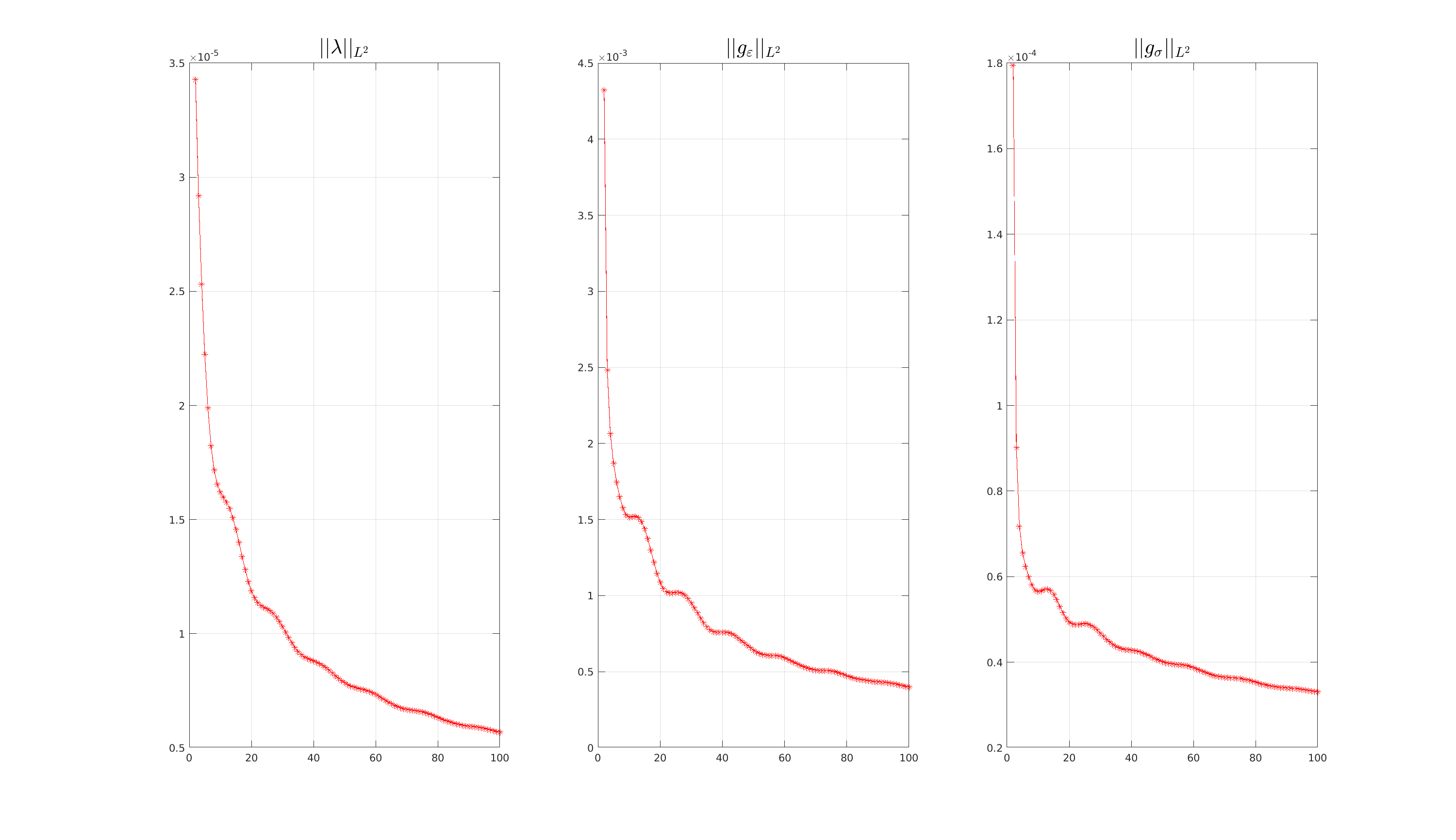}
    \put(-209,-8){a)}
    \put(-123,-8){b)}
    \put(-38,-8){c)}
    \caption{\textit{2D tests, test 2.
Convergence plots of $L^2$ norms over   CGA iterations:~
a) $\| \lambda \|$;~ b)  $\|g_\varepsilon \|$; ~ c) $\| g_\sigma\|$.}}
    \label{fig:goodnorms}
\end{figure}

Finally, in \autoref{fig:goodnorms} we see the same behavior as before, where the norms of $\lambda$, $g_\varepsilon$ and $g_\sigma$ tend towards zero  or to stabilization  what is in accordance with   stopping criterion in the Algorithm 1.

\subsection{3D tests}

In this section we discuss three dimensional numerical results of
 reconstruction of both functions, $\varepsilon$ and $\sigma$,  in the model problem 
\eqref{model1}
using
 conjugate gradient algorithm (CGA) and adaptive conjugate gradient
 algorithm (ACGA) presented in  section \ref{sec:algorithms}.
ACGA is based on the local adaptive finite element mesh refinements
 utilizing a posteriori error indicators for the reconstruction of functions
 $\varepsilon, \sigma$.
  We refer to \cite{BL1,ICEAA2024_LB,BOOK}  for
  more details about this method.
  
For solutions
 of the forward and adjoint problems was used adaptive domain
 decomposition finite element/finite difference method (ADDFE/FDM) developed in
 \cite{BL1, BL2} and implemented in the software package WavES \cite{waves}.
 We refer to \cite{BL1,BL2} for details about numerical implementation
 of the ADDFE/FDM method for time-dependent Maxwell's system in conductive media.

\begin{figure}[tbp]
  \begin{center}
    \begin{tabular}{cc}
      \includegraphics[trim = 1.0cm 0.0cm 1.0cm 0.0cm, scale=0.16, clip=]{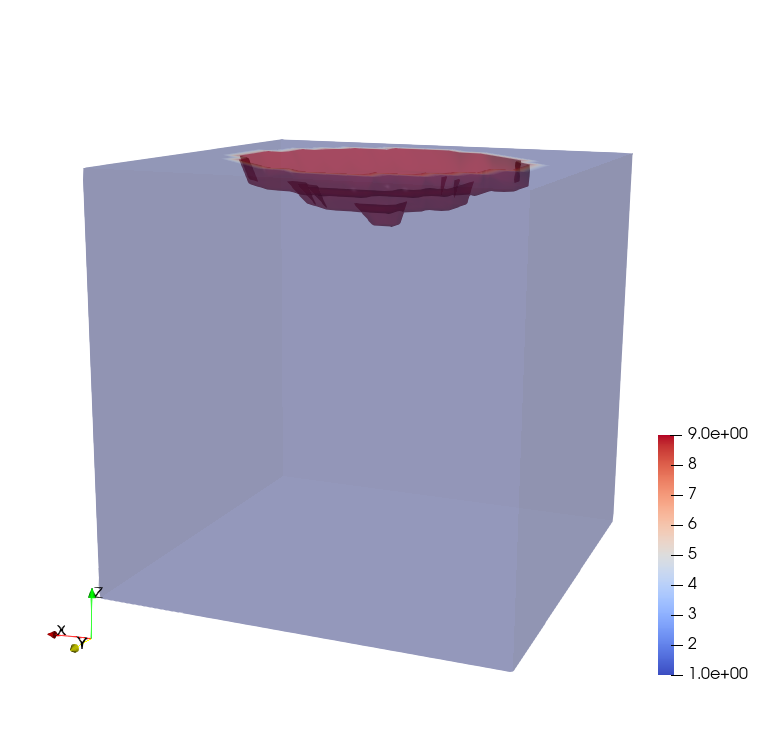} &
      \includegraphics[trim = 1.0cm 0.0cm 1.0cm 0.0cm, scale=0.16, clip=]{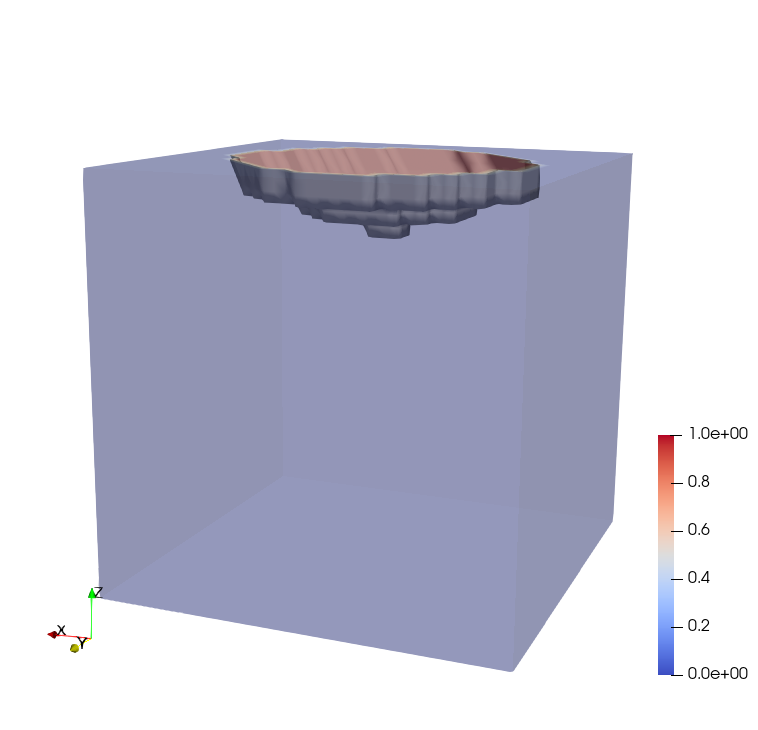}
      \\
      a)  $  \max  \varepsilon_w  \approx 9$        & b)   $ \max \sigma_w \approx 1.2$ \\
      \includegraphics[trim = 8.0cm 4.0cm 4.0cm 4.0cm, scale=0.16, clip=]{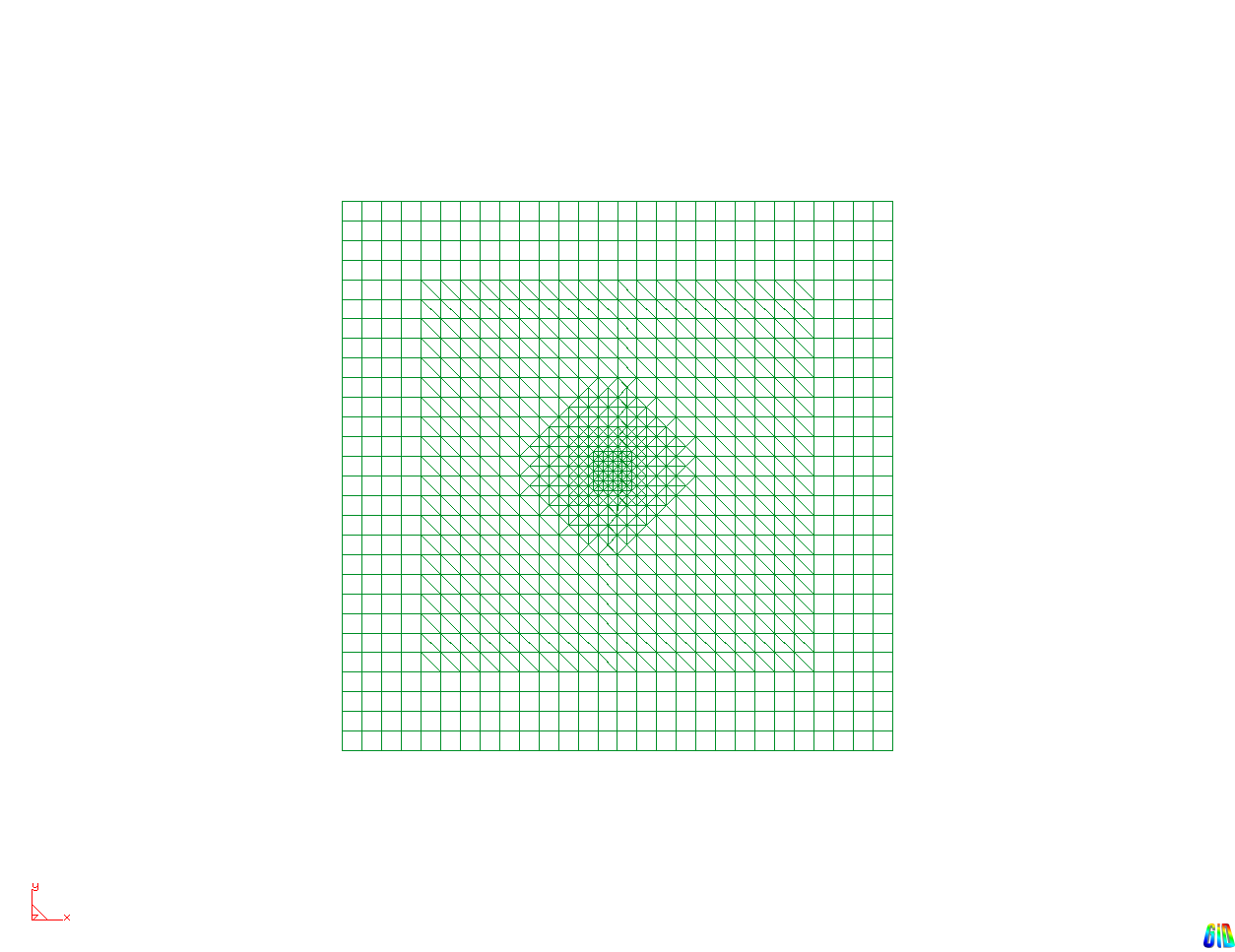}   &
      \includegraphics[trim = 8.0cm 4.0cm 4.0cm 4.0cm, scale=0.16, clip=]{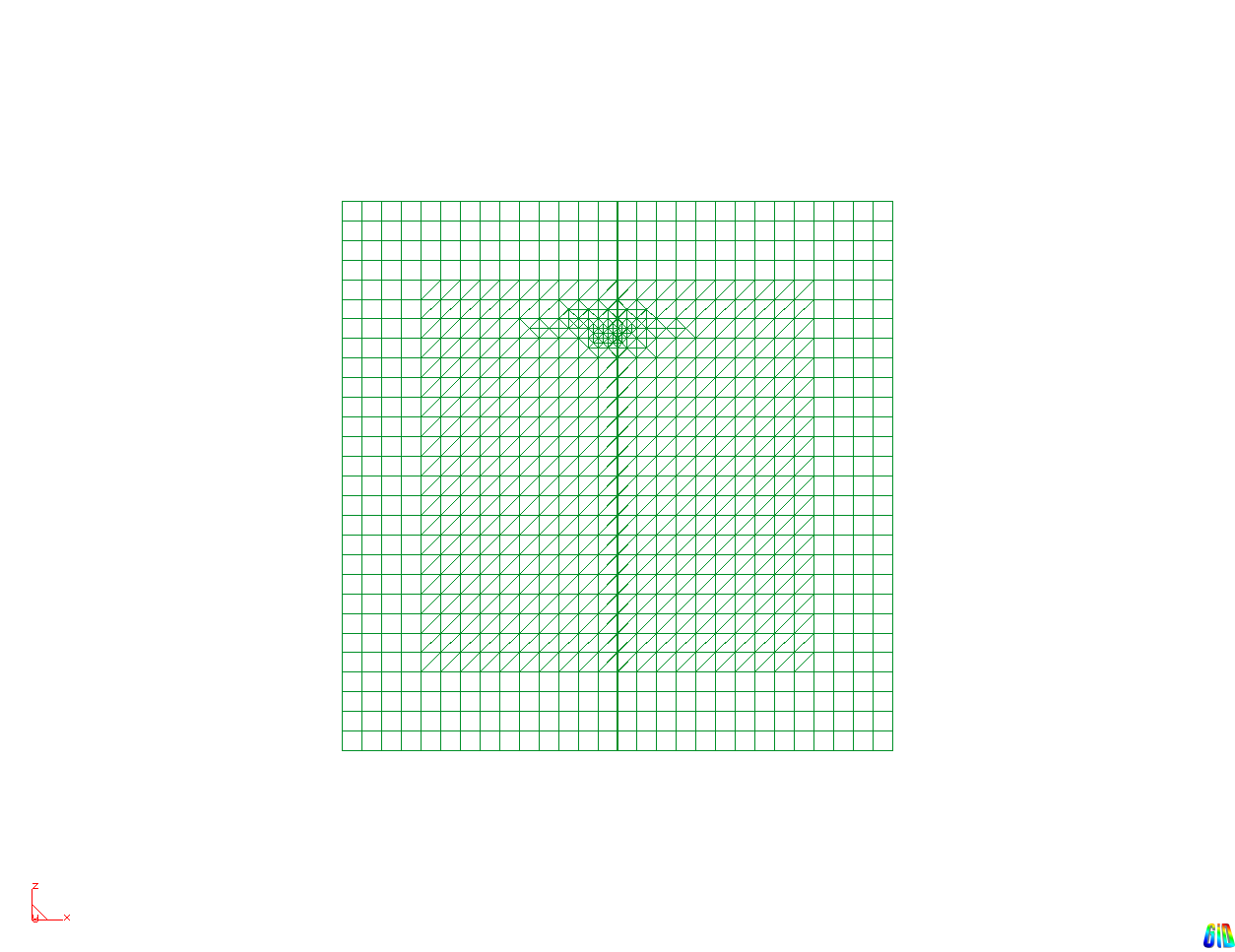}   \\
      c)  $x_1 x_2$-view    & d) $x_3 x_1$- view
    \end{tabular}
  \end{center}
  \caption{\small 3D tests, test 1.
    Dielectric weighted properties  taken in our computations: a) $\varepsilon_w$, b) $\sigma_w$.
    c), d) Projections of the locally refined hybrid FE/FD mesh taken for generation of data.
  }
  \label{fig:numex1}
\end{figure}

\begin{figure}[h!]
  \begin{center}
    \begin{tabular}{cc}
      \includegraphics[trim = 1.0cm 0.0cm 1.0cm 0.0cm, scale=0.16, clip=]{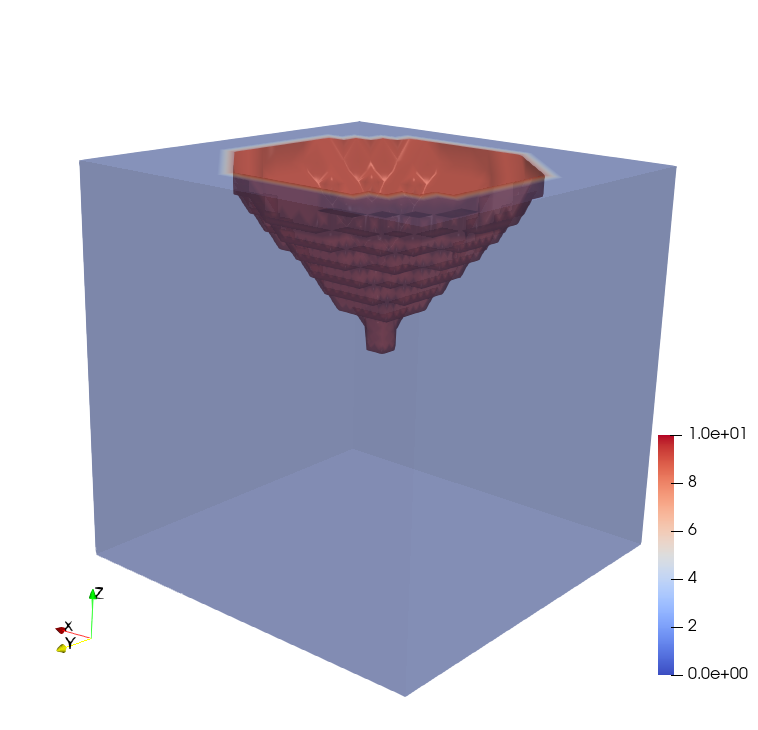} &
      \includegraphics[trim = 1.0cm 0.0cm 1.0cm 0.0cm, scale=0.16, clip=]{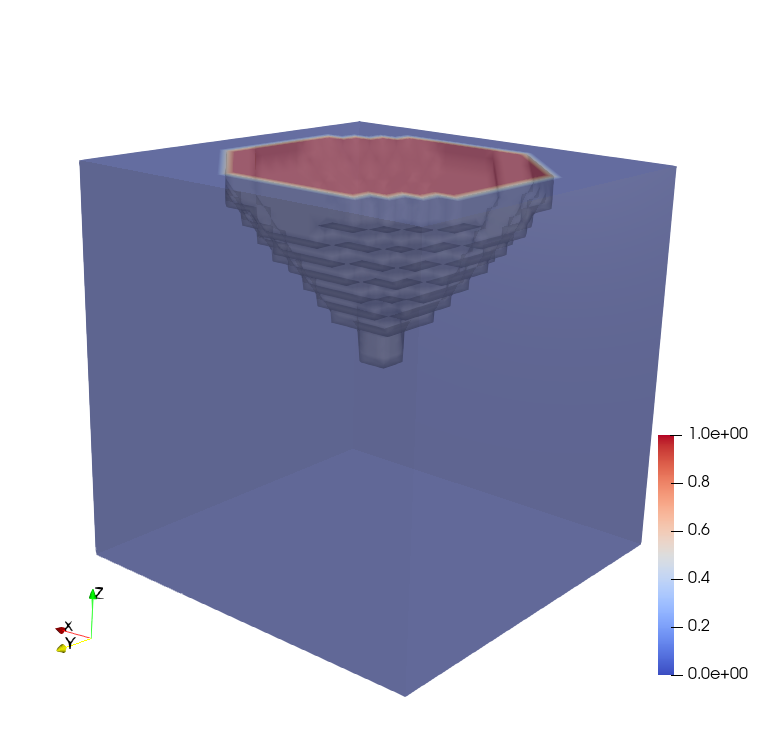}
      \\
      a)  $  \max  \varepsilon_w  \approx 9$        & b)   $ \max \sigma_w \approx 1.2$ \\
      \includegraphics[trim = 8.0cm 4.0cm 4.0cm 4.0cm, scale=0.16, clip=]{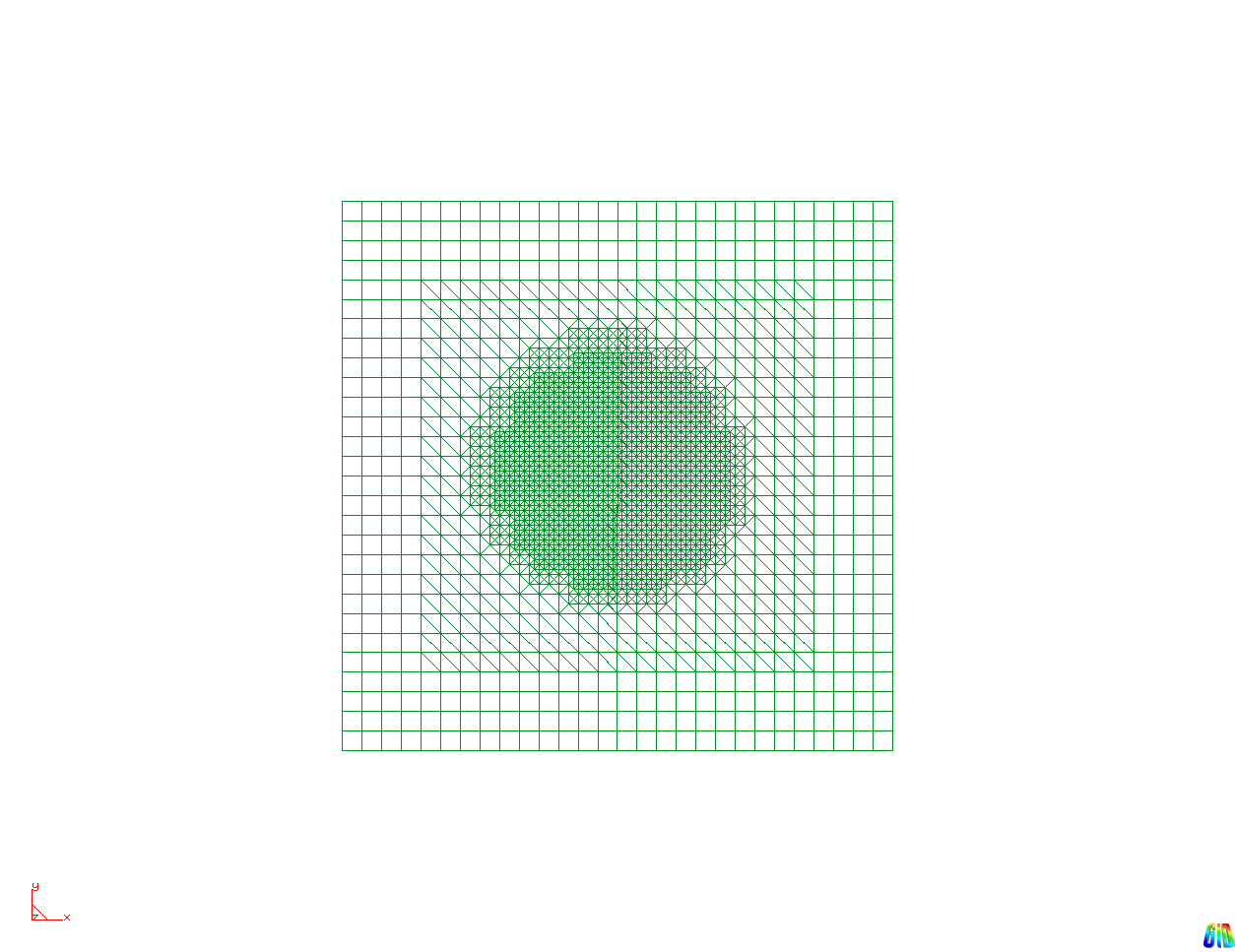}   &
      \includegraphics[trim = 8.0cm 4.0cm 4.0cm 4.0cm, scale=0.16, clip=]{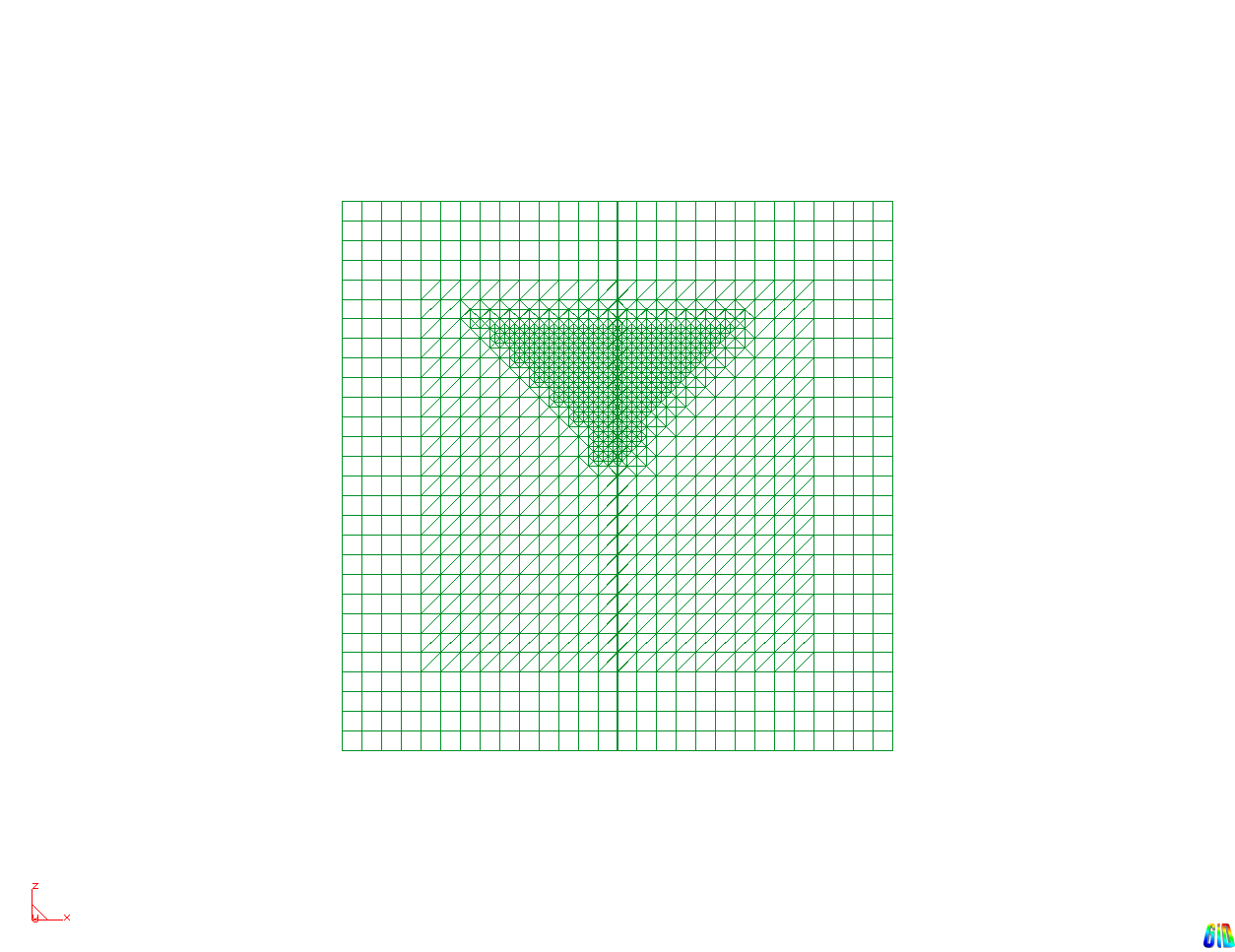}   \\
      c)  $x_1 x_2$-view    & d) $x_3 x_1$- view
    \end{tabular}
  \end{center}
  \caption{\small 3D tests, test 2.
    Dielectric weighted properties  taken in our computations: a) $\varepsilon_w$, b) $\sigma_w$.
    c), d) Projections of the locally refined hybrid FE/FD mesh taken for generation of data.
  }
  \label{fig:numex2}
\end{figure}

In order to apply ADDFE/FDM in our computations
we
split the computational domain $\Omega$ into two domains:  $\Omega_{\rm FEM}$ where we apply the finite element method, and
$ \Omega_{\rm FDM}$ where we use the
 finite difference method,
such that
$\Omega : = \Omega_{\rm FEM} \cup \Omega_{\rm FDM}$ with
$\Omega_{\rm FEM} \subset \Omega$.
We discretize the  finite element domain $\Omega_{\rm FEM}$
 by
tetrahedral elements,  and the  finite difference
domain $\Omega_{\rm FDM}$ - by hexahedral elements - see \cite{BL1,BL2} for further
details of discretization in the ADDFE/FDM.
In our computations  we set
the dimensionless computational domain $\Omega$ as
$ \Omega = \left\{ x= (x_1,x_2, x_3) \in (-2, 12) \times (-2, 12) \times (-2,
 12) \right\}$,
and the computational domain $\Omega_{\rm FEM}$ as
$ \Omega_{\rm FEM} =
 \left\{ x= (x_1,x_2, x_3) \in (0, 10) \times (0, 10) \times (0,
 10) \right\}$.
The domain $\Omega_{\rm FEM}$  represents the 3D models of malign melanoma
 (MM) embedded  within a homogeneous domain
 which has the size $10 \times 10 \times 10$ mm. The geometrical models  of MM growth
  were
developed in
\cite{IEEE2024}  and they have realistic values of  dielectric permittivity and conductivity of  MM  at 6 GHz.
We note that  based on geometrical  models \cite{IEEE2024} in the recent work \cite{ICEAA2025_LB}
were constructed the finite element meshes   for  MM  growth  on the skin, and two of them
  we use in the current tests.
To proceed further,
  we assign values of $\varepsilon$
and $\sigma$ in $\Omega_{\rm FEM}$ accordingly to the test values of the Table
\ref{tab:table1}, and we set $\varepsilon =1$
and $\sigma = 0$ in $\Omega \setminus \Omega_{\rm FEM}$.
Figures  \ref{fig:numex1}, \ref{fig:numex2} show simplified models of MM placed in homogeneous domain   with exact values of the relative dielectric permittivity and conductivity functions given in the Table \ref{tab:table1}.
Dielectric properties shown on this figure
have weighted values of $\varepsilon$ and $\sigma$ which correspond
to the test values presented in the Table \ref{tab:table1}. Our ongoing work is on the reconstructing  non-weighted  real values of  $\varepsilon$ and $\sigma$
 presented in the Table  \ref{tab:table1}, and will be reported later.
To proceed further, we decompose the boundary $\partial \Omega$ of the domain $\Omega$ into
three parts as follows: $\partial
 \Omega =\partial _{1} \Omega \cup \partial _{2} \Omega \cup \partial
 _{3} \Omega$. Here, $\partial _{1} \Omega$ and $\partial _{2} \Omega$
are, respectively, top and bottom sides of $\Omega$, and $\partial
 _{3} \Omega$ is the union of left, right, front and back sides of this
domain. The time-dependent observations are simulated at
the backscattered boundary
$\Gamma_1 :=
 \partial_1 \Omega \times (0,T)$ of $\Omega$. We also define $\Gamma_{1,1} := \partial_1
 \Omega \times (0,t_1]$, $\Gamma_{1,2} := \partial_1 \Omega \times
 (t_1,T)$, and $\Gamma_3
 := \partial_3 \Omega \times (0, T)$.
In our computations we use the following stabilized
model problem:
\begin{equation}\label{model1}
  \begin{split}
    \varepsilon \ptt E & +
    \nabla (\nabla \cdot E) - \Delta E  -
    \nabla  (\nabla \cdot ( \varepsilon  E))
    = -  \sigma \pt E                              ~\mbox{in}~ \Omega_T,                   \\
    E(x,0)                & = 0, ~~~\pt E(x,0) = 0 ~\mbox{in}~ \Omega,                     \\
    \pn E                 & = P(t)                 ~\mbox{on}~ \Gamma_{1,1},               \\
    \pn E                 & = - \pt E              ~\mbox{on}~ \Gamma_{1,2} \cup \Gamma_2, \\
    \pn E                 & = 0                    ~\mbox{on}~ \Gamma_3.                   \\
  \end{split}
\end{equation}
We initialize a plane wave $P(t) = (0,P_2,0)(t)$ only for the one component $E_2$ of the electric field
$E=(E_1, E_2, E_3)$ at $\Gamma_{1,1}$
in (\ref{model1})
in time $t=[0,12.0]$ and define it as
\begin{equation}\label{f}
  \begin{split}
    P_2(t) =\left\{
    \begin{array}{ll}
      \sin \left( \omega t \right) ,\qquad & \text{ if }t\in \left( 0,\frac{2\pi }{\omega }
      \right) ,                                                                             \\
      0,                                   & \text{ if } t>\frac{2\pi }{\omega }.
    \end{array}
    \right.
  \end{split}
\end{equation}

\begin{table}[ht]
  \centering
  \caption{\textit{Tissue types and corresponding realistic
      values of $\varepsilon$ and $\sigma$ (S/m) at 6 Ghz for skin with malign melanoma  (MM)
      used in our numerical experiments. }}
  \begin{tabular}{|p{2.3cm}|p{0.8cm}|p{0.9cm}|p{1cm}|p{1cm}|p{1cm}|}
    \hline
\multicolumn{1}{|c|}{}  & \multicolumn{2}{c|}{Real values}  & \multicolumn{2}{c|}{Test values}  & \multicolumn{1}{c|}{Depth}
\\
    \hline
    Tissue type      & $\varepsilon$                    & $\sigma$       & $\varepsilon_w $ & $\sigma_w$ &  \\
                     &                                  & (S/m)                            &                                & (S/m)                & (mm)   \\
    \hline
    Immersion medium & 32                               & 4                                & 1                          & 0                    & 2      \\
    \hline
    Epidermis        & 35                               & 4                                & 1                          & 0                    & 1      \\
    \hline
    Dermis           & 40                               & 9                                & 1                          & 0                    & 3.5    \\
    \hline
    Fat              & 9                                & 1                                & 1                          & 0                    & 5.5    \\
    \hline
    Tumor stage 1    & 45                               & 6                                & 8                         & 1.2              & $< 1$       \\
    \hline
    Tumor stage 2    & 50                               & 6                                & 9                         & 1.2              & $ > 1$       \\
    \hline
    Tumor stage 3    & 60                               & 6                                & -                              & -                    & $> 1$    \\
    \hline
  \end{tabular}
  \label{tab:table1}
\end{table}

\begin{table}[h!]
    \centering
    \caption{\textit{Geometrical properties of the MM  models are taken from \cite{IEEE2024}.}}
    \begin{tabular}{|c c c c|}
    \hline 
        Month & Shape & Diameter (mm) & Depth (mm) \\
    \hline
    \hline
        6 & Cone & 6.28 & 1.08 \\
        22 & Cone & 6.96 & 4.06 \\
    \hline
    \end{tabular} 
    \label{tab:table2}
\end{table}

The goal of our numerical tests is to reconstruct weighted dielectric
permittivity $\varepsilon_w$ and conductivity $\sigma_w$ functions
 for MM   at month 6 and at month 22 with  dielectric  and geometrical 
 properties presented in  the Table \ref{tab:table1} and in the Table   \ref{tab:table2}, correspondingly,
 using time-dependent
backscattered electric field $E =(E_1,E_2,E_3)$ at $\Gamma_{1}$.

For generation of data 
 we compute first  components of the electric field $E =(E_1,E_2,E_3)$ at $\Gamma_{1}$ using known values of
  $\varepsilon_w$  and $\sigma_w$  given as the test values in the Table
 \ref{tab:table1}, and use them
in the model problem \eqref{model1}.
After we add the
   normally distributed Gaussian noise $\delta = 10\%$  with the mean $\mu=0$
  to the simulated electric field  $E =(E_1,E_2,E_3)$  at the backscattered boundary  $\Gamma_1$.
  Note that  in order to avoid variational crimes the data is generated on the five time locally refined
   mesh which does not coincide  with the finite element mesh  used for computations in the ACGA.
 Figure \ref{fig:numex1} - c), d) and Figure \ref{fig:numex2} - c), d)
show projections of locally refined finite element meshes used for generation of
backscattered data in Test 1 and Test 2, correspondingly.

We have chosen
 numbers $ \widetilde{\beta}_\varepsilon = 0.8 , \widetilde{\beta}_\sigma  = 0.8$
  in the mesh refinement criterion \eqref{meshref}
 in both tests.
  Relative errors presented in the Tables  \ref{tab:table1ACGA},  \ref{tab:table2ACGA}  are computed as
  follows:
\begin{equation}\label{relerrors}
 e_\varepsilon = \frac{\max_{\Omega_{\rm FEM} } |\varepsilon_w -  {{\varepsilon_w}_h}_k | }{ {\rm nno } \cdot \max_{\Omega_{\rm FEM}} |\varepsilon_w|}, ~~~
 e_\sigma =  \frac{\max_{\Omega_{\rm FEM} } |\sigma_w - {{\sigma_w}_h}_k  | }{  {\rm nno} \cdot \max_{\Omega_{\rm FEM}} |\sigma_w|},
\end{equation}
where $ {\rm nno} $  is the number of nodes in the  finite element mesh.

\begin{figure}[h!]
  \begin{center}
    \begin{tabular}{ccc}
      \includegraphics[trim = 3.0cm 0.0cm 1.5cm 4.0cm, scale=0.13, clip=]{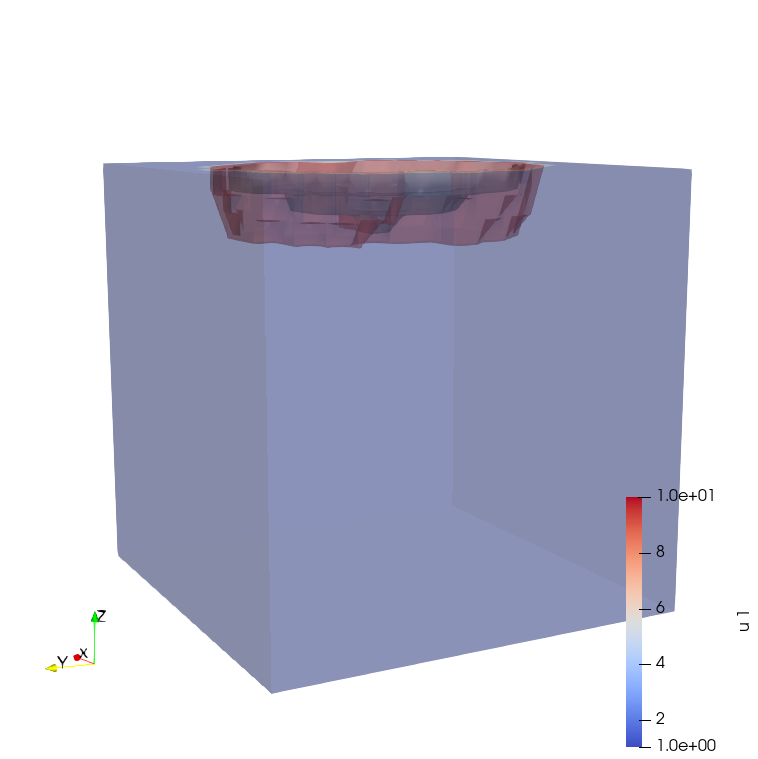} &
 \includegraphics[trim = 3.0cm 0.0cm 1.5cm 4.0cm, scale=0.13, clip=]{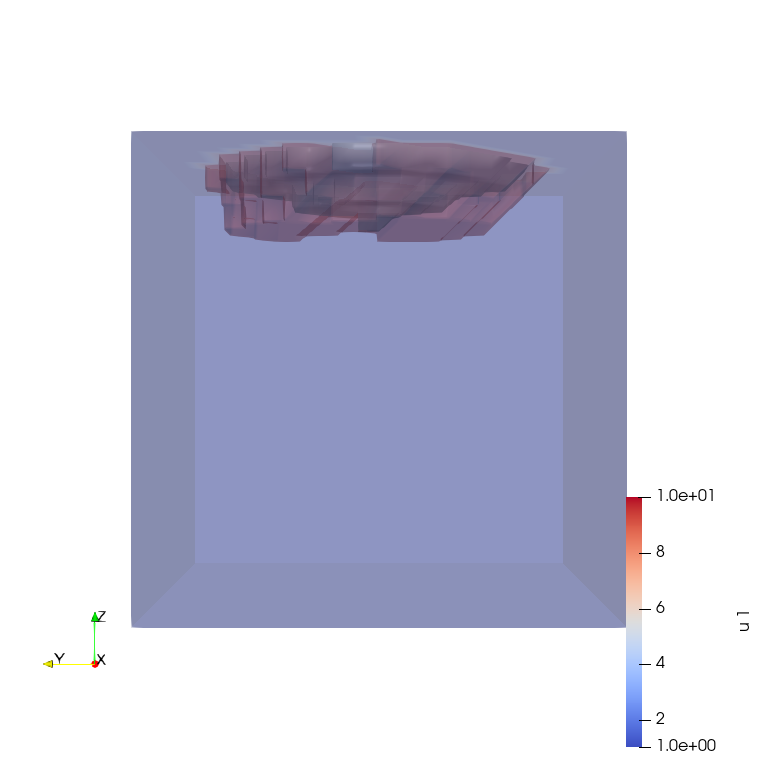} &
  \includegraphics[trim = 3.0cm 0.0cm 1.5cm 4.0cm, scale=0.13, clip=]{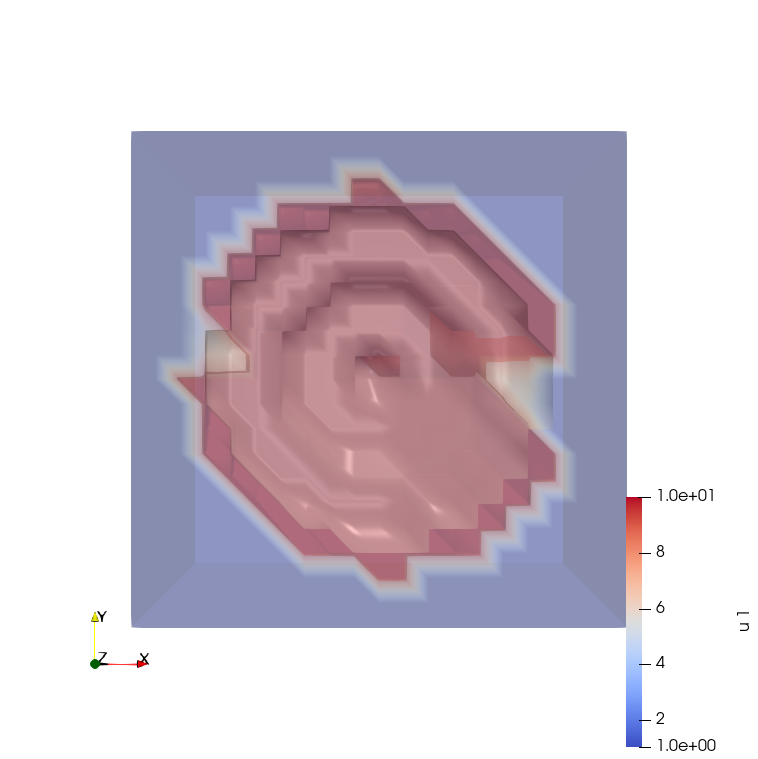} \\
    a)  prospect view   & b) $x_2 x_3$ view  &  c)  $x_1 x_2$  view  \\
      \includegraphics[trim = 3.0cm 0.0cm 1.5cm 4.0cm, scale=0.13, clip=]{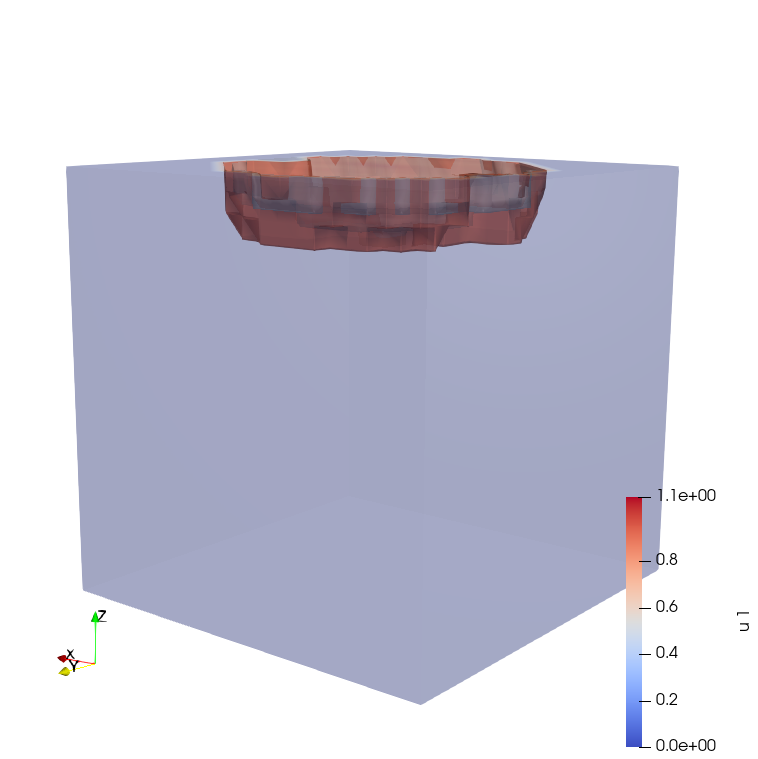}  &
   \includegraphics[trim = 3.0cm 0.0cm 1.5cm 4.0cm, scale=0.13, clip=]{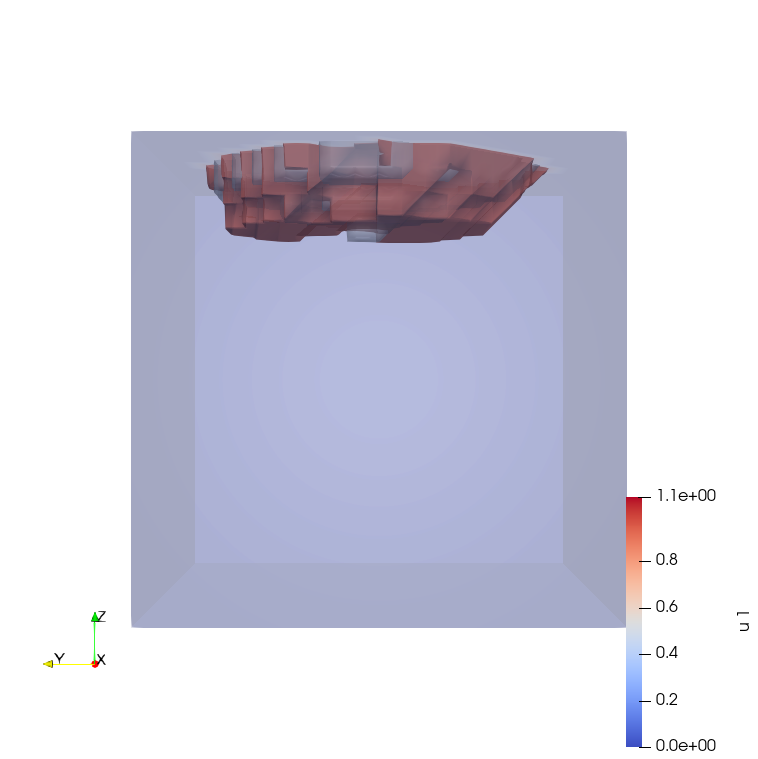}  &     
  \includegraphics[trim = 3.0cm 0.0cm 1.5cm 4.0cm, scale=0.13, clip=]{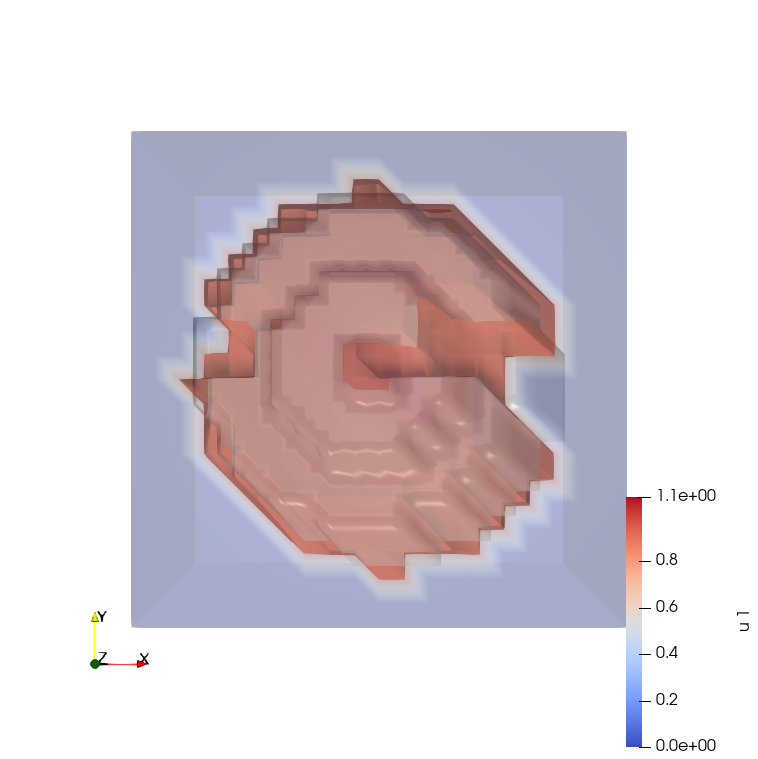}
  \\
      d)  prospect view    & e)   $x_2 x_3$ view  & f)  $x_1 x_2$  view   \\
   \includegraphics[trim = 7.0cm 0.0cm 7cm 4.0cm, scale=0.13, clip=]{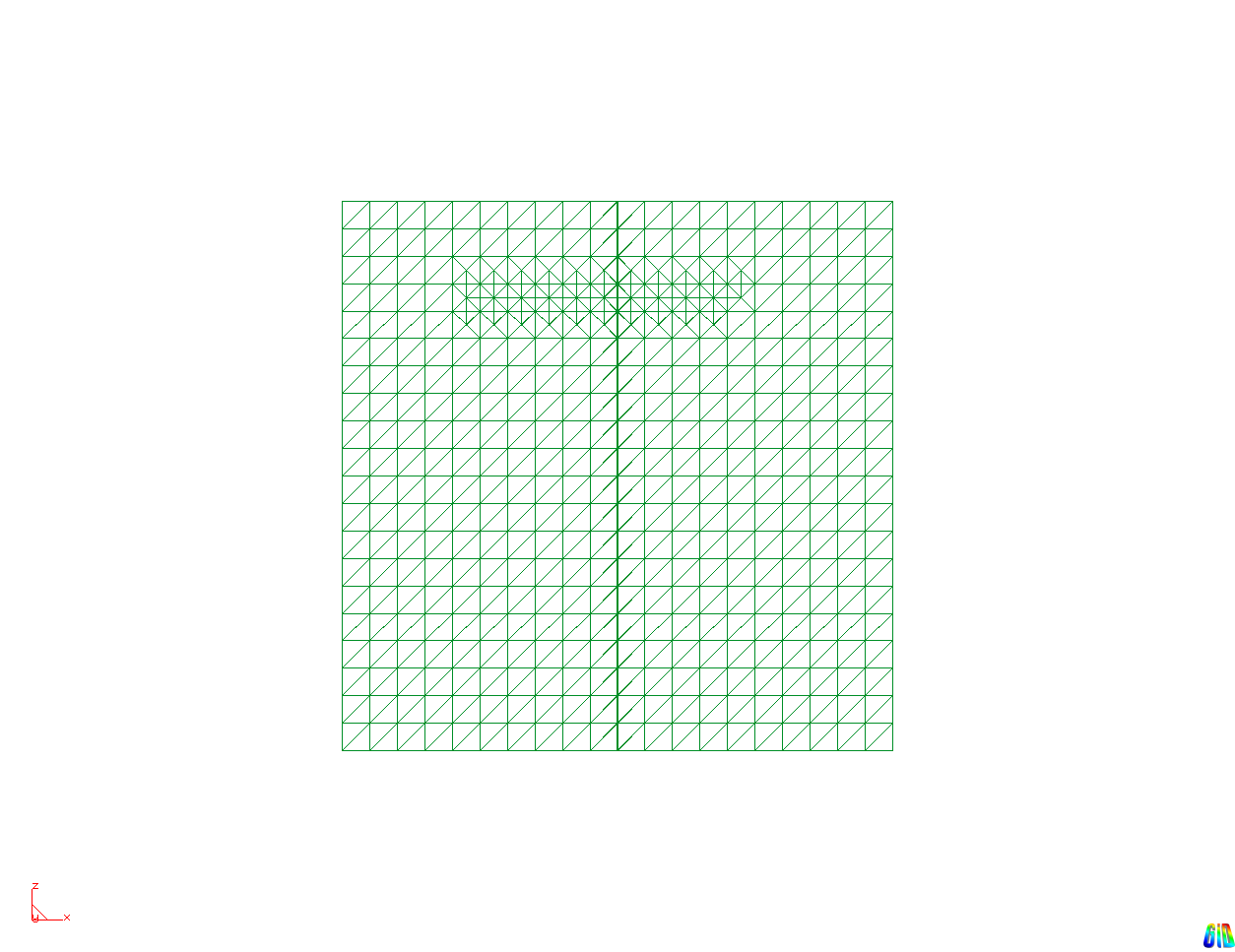}  &    
  \includegraphics[trim = 7.0cm 0.0cm 7cm 4.0cm, scale=0.13, clip=]{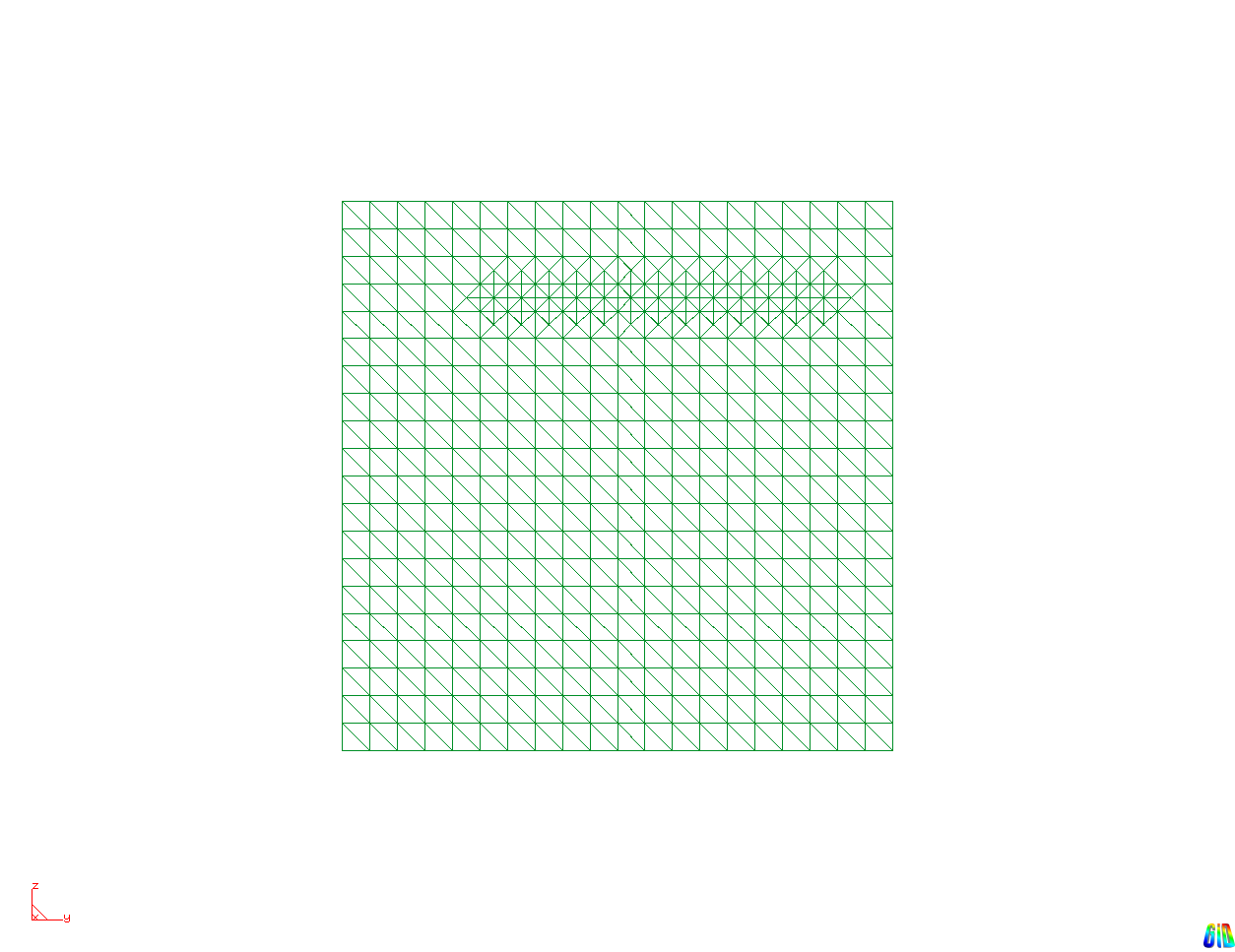}
  &
  \includegraphics[trim = 7.0cm 0.0cm 7cm 4.0cm, scale=0.13, clip=]{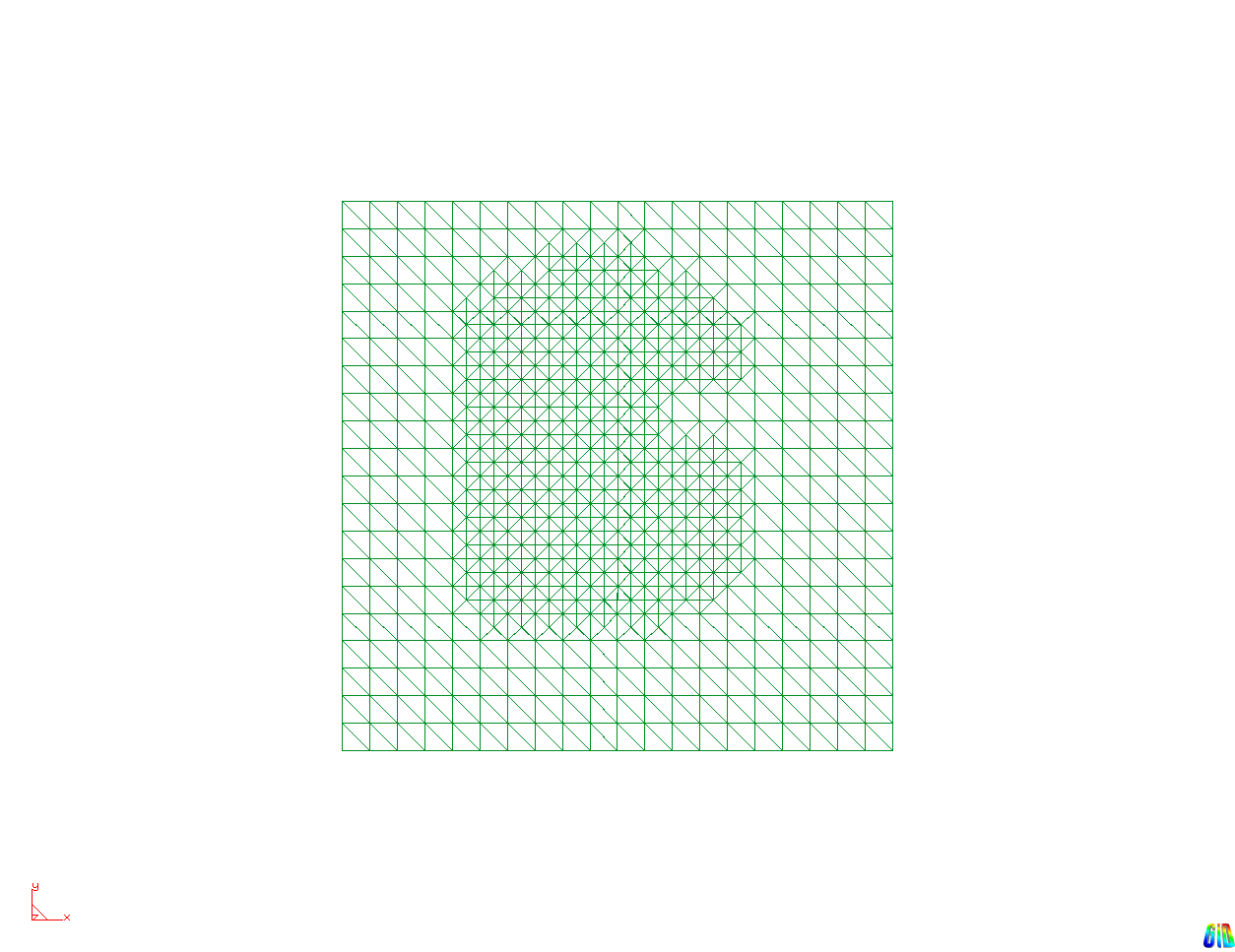}  
\\
    g) $x_1 x_3$  view   & h)  $x_2 x_3$  view   & i)   $x_1 x_2$  view    \\
    \end{tabular}
  \end{center}
  \caption{\small 3D tests, test 1. Performance of ACGA on the four times refined  mesh ${K_h}_4$. Isosurfaces of reconstruction  for 
$\max {\varepsilon_w}_h  \approx 9$     are presented in red color, and   for $\max {\sigma_w}_h \approx 0.81$  are shown also in red color.
a) - c) The weighted reconstruction of ${\varepsilon_w}_h$ (outlined in red color).
  d) - f) The weighted reconstruction of ${\sigma_w}_h$ (outlined in red color). g)-i) Projections of the adaptively refined mesh ${K_h}_4$.
    The noise level in the data for electric field is $\delta= 10\%$.  See Table  \ref{tab:table2ACGA} for obtained
    contrasts for $\max_{\Omega_{\rm FEM}} {\varepsilon_w}_h$ and $\max_{\Omega_{\rm FEM}} {\sigma_w}_h$
     on meshes ${K_h}_k, k=0,...,4$.
    For comparison we also present exact isosurfaces  for $\varepsilon_w$ and $\sigma_w$
with values corresponding to the reconstructed ones, outlined in opaque.
  }
  \label{fig:ACGATest1ref4}
\end{figure}

\begin{table}[ht]
  \centering
  \caption{\textit{ 3D tests, test 1.  Computational results of the reconstructions
   ${{\varepsilon_w}_h}_k := \max_{\Omega_{\rm FEM}} {{\varepsilon_w}_h}_k$,  $ {{\sigma_w}_h}_k :=  \max_{\Omega_{\rm FEM}} {{\sigma_w}_h}_k$
on a coarse and on adaptively
refined meshes ${K_h}_k, k=0,1,2,3,4$, together with relative errors  for obtained reconstructions.
 The noise level in the data for electric field is $\delta= 10\%$.
Here,
 $\max_{\Omega_{\rm FEM}} {{\varepsilon_w}_h}_k$,  $\max_{\Omega_{\rm FEM}} {{\sigma_w}_h}_k$
 denote maximum of the computed functions ${\varepsilon_w}_h$,  ${\sigma_w}_h$, respectively,
 on $k$ times refined mesh ${K_h}_k$, in the domain  $\Omega_{\rm FEM}$,  and $M_k$ denotes
 the final number of iterations in the ACGA on $k$  times refined
mesh ${K_h}_k$ for reconstructed functions  $  {\varepsilon_w}_h$, ${\sigma_w}_h$, $k = 0,1,2,3,4.$ }}
\begin{tabular}{|l|l|l|l|l|l|l|l|l|} \hline
Mesh & nno & $ \| g_\varepsilon \|$/nno  & $ \| g_\sigma \|$/nno &   $ {{\varepsilon_w}_h}_k $ & $ e_\varepsilon $ & ${{\sigma_w}_h}_k  $ & $  e_\sigma $ &   $M_k$    \\
\hline
${K_h}_0$ & 9261   & 7.8393e-06  & 9.2647e-06  & 6   &  3.5990e-05  &  0.25 & 8.5488e-05  &   2   \\
${K_h}_1$ & 9261   & 7.7529e-06  & 1.1943e-05 & 9.5 &  6.0037e-06  &  0.6  & 5.3990e-05  & 1 \\
${K_h}_2$ & 9261   & 7.7961e-06  & 1.1964e-05  & 10  &  1.1997e-05  &  0.64 & 5.0394e-05 & 1\\
${K_h}_3$ & 9261   & 7.8069e-06  &  1.0107e-05 & 10  &  1.1997e-05   &  0.95 & 2.2492e-05 & 1\\
${K_h}_4$ & 10116  & 6.9790e-06  & 6.3464e-06  & 10  &  1.0983e-05  &  1.0   &  1.6479e-05   & 1\\  
\hline
\end{tabular}
  \label{tab:table1ACGA}
  \end{table}


\begin{table}[ht]
  \centering
  \caption{\textit{ 3D tests, test 2.  Computational results of the reconstructions
  $ {{\varepsilon_w}_h}_k := \max_{\Omega_{\rm FEM}} {{\varepsilon_w}_h}_k$,  $ {{\sigma_w}_h}_k:=  \max_{\Omega_{\rm FEM}} {{\sigma_w}_h}_k$
on a coarse and on adaptively
refined meshes ${K_h}_k, k=0,1,2,3,4$, together with relative errors  for obtained reconstructions.
 The noise level in the data for electric field is $\delta= 10\%$.
Here,
 $\max_{\Omega_{\rm FEM}} {{\varepsilon_w}_h}_k$,  $\max_{\Omega_{\rm FEM}} {{\sigma_w}_h}_k$
 denote maximum of the computed functions ${\varepsilon_w}_h$,  ${\sigma_w}_h$, respectively,
 on $k$ times refined mesh ${K_h}_k$, in the domain  $\Omega_{\rm FEM}$,  and $M_k$ denotes
 the final number of iterations in the ACGA on $k$  times refined
mesh ${K_h}_k$ for reconstructed functions  $  {\varepsilon_w}_h$, ${\sigma_w}_h,$ $k = 0,1,2,3,4.$}}
\begin{tabular}{|l|l|l|l|l|l|l|l|l|} \hline
Mesh & nno &  $\| g_\varepsilon \|$/nno & $ \| g_\sigma \|$/nno &   $  {{\varepsilon_w}_h}_k  $ &  $ e_\varepsilon $ & $ {{\sigma_w}_h}_k  $ & $  e_\sigma $ &   $M_k$    \\
\hline

${K_h}_0$ & 9261   & 7.8393e-06  & 9.2647e-06  & 10   &  3.5990e-05  &  0.6   & 8.5488e-05  &   4   \\
${K_h}_1$ & 11240  & 6.3879e-06  & 9.8399e-06 & 10    &  4.9466e-06  &  0.57   & 4.4484e-05  & 1 \\
${K_h}_2$ & 18642  & 3.8730e-06  & 5.9436e-06  & 10   &  5.9597e-06  &  0.62   & 2.5035e-05  & 1\\
${K_h}_3$ & 50271  &  1.4382e-06  &  1.8619e-06  & 10 &  2.2100e-06   &  0.68   &  4.1435e-06 & 1\\
${K_h}_4$ & 177795 &  3.9709e-07 & 3.6109e-07  & 10   &  6.2488e-07  &  0.77    &  9.3760e-07   & 1\\  

\hline
\end{tabular}
  \label{tab:table2ACGA}
  \end{table}

\begin{figure}[h]
  \begin{center}
    \begin{tabular}{ccc}
      \includegraphics[trim = 3.0cm 0.0cm 1.5cm 4.0cm, scale=0.15, clip=]{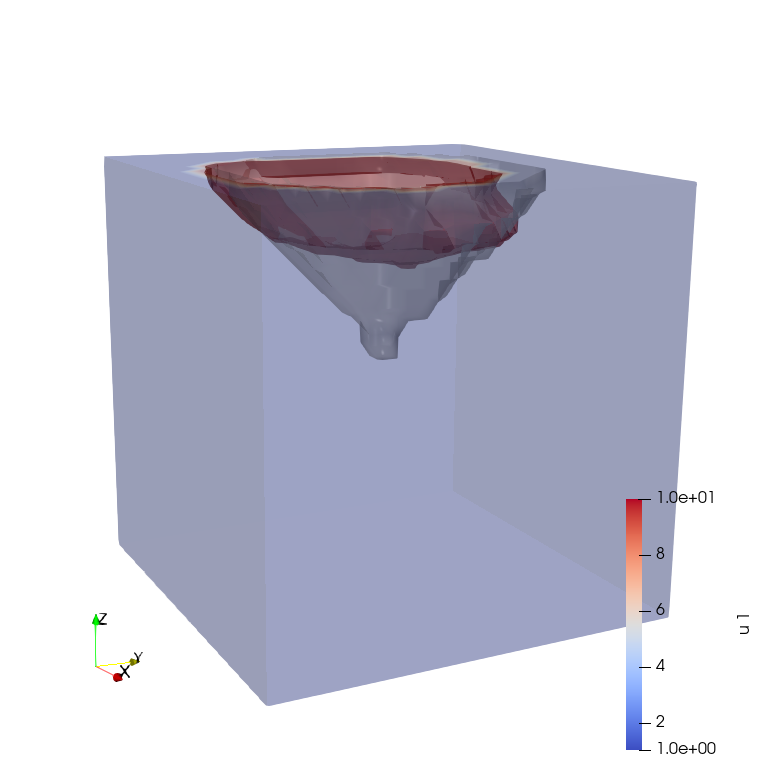} &
 \includegraphics[trim = 3.0cm 0.0cm 1.5cm 4.0cm, scale=0.15, clip=]{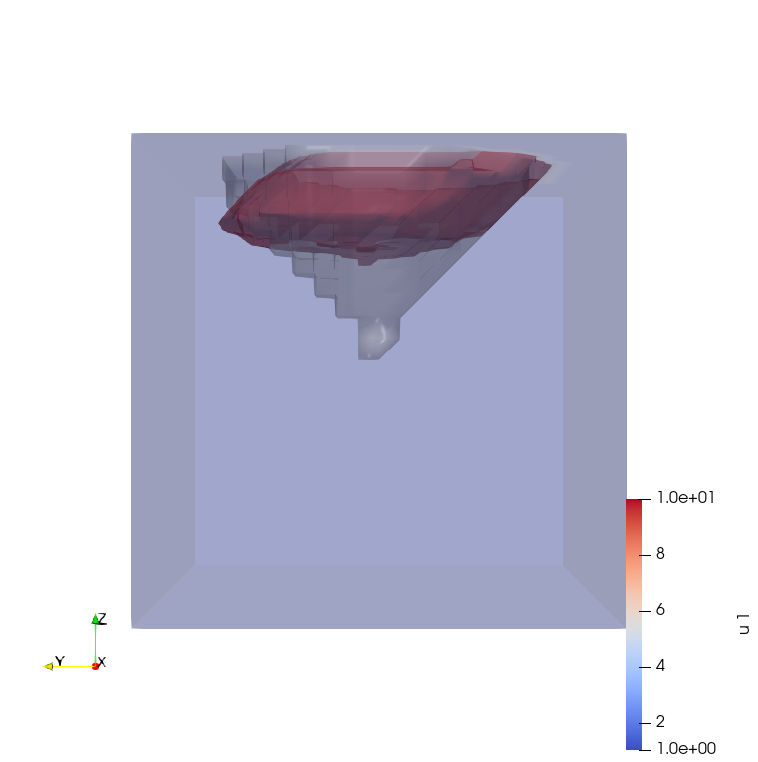} &
  \includegraphics[trim = 3.0cm 0.0cm 1.5cm 4.0cm, scale=0.15, clip=]{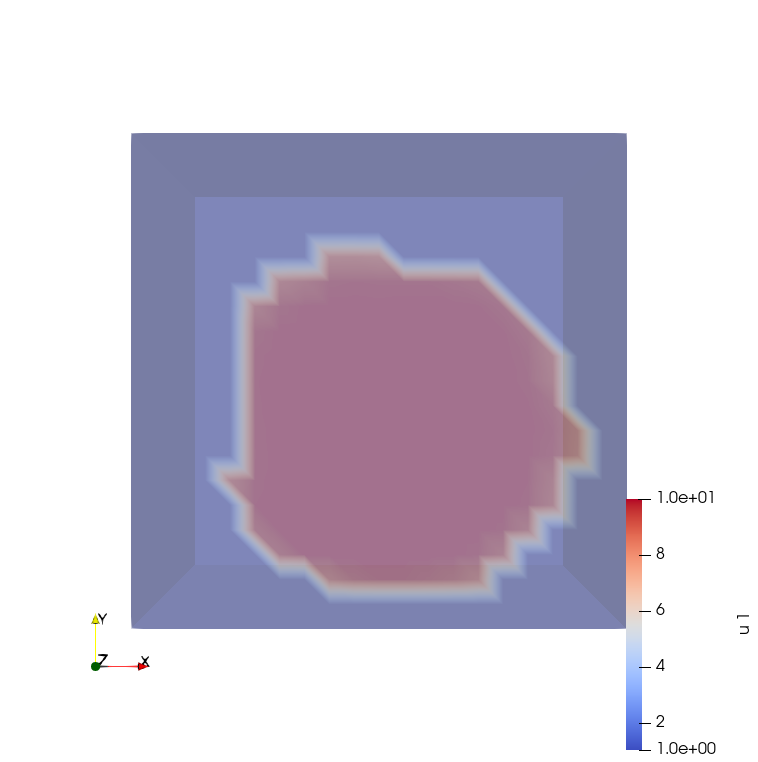} \\
    a)  prospect view   & b) $x_2 x_3$ view  &  c)  $x_1 x_2$  view  \\
      \includegraphics[trim = 3.0cm 0.0cm 1.5cm 4.0cm, scale=0.15, clip=]{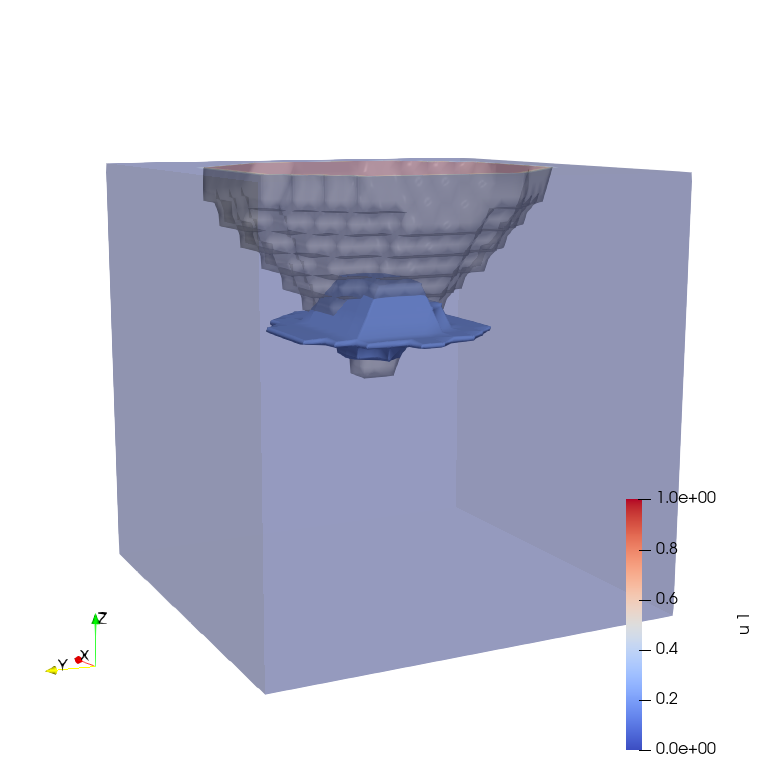}  &
   \includegraphics[trim = 3.0cm 0.0cm 1.5cm 4.0cm, scale=0.15, clip=]{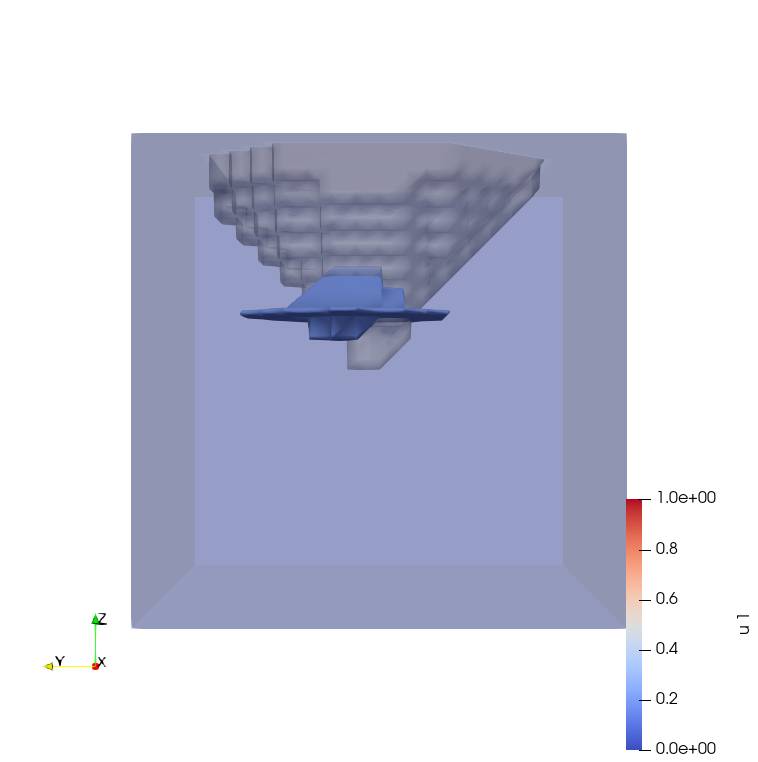}  &     
  \includegraphics[trim = 3.0cm 0.0cm 1.5cm 4.0cm, scale=0.15, clip=]{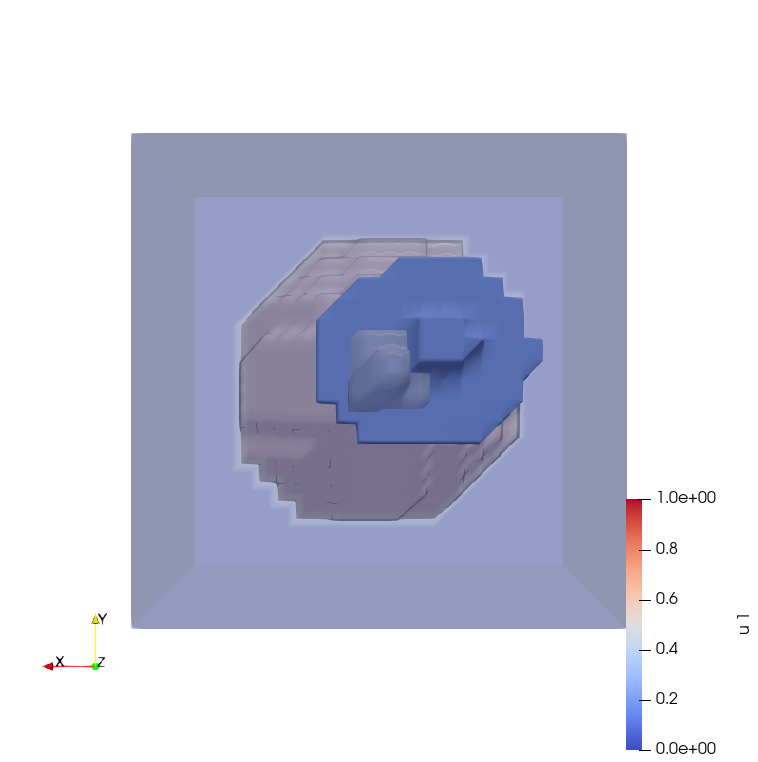}
  \\
      d)  prospect view    & e)   $x_2 x_3$ view  & f)  $x_1 x_2$  view   \\
    \end{tabular}
  \end{center}
  \caption{\small  3D tests, test 2. Performance of CGA on the coarse mesh ${K_h}_0$. Isosurfaces of reconstruction  for 
$\max {\varepsilon_w}_h  \approx 9$     are presented in red color, and
for $\max {\sigma_w}_h \approx 0.6$  are shown in blue color.
a) - c) The weighted reconstruction of ${\varepsilon_w}_h$ (outlined in red color).
  d) - f) The weighted reconstruction of ${\sigma_w}_h$ (outlined in blue color).
    The noise level in the data for electric field is $\delta= 10\%$.  See Table \ref{tab:table2ACGA} for obtained
    contrasts for $\max_{\Omega_{\rm FEM}} {\varepsilon_w}_h$ and $\max_{\Omega_{\rm FEM}} {\sigma_w}_h$.
    For comparison we also present exact isosurfaces  for $\varepsilon_w$ and $\sigma_w$
with values corresponding to the reconstructed ones, outlined in opaque.
  }
  \label{fig:CGA}
\end{figure}


\begin{figure}[h!]
  \begin{center}
    \begin{tabular}{ccc}
      \includegraphics[trim = 3.0cm 0.0cm 1.5cm 4.0cm, scale=0.13, clip=]{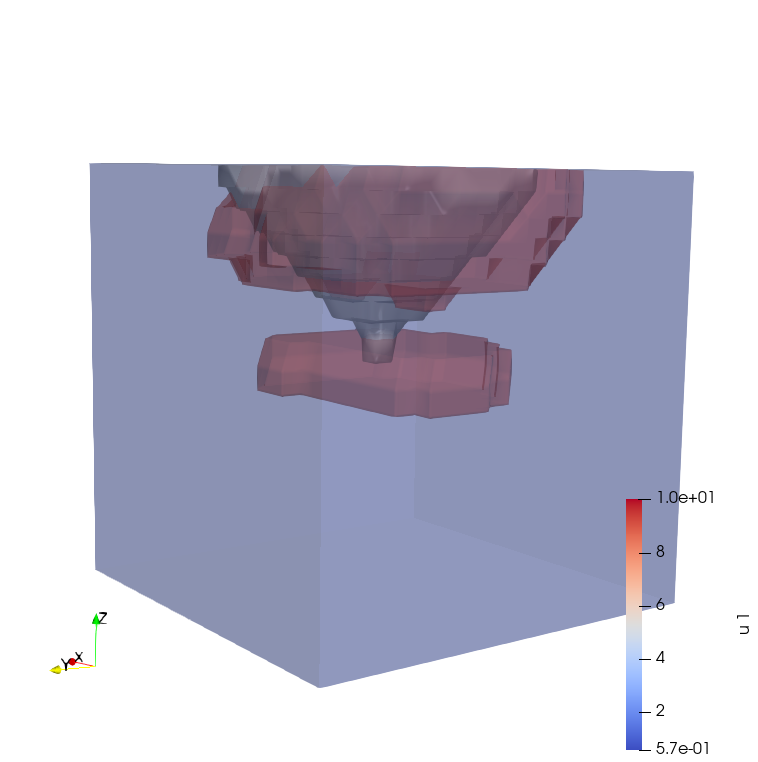} &
 \includegraphics[trim = 3.0cm 0.0cm 1.5cm 4.0cm, scale=0.13, clip=]{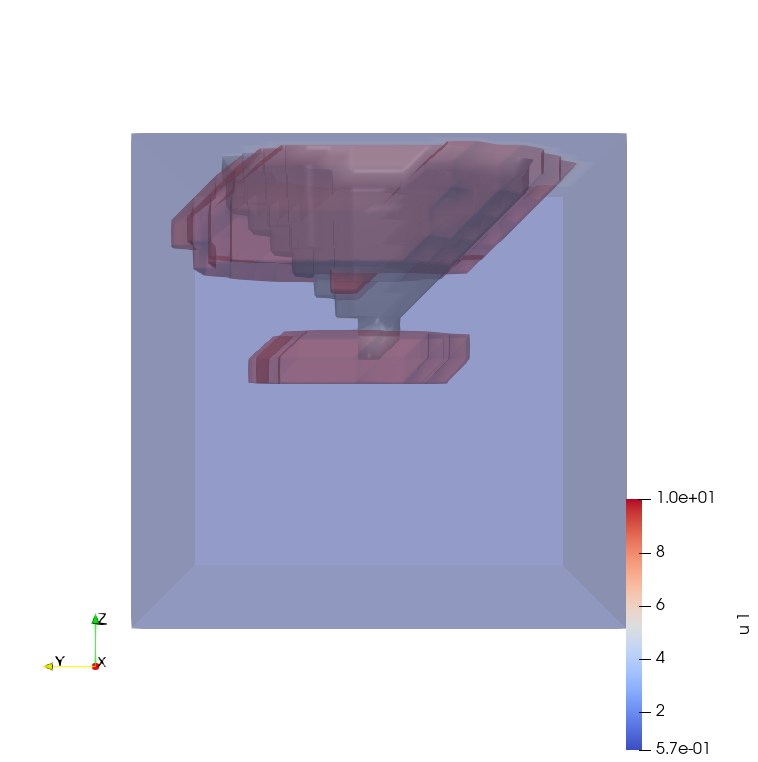} &
  \includegraphics[trim = 3.0cm 0.0cm 1.5cm 4.0cm, scale=0.13, clip=]{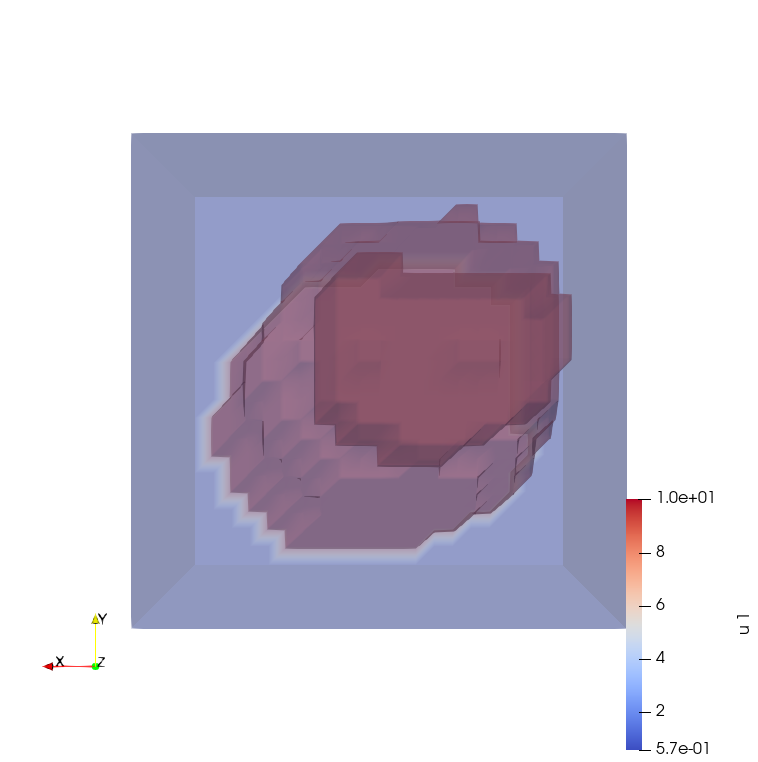} \\
    a)  prospect view   & b) $x_2 x_3$ view  &  c)  $x_1 x_2$  view  \\
      \includegraphics[trim = 3.0cm 0.0cm 1.5cm 4.0cm, scale=0.13, clip=]{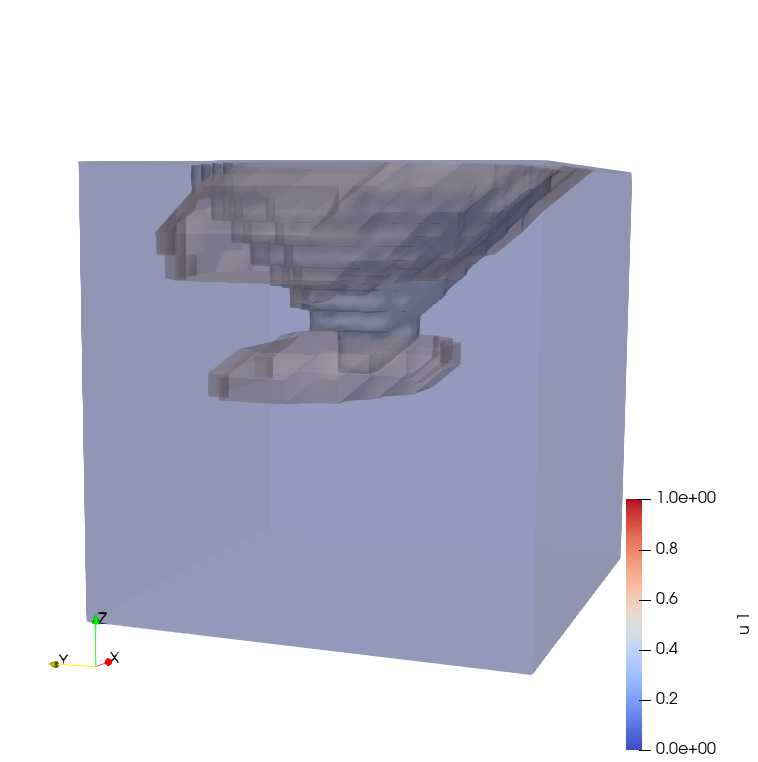}  &
   \includegraphics[trim = 3.0cm 0.0cm 1.5cm 4.0cm, scale=0.13, clip=]{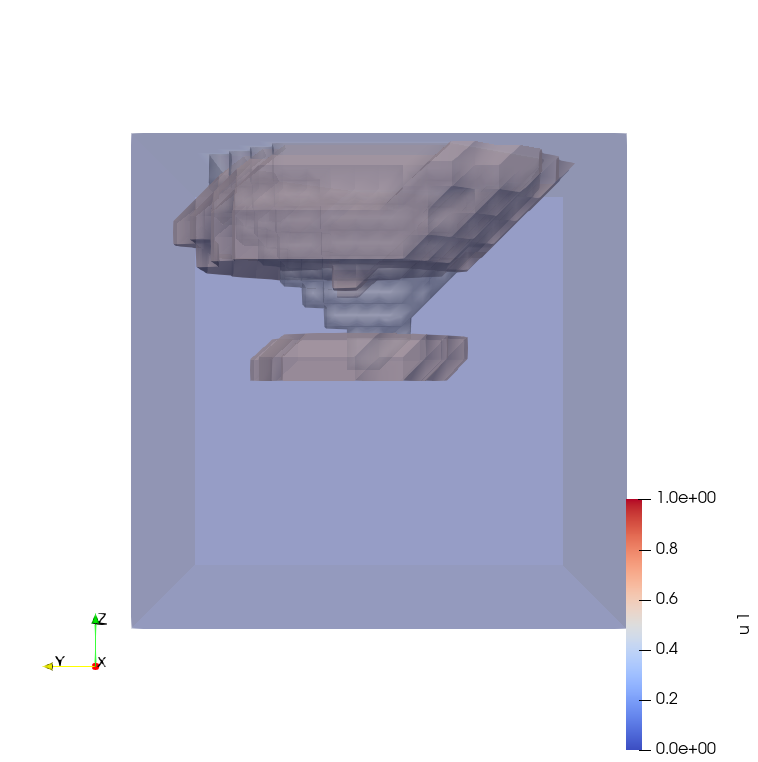}  &     
  \includegraphics[trim = 3.0cm 0.0cm 1.5cm 4.0cm, scale=0.13, clip=]{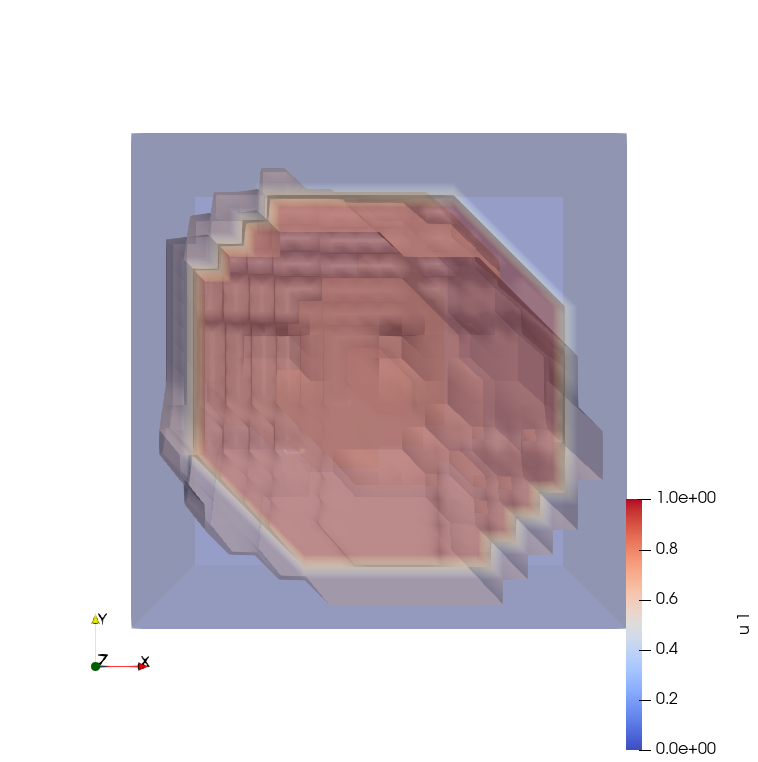}
  \\
      d)  prospect view    & e)   $x_2 x_3$ view  & f)  $x_1 x_2$  view   \\
   \includegraphics[trim = 7.0cm 0.0cm 7cm 4.0cm, scale=0.13, clip=]{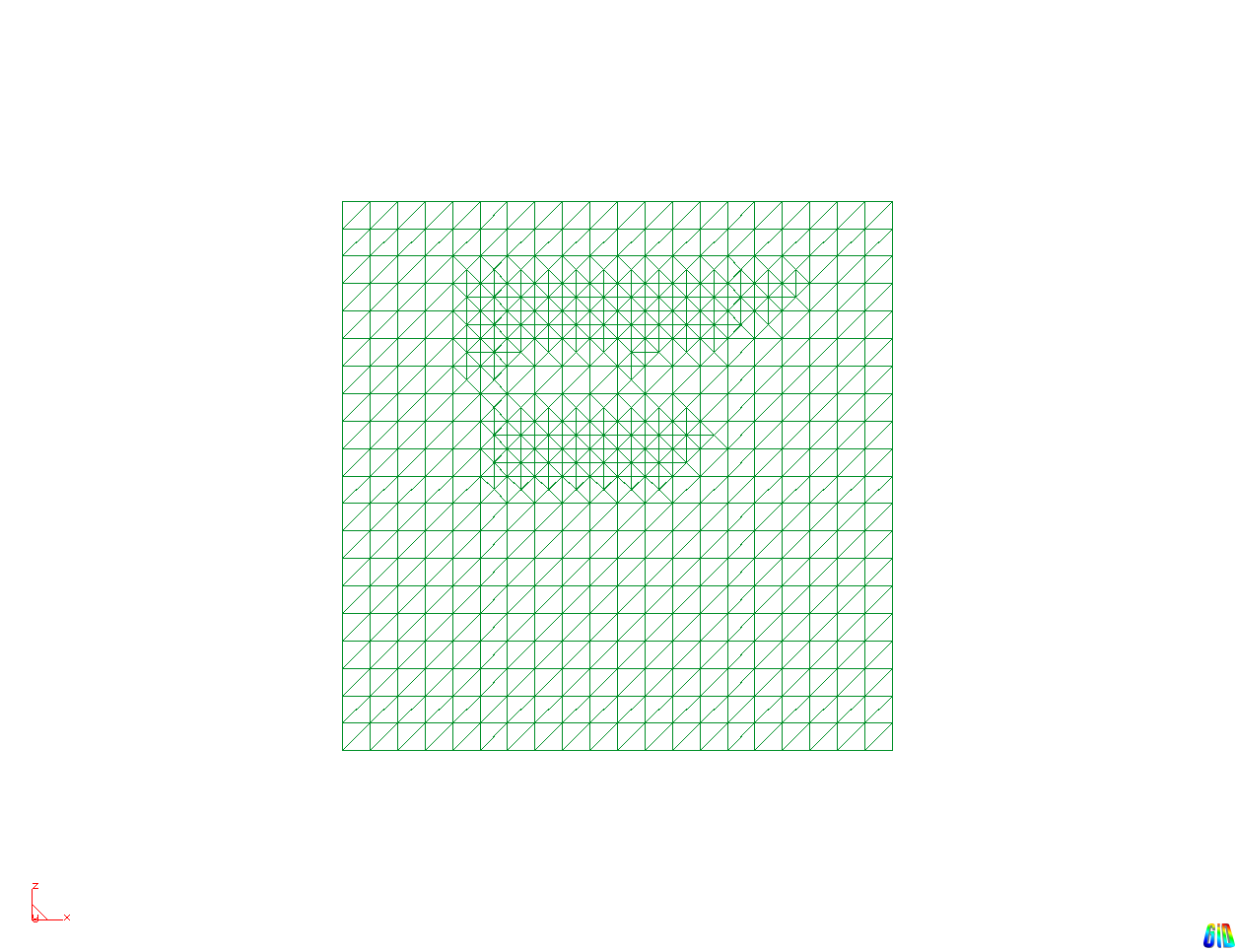}  &    
  \includegraphics[trim = 7.0cm 0.0cm 7cm 4.0cm, scale=0.13, clip=]{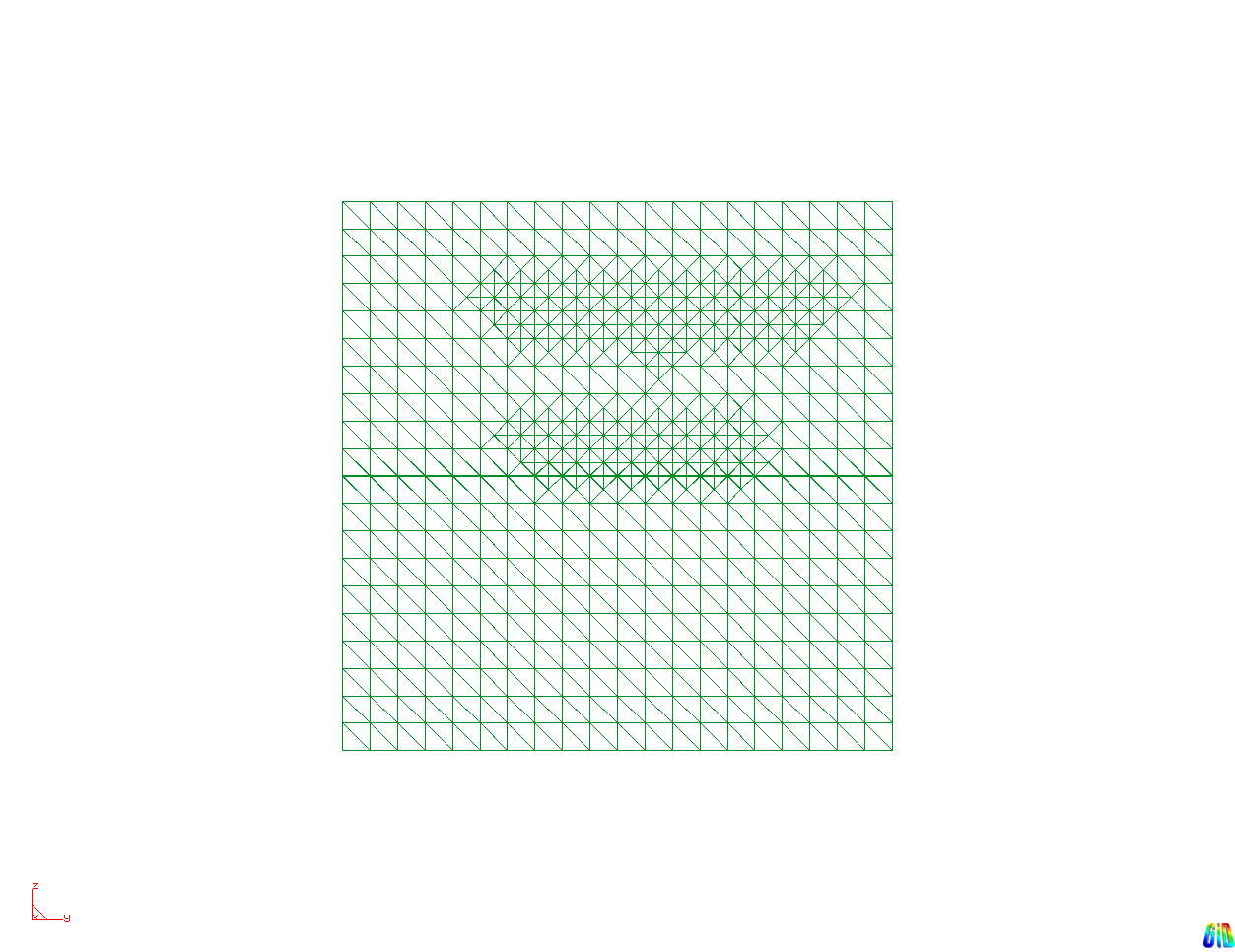}
  &
 \includegraphics[trim = 7.0cm 0.0cm 7cm 4.0cm, scale=0.13, clip=]{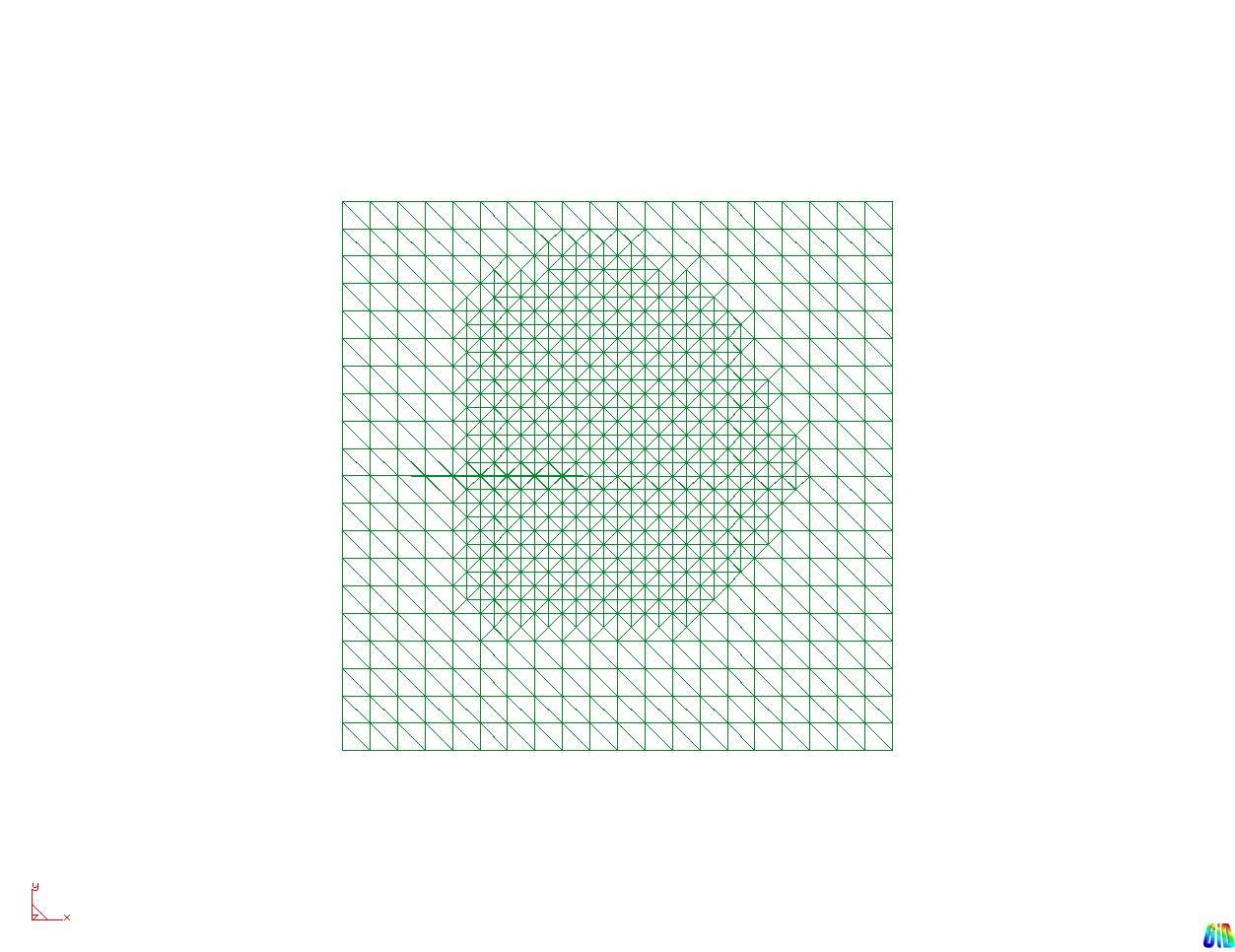}  
\\
    g) $x_1 x_3$  view   & h)  $x_2 x_3$  view   & i)   $x_1 x_2$  view    \\
    \end{tabular}
  \end{center}
  \caption{\small 3D tests, test 2. Performance of ACGA on the one time refined  mesh ${K_h}_1$. Isosurfaces of reconstruction  for 
$\max {\varepsilon_w}_h  \approx 9$     are presented in red color, and   for $\max {\sigma_w}_h \approx 0.57$  are shown in rose color.
a) - c) The weighted reconstruction of ${\varepsilon_w}_h$ (outlined in red color).
  d) - f) The weighted reconstruction of ${\sigma_w}_h$ (outlined in  rose color). g)-i) Projections of the adaptively refined mesh ${K_h}_1$.
    The noise level in the data for electric field is $\delta= 10\%$.  See Table  \ref{tab:table2ACGA} for obtained
    contrasts for $\max_{\Omega_{\rm FEM}} {\varepsilon_w}_h$ and $\max_{\Omega_{\rm FEM}} {\sigma_w}_h $.
    For comparison we also present exact isosurfaces  for $\varepsilon_w$ and $\sigma_w$
with values corresponding to the reconstructed ones, outlined in opaque.
  }
  \label{fig:ACGAref1}
\end{figure}


\begin{figure}[h!]
  \begin{center}
    \begin{tabular}{ccc}
      \includegraphics[trim = 3.0cm 0.0cm 1.5cm 4.0cm, scale=0.13, clip=]{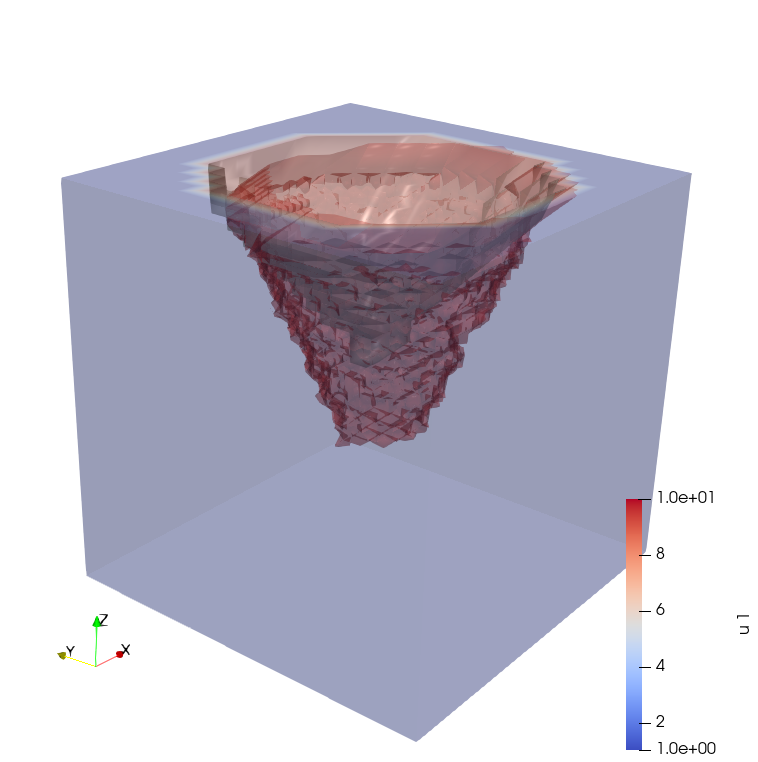} &
 \includegraphics[trim = 3.0cm 0.0cm 1.5cm 4.0cm, scale=0.13, clip=]{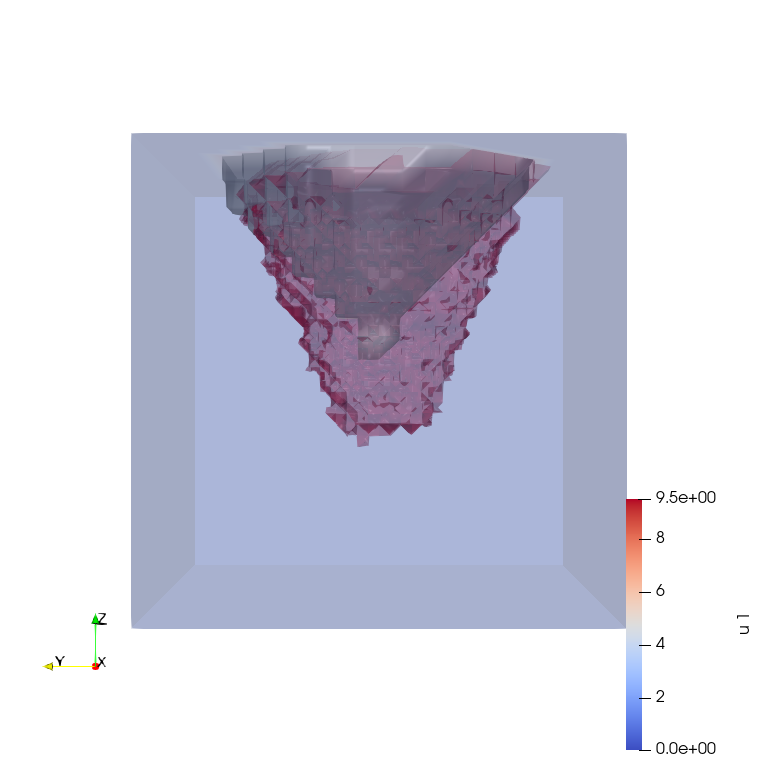} &
  \includegraphics[trim = 3.0cm 0.0cm 1.5cm 4.0cm, scale=0.13, clip=]{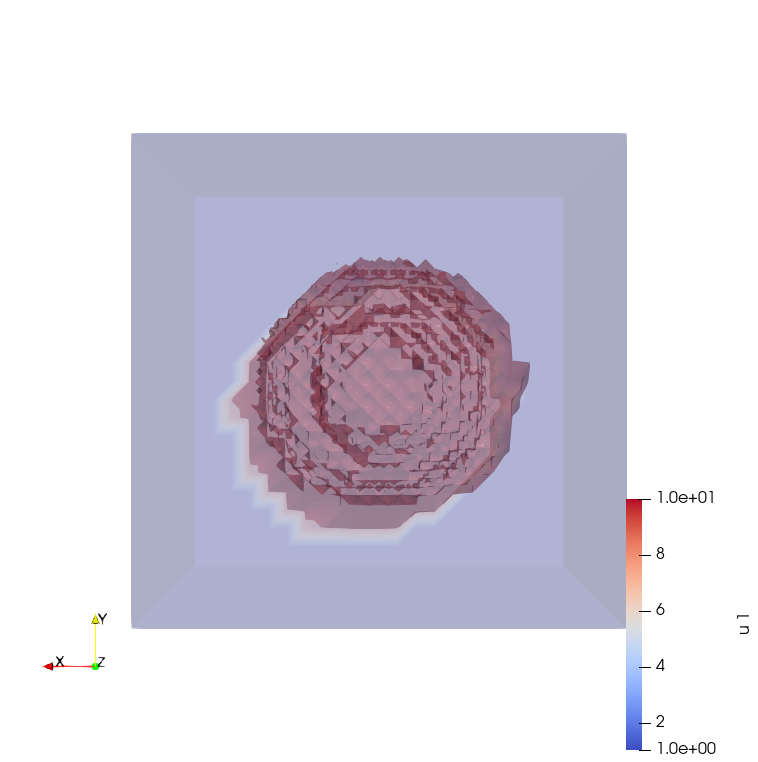} \\
    a)  prospect view   & b) $x_2 x_3$ view  &  c)  $x_1 x_2$  view  \\
      \includegraphics[trim = 3.0cm 0.0cm 1.5cm 4.0cm, scale=0.13, clip=]{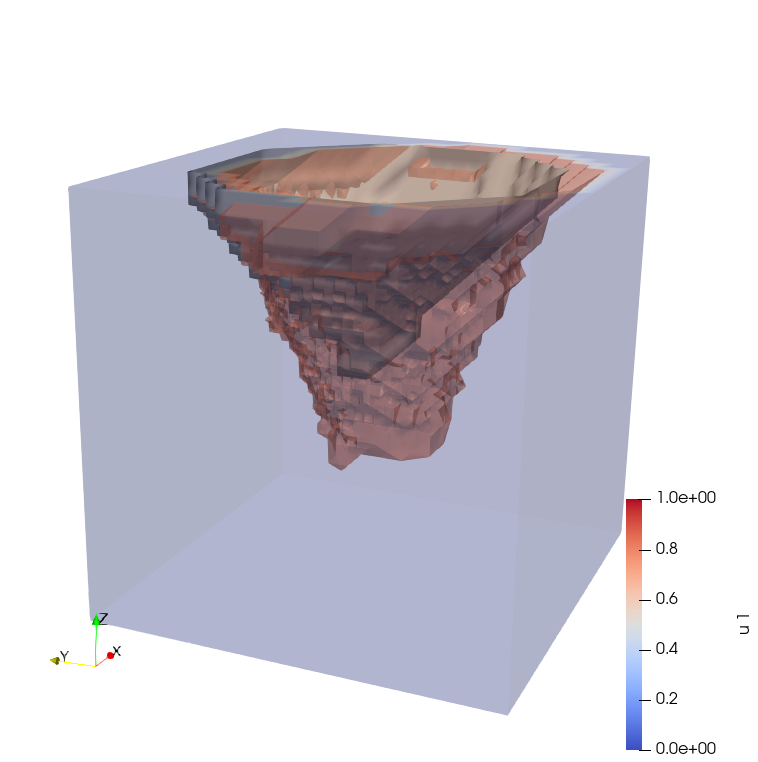}  &
   \includegraphics[trim = 3.0cm 0.0cm 1.5cm 4.0cm, scale=0.13, clip=]{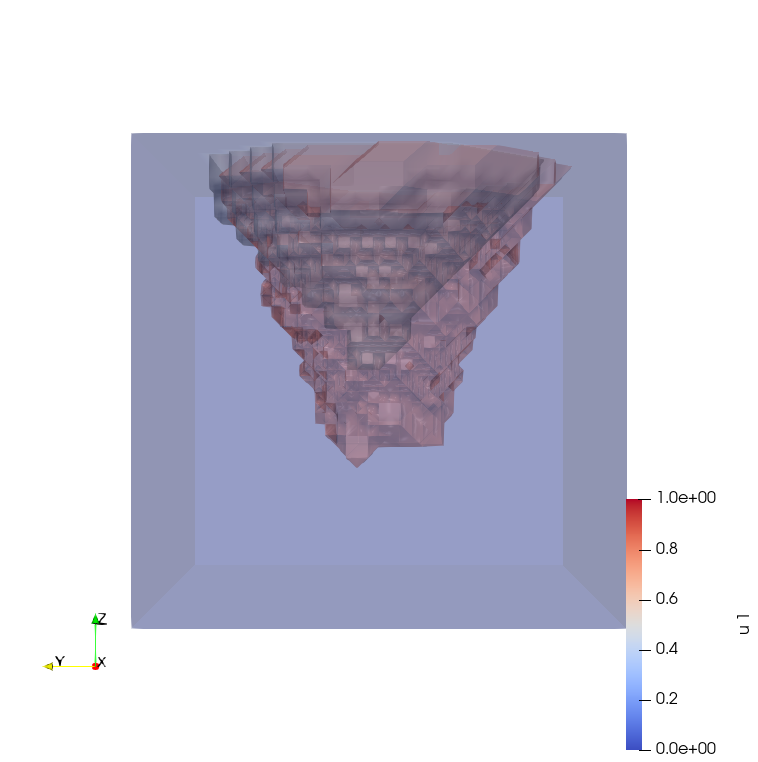}  &     
  \includegraphics[trim = 3.0cm 0.0cm 1.5cm 4.0cm, scale=0.13, clip=]{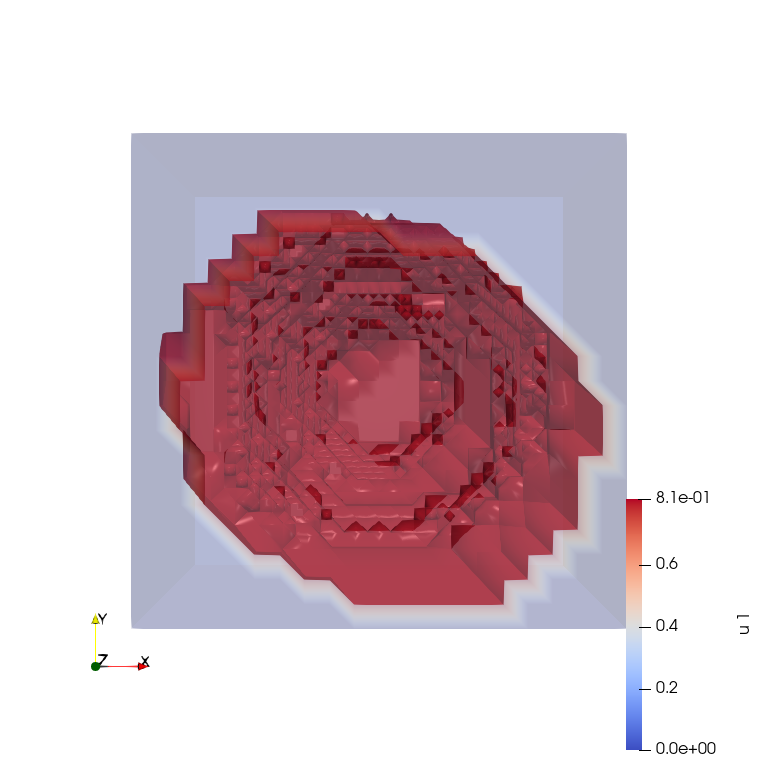}
  \\
      d)  prospect view    & e)   $x_2 x_3$ view  & f)  $x_1 x_2$  view   \\
   \includegraphics[trim = 7.0cm 0.0cm 7cm 4.0cm, scale=0.13, clip=]{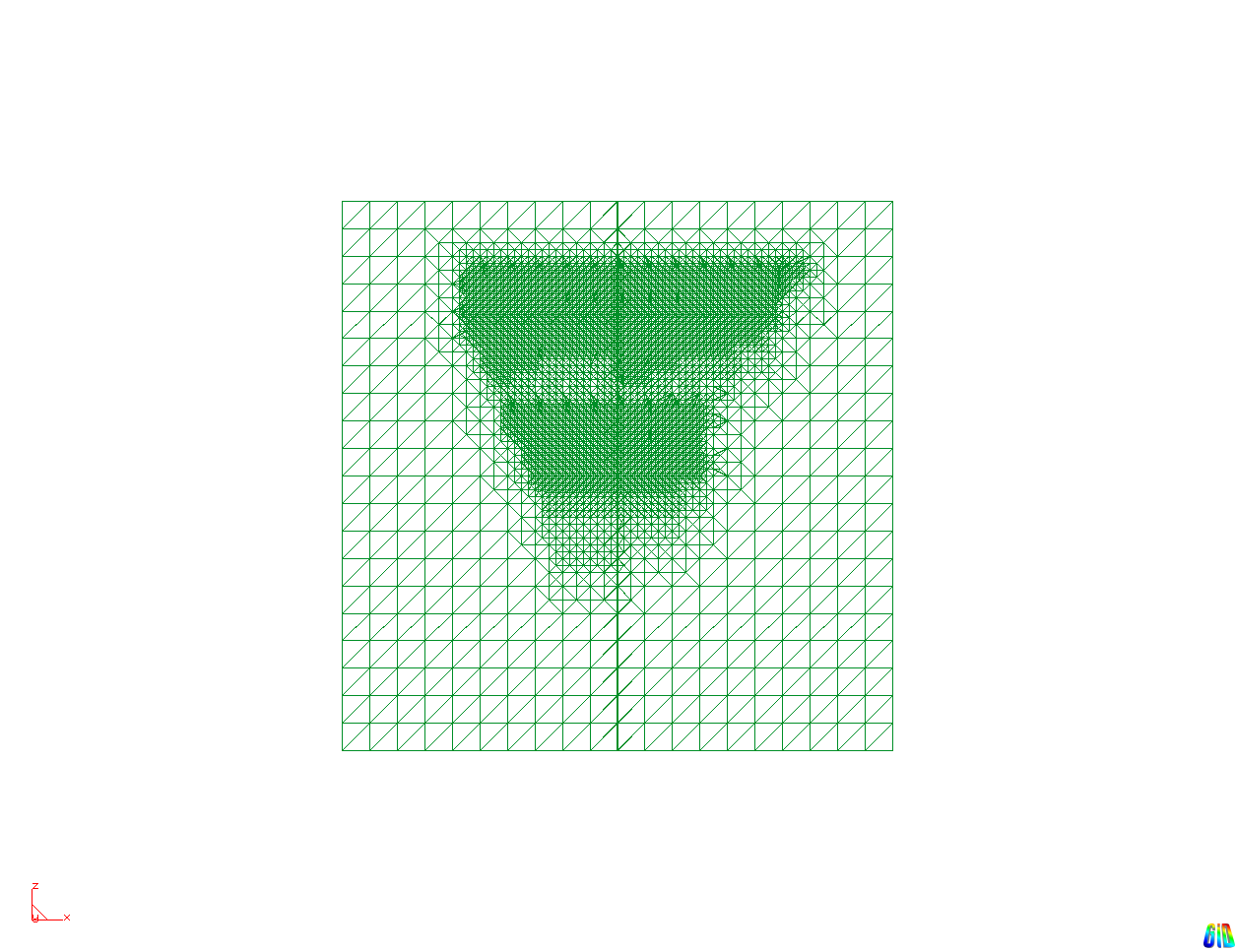}  &    
  \includegraphics[trim = 7.0cm 0.0cm 7cm 4.0cm, scale=0.13, clip=]{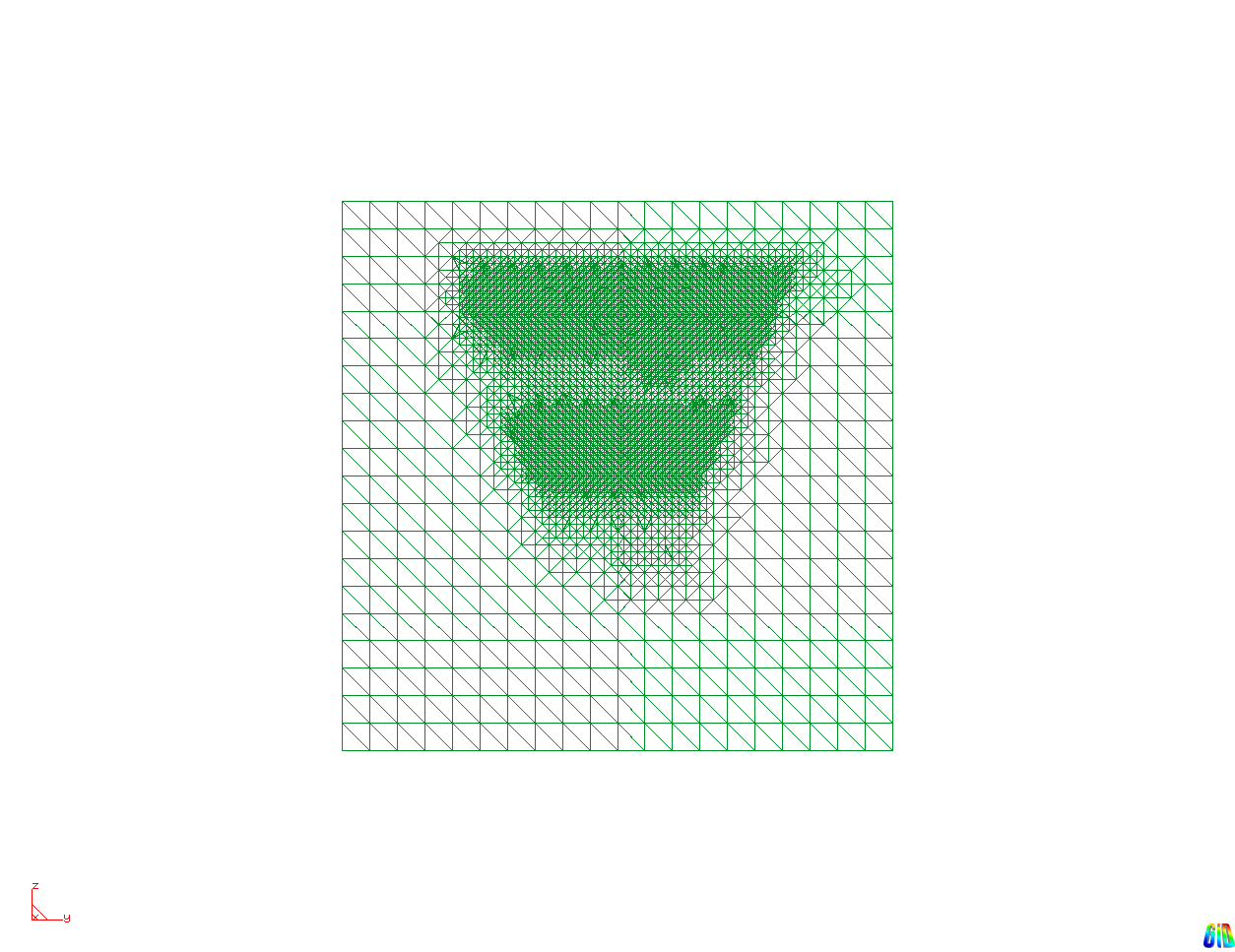}
  &
  \includegraphics[trim = 7.0cm 0.0cm 7cm 4.0cm, scale=0.13, clip=]{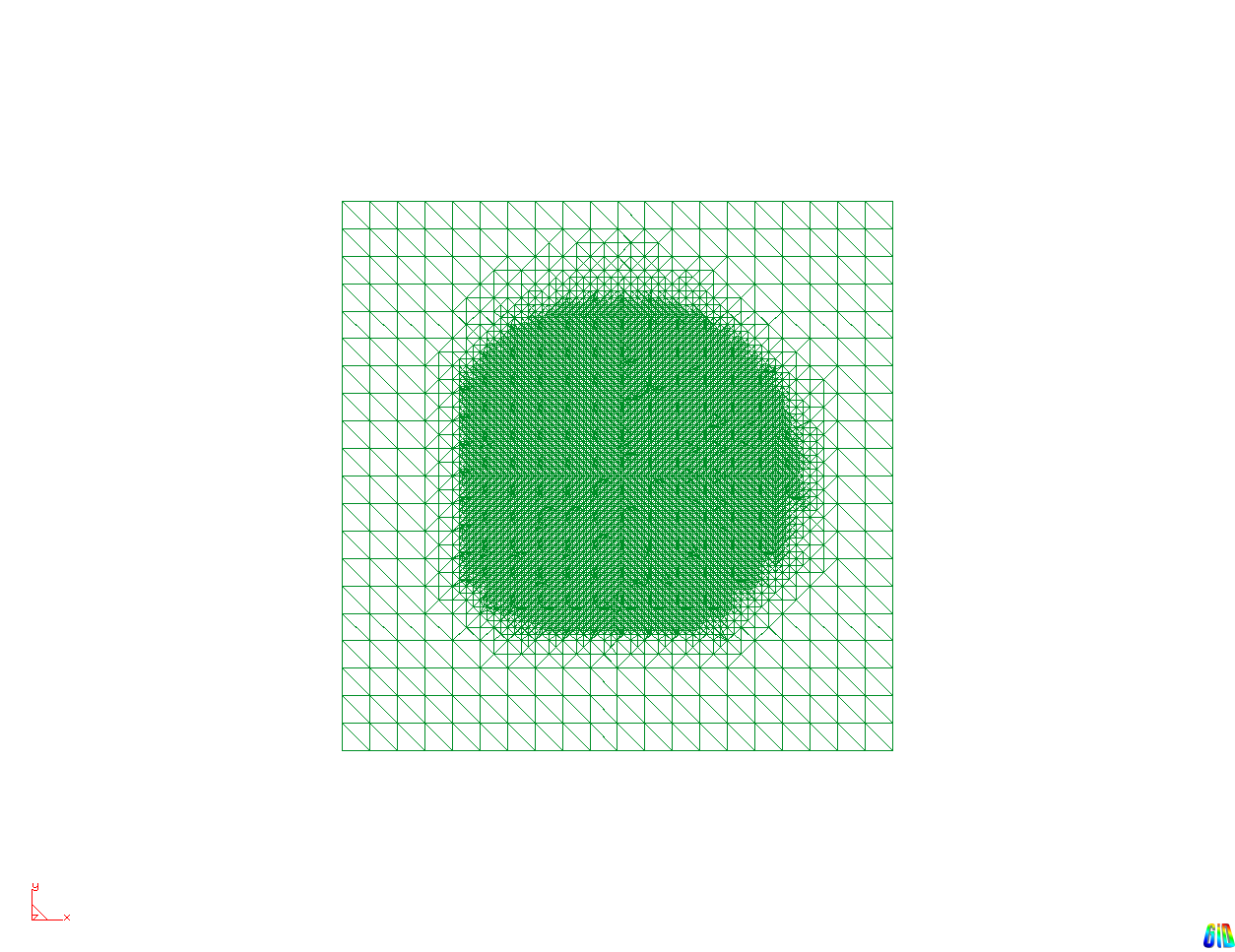}  
\\
    g) $x_1 x_3$  view   & h)  $x_2 x_3$  view   & i)   $x_1 x_2$  view    \\
    \end{tabular}
  \end{center}
  \caption{\small 3D tests, test 2. Performance of ACGA on the four times refined  mesh ${K_h}_4$. Isosurfaces of reconstruction  for 
$\max {\varepsilon_w}_h  \approx 9$     are presented in red color, and
for $\max {\sigma_w}_h \approx 0.77$  are shown also in red color.
a) - c) The weighted reconstruction of ${\varepsilon_w}_h$ (outlined in red color).
  d) - f) The weighted reconstruction of ${\sigma_w}_h$ (outlined in red color). g)-i) Projections of the adaptively refined mesh ${K_h}_4$.
    The noise level in the data for electric field is $\delta= 10\%$.  See Table  \ref{tab:table2ACGA} for obtained
    contrasts for $\max_{\Omega_{\rm FEM}} {\varepsilon_w}_h$ and $\max_{\Omega_{\rm FEM}} {\sigma_w}_h$.
    For comparison we also present exact isosurfaces  for $\varepsilon_w$ and $\sigma_w$
with values corresponding to the reconstructed ones, outlined in opaque.
  }
  \label{fig:ACGAref4}
\end{figure}


We take as an initial
 guess $\varepsilon^0 = 1, \sigma^0 = 0$ for all points of the
 computational domain $\Omega$ and run ACGA algorithm.
This choice has shown good reconstruction results and has been used
in several previous computational studies - see details in \cite{BKS, BL1, ICEAA2024_LB, BOOK}.
Physically it aligns with the
initial computations taken in the homogeneous  non-conductive  domain.

\textbf{Test 1}

Figures \ref{fig:ACGATest1ref4} show reconstruction of the weighted
 values of the relative dielectric permittivity and conductivity
 functions obtained by ACGA algorithm.  Using this figure we observe
 that reconstructions of both functions converges to the depth of the
 MM as mesh refinements are applied.  We also note that
 in this test the MM is located on the top of the backscattered domain
 and the depth of penetration inside the computational domain is only
 1.08 mm.  As a result, the maximal contrast for $\varepsilon_w$ is already
  achieved
  on the coarse mesh
and further mesh refinements do not lead to a significant improvement in the contrast for  $\varepsilon_w$.
However, the shape and the
 contrast for $\sigma_w$ is improved significantly with mesh
 refinements. We refer to the Table  \ref{tab:table1ACGA} for obtained
  computational relative errors $e_{\varepsilon}, e_{\sigma}$
 and 
contrasts for $\max_{\Omega_{\rm FEM}} {\varepsilon_w}_h$ and $\max_{\Omega_{\rm FEM}} {\sigma_w}_h$.

Since  we selected
 numbers $ \widetilde{\beta}_\varepsilon = 0.8 , \widetilde{\beta}_\sigma  = 0.8$
  in the mesh refinement criterion  in ACGA,
 the corresponding reconstructed values trigger mesh refinement only for
$k=4$.  In other words, these parameters
    are too large for the Test 1, but we wanted to test the same parameters for both tests
 in order to ensure a more consistent and reliable comparison.
   We note that the same  values
   $ \widetilde{\beta}_\varepsilon = 0.8 , \widetilde{\beta}_\sigma  = 0.8$  provide very  effective local mesh refinements  for
   all meshes in  Test 2.

\newpage

\textbf{Test 2}

Figures \ref{fig:CGA}, \ref{fig:ACGAref1}, and \ref{fig:ACGAref4} show
 reconstruction of the weighted values of the relative dielectric
 permittivity and conductivity functions obtained by CGA and ACGA
 algorithms, respectively.  Comparing results of the reconstruction
 shown on Figure \ref{fig:CGA} with reconstructions of Figure
 \ref{fig:ACGAref1} and \ref{fig:ACGAref4} we can conclude that the
 local adaptive mesh refinement, or ACGA algorithm, significantly
 improves results of 3D reconstruction of $\varepsilon_w$ and
 $\sigma_w$. Using Figure \ref{fig:ACGAref4} we observe that
 reconstructions of both functions converges to the depth of the MM as
 mesh refinements are applied.  The final reconstructions are obtained
 on the 4 times locally refined mesh ${K_h}_4$ and are shown in the
 Figure \ref{fig:ACGAref4}.  From this figure, we also observe that
 the conical shape of the obtained reconstructions accurately
 represents the shape of the MM at a depth of 4.06 mm, consistent with
 the mesh refinements (see Figures \ref{fig:ACGAref4}-g), h), i) )
 thought the depth of the reconstructed MM is slightly overestimated.
 We refer to the Table  \ref{tab:table2ACGA} for obtained
  computational relative errors $e_{\varepsilon}, e_{\sigma}$
 and 
contrasts for $\max_{\Omega_{\rm FEM}} {\varepsilon_w}_h$ and $\max_{\Omega_{\rm FEM}} {\sigma_w}_h$
 on locally adaptively refined meshes.
We hope that performing the
 initial computations on a finer mesh with a size of $h=0.25 $
 instead of $h=0.5$  will provide a more accurate depth for the MM
 in Test 2. This hypothesis  requires more computational resources
  and can be tested in future.

\FloatBarrier

\section{Conclusions}

\label{sec:concl}

In the current work we developed variational optimization approach for
  reconstruction of dielectric permittivity and conductivity functions
  using time-dependent scattered measurements of the electric field
  collected at the part of the boundary of the computational
  domain. This is typical CIP which is an ill-posed problem.
We solve CIP by minimizing the regularised Tikhonov's
 functional via Lagrangian approach.  We present Lagrangian, derive
 optimality conditions and establish proofs for stability estimates
 for the forward and adjoint problems. Further, we   present proof of the Fr\'echet
 differentiability of the  regularized Tikhonov's functional and existence resut of the
 solution of CIP.
Our  two- and three-dimensional 
computational tests show qualitative and quantitative
reconstruction of dielectric permittivity and conductivity functions in several examples.

Future research will focus on reconstructing unweighted dielectric
permittivity and conductivity functions for the detection of malignant
melanoma on the skin, using both simulated and experimentally acquired
data. This effort will require the development of a new adaptive
DDFE/FDM method, specifically designed for finite difference domains,
in which a damped wave equation replaces the acoustic model.

\section*{Acknowledgment}

The research of authors  is supported by the Swedish Research Council grant VR 2024-04459
 and STINT grant MG2023-9300.


\medskip
\medskip


\end{document}